\documentclass[12pt, reqno]{amsart}
\usepackage{float}
\usepackage{amssymb, amsmath}
\usepackage{bbm, esint, times, url}
\usepackage{tikz}
\usetikzlibrary{calc, arrows.meta, patterns, decorations.text}
\tikzset{>=Stealth}

\usepackage[numbers, sort]{natbib}
\usepackage[colorlinks, pagebackref = true]{hyperref}
\usepackage{breakurl}
\usepackage{graphicx}
\usepackage{enumitem}
\usepackage{aliascnt}
\usepackage{listings}
\usepackage[margin = 1in]{geometry}

\newtheorem*{rmks}{Remarks}

\newtheorem{thm}{Theorem}[section]

\newaliascnt{lma}{thm}
\newtheorem{lma}[lma]{Lemma}
\aliascntresetthe{lma}

\newaliascnt{prop}{thm}

\aliascntresetthe{prop}

\newaliascnt{cor}{thm}
\newtheorem{cor}[cor]{Corollary}
\aliascntresetthe{cor}

\newaliascnt{conj}{thm}
\newtheorem{conj}[conj]{Conjecture}
\aliascntresetthe{conj}

\newaliascnt{rmk}{thm}
\newtheorem{rmk}[rmk]{Remark}
\aliascntresetthe{rmk}

\newaliascnt{defn}{thm}

\aliascntresetthe{defn}

\newcommand{\R}{\mathbb{R}}

\numberwithin{equation}{section}

\begin{document}

\title[Monotonicity of positive solutions in triangles]{
Monotonicity of positive solutions to semilinear elliptic equations with mixed boundary conditions in triangles
}
\author[R. Li]{Rui Li}
\address{College of Big Data and Internet, Shenzhen Technology University, Shenzhen 518118, P. R. China}
\email{lirui@sztu.edu.cn}

\author[R. Yao]{Ruofei Yao}
\thanks{Ruofei Yao is the corresponding author}
\address{School of Mathematics, South China University of Technology, Guangzhou 510641, P. R. China}
\email{yaorf5812@126.com; ruofeiyaopde@gmail.com}

\begin{abstract}
We study positive solutions of semilinear elliptic equations in a planar triangular domain under mixed boundary conditions, consisting of homogeneous Dirichlet boundary conditions on one side and homogeneous Neumann boundary conditions on the remaining two sides. Using the method of moving planes, we prove that if the Neumann vertex is non-obtuse, then every positive solution is strictly increasing in the direction of the inward unit normal to the Dirichlet side. If the Neumann vertex is obtuse, we show that monotonicity instead holds in the direction of the outward normal direction to the longer Neumann side, under certain technical conditions. Furthermore, by applying the maximum principle, we demonstrate that these monotonicity properties extend to the first mixed eigenfunction of the Laplacian: its unique global maximum lies on the longer Neumann side and coincides with the Neumann vertex precisely when that vertex is non-obtuse or when the Neumann sides have equal length. This answers a question raised in the Polymath7 research thread 1 regarding the location of extrema for mixed boundary eigenfunctions in triangles.
\end{abstract}


\date{\today}
\maketitle



\section{Introduction}

This paper investigates the monotonicity properties of positive solutions to semilinear elliptic equations with mixed boundary conditions:
\begin{equation} \label{LY101}
\begin{cases}
\Delta {u} + {f}({u}) = 0 & \text{in } \Omega, \\
{u} = 0 & \text{on } \Gamma_{D}, \\
\frac{\partial {u}}{\partial{\nu}} = 0 & \text{on } \Gamma_{N} = \partial\Omega \setminus \Gamma_{D},
\end{cases}
\end{equation}
where \({f}\) is locally Lipschitz continuous on \(\R\) and \(\Omega\) is a bounded domain in \(\R^{n}\) with \(n \geq 2\). The boundary \(\partial\Omega\) is partitioned into two disjoint parts: \(\Gamma_{D}\), a relatively closed subset of \(\partial\Omega\) on which Dirichlet boundary conditions are imposed, and \(\Gamma_{N} = \partial\Omega \setminus \Gamma_{D}\), where Neumann boundary conditions hold. Here, \(\nu\) denotes the unit outward normal vector along \(\partial\Omega\).

In the study of differential equations, a fundamental question is whether positive solutions have symmetry or perhaps monotonicity in some direction. Investigating such qualitative properties plays a crucial role in various aspects of analysis, such as blow-up analysis, a priori estimates, and Liouville-type theorems. One of the most powerful and widely used techniques for studying these properties is the \emph{method of moving planes}. This method was originally introduced by Alexandrov \cite{Ale56} in his seminal work on surfaces with constant mean curvature, and was subsequently adapted by Serrin \cite{Ser71} to establish symmetry results for overdetermined boundary value problems. Substantial progress was later made by Gidas, Ni, and Nirenberg \cite{GNN79, GNN81}, who refined the method to obtain celebrated symmetry and monotonicity results for positive solutions in bounded and unbounded domains, particularly when such solutions vanish on the boundary or decay at infinity. Building on these foundational contributions, Berestycki and Nirenberg further developed and extended the method in a series of influential works \cite{BN88, BN90, BN91}. In particular, in \cite{BN91}, they not only generalized the moving plane method but also introduced the \emph{sliding method}, which leverages a sharp version of the maximum principle in domains of small measure to derive symmetry and monotonicity results. This line of research was subsequently advanced by Berestycki, Caffarelli, and Nirenberg, who carried out a series of groundbreaking studies on the qualitative behavior of positive solutions \cite{BCN93, BCN96, BCN97a, BCN97b}. Since then, the symmetry and monotonicity of solutions have remained a central theme in the field, continuing to attract substantial attention in the literature (see, e.g., \cite{CGS89, CL91, LZ95, FSV08, CEG16, GM18, Far20, DF22} and the references therein).

Qualitative properties for mixed boundary value problems of the type \eqref{LY101} have been extensively investigated. In cylindrical domains, Berestycki and Nirenberg \cite{BN88, BN90} established symmetry and monotonicity results, while the second author \cite{CWY23} subsequently demonstrated one-dimensional symmetry in half-cylindrical domains. Berestycki and Pacella \cite{BP89} and Zhu \cite{Zhu01} established radial symmetry results for domains that are spherical sectors, while further radial symmetry results were obtained in \cite{dCP89, CW92, SGC97} for infinite sectorial cones. More recently, further symmetry and monotonicity properties for mixed boundary problems in various classes of bounded domains have been obtained by the second author \cite{CY18, YCL18, CLY21, YCG21, MY26, Yao26}. In the context of nonlinear mixed boundary problems, several investigations \cite{Ter95, DR04, WH07, DP19} have likewise addressed symmetry properties.

The linear case of \eqref{LY101} presents equally rich structure. In 2012, the Polymath7 project “The Hot Spots Conjecture for Acute Triangles” \cite{Pol12} initiated extensive discussions on the location of extrema of Laplace eigenfunctions in triangles. Subsequent progress and related developments are documented in \cite{JM22ar}.
In addition, the second author provided a comprehensive characterization of the second Neumann eigenfunction in \cite{CGY26}. Furthermore, in 2004 Ba\~{n}uelos, Pang, and Pascu \cite{BPP04} introduced the “hot spots” property for the mixed (Dirichlet--Neumann) eigenvalue problem, posing the question of whether the positive eigenfunction attains its maximum exclusively on the Neumann portion of the boundary. Within the framework of \cite{Pol12}, a specific question concerning linear mixed problems in triangles is raised (see comment 4 in the \href{https://polymathprojects.org/2012/06/12/polymath7-research-thread-1-the-hot-spots-conjecture/}{\color{black}Polymath7 research thread 1}).

\begin{conj}[\cite{Pol12}] \label{conj101}
The maximum of the first (positive) non-trivial eigenfunction for a triangle with mixed boundary conditions (two sides Neumann, and one side Dirichlet) occurs at the corner opposite to the Dirichlet side.
\end{conj}

In this paper, we investigate the monotonicity properties of the positive solution \({u}\) to \eqref{LY101} in a planar triangle. Throughout this paper, the nonlinearity \(f\) is assumed to be locally Lipschitz continuous on \(\R\); that is, \({f} \in \operatorname{Lip}_{\mathrm{loc}}(\R)\), and \(\Omega\) is an open triangle, with \(\Gamma_{D}\) constituting one side of the triangle and \(\Gamma_{N} = \partial\Omega \setminus \Gamma_{D}\). See \autoref{fig31MMP} for an illustration. In what follows, we refer to the vertex opposite \(\Gamma_{D}\) as the \textbf{Neumann vertex}, and to the other two vertices as the \textbf{mixed vertices}.

Our first main result concerns the symmetry and monotonicity of solutions in isosceles triangles.

\begin{thm} \label{thm102}
Let \({f} \in \operatorname{Lip}_{\mathrm{loc}}(\R)\), and let \({u}\) be a positive solution of \eqref{LY101} in a planar triangle \(\Omega\), where \(\Gamma_{D}\) is one side of the triangle \(\Omega\).
If the two Neumann sides have equal lengths, then \({u}\) is symmetric with respect to the angle bisector at the Neumann vertex and satisfies \(\nabla {u} \cdot \mathbf{n} > 0\) in \(\Omega\), where \(\mathbf{n}\) is the unit inward normal vector to \(\Gamma_{D}\).
\end{thm}

The second main result is stated as follows.

\begin{thm} \label{thm103}
Let \({f} \in \operatorname{Lip}_{\mathrm{loc}}(\R)\), and let \({u}\) be a positive solution of \eqref{LY101} in a planar triangle \(\Omega\), where \(\Gamma_{D}\) is one side of the triangle \(\Omega\).
If the Neumann vertex is \textbf{non-obtuse}, then \(\nabla {u} \cdot \mathbf{n} > 0\) in \(\Omega\), where \(\mathbf{n}\) is the unit inward normal vector to \(\Gamma_{D}\).
\end{thm}

These results were previously announced in \cite{CGY26}. We now offer several remarks:
\begin{enumerate} [label = \rm(\arabic*)]
\item
Both \autoref{thm102} and \autoref{thm103} are known to hold when the Neumann vertex is a right angle. This follows directly from the results of \cite{GNN79, BN91} by reflecting the solution across the Neumann boundaries.
\item
In the linear case, \autoref{thm103} holds when the Dirichlet side meets a Neumann side at a right angle. This result originates from the study of the hot spots conjecture for planar domains with two axes of symmetry (see \cite[Theorem 1.1]{JN00}).
\item
The monotonicity of the first mixed eigenfunction (corresponding to the linear version of \autoref{thm103}) can be leveraged to derive mixed eigenvalue inequalities for triangles. In particular, the work of the second author in \cite{CGY26} confirms \cite[Conjecture 1.2]{Siu16}. These eigenvalue inequalities play a crucial role in establishing the complete form of the hot spots conjecture for triangles, as detailed in \cite{CGY26}.
\end{enumerate}

In the case of an obtuse Neumann vertex, the solution \({u}\) may not be monotone in the direction normal to \(\Gamma_{D}\). Nevertheless, one can deduce monotonicity in another fixed direction.

\begin{thm} \label{thm104}
Let \({f} \in \operatorname{Lip}_{\mathrm{loc}}(\R)\), and let \({u}\) be a positive solution of \eqref{LY101} in a planar triangle \(\Omega\), where \(\Gamma_{D}\) is one side of the triangle \(\Omega\). Suppose that the Neumann vertex is \textbf{obtuse}. If the angles \(\alpha\) and \(\beta\) at the two mixed vertices satisfy
\begin{equation} \label{LY103}
\max\{\alpha, \beta\} \geq \min \left\{ \frac{\pi}{4}, 2\alpha + 2\beta - \frac{\pi}{2} \right\},
\end{equation}
then \(\nabla {u} \cdot \mathbf{n} < 0\) in \(\Omega\), where \(\mathbf{n}\) denotes the unit inward normal vector to the longer Neumann side.
Moreover, the maximum of \({u}\) is attained at the Neumann vertex if and only if the two Neumann sides have equal lengths.
\end{thm}

When condition \eqref{LY103} is not satisfied, the nonlinear obtuse case is not covered by the present argument. Nevertheless, in the linear case, the corresponding monotonicity result holds.

\begin{thm} \label{thm105}
Let \({u} > 0\) be the eigenfunction corresponding to the first eigenvalue of the mixed Dirichlet–Neumann problem in a planar triangle \(\Omega\) with one Dirichlet side \(\Gamma_{D}\) and two Neumann sides \(\Gamma_{N}\), and assume that the Neumann vertex is obtuse. Then \({u}\) satisfies the following properties:
\begin{enumerate} [label = \rm(\arabic*)]
\item
\({u}\) is strictly decreasing in the inward normal direction to the longer Neumann side in \(\Omega\);
\item
\({u}\) has at most one non-vertex critical point, which, if it exists, is non-degenerate and lies in the interior of the longer Neumann side;
\item
the maximum of \({u}\) is attained at a vertex if and only if the triangle \(\Omega\) is isosceles.
\end{enumerate}
\end{thm}

In relation to \autoref{conj101}, \autoref{thm102} and \autoref{thm103} resolve the conjecture in the semilinear case when the Neumann vertex is non-obtuse. Meanwhile, \autoref{thm104} provides a partial answer
in the semilinear obtuse case under the additional condition \eqref{LY103}. In the linear case, \autoref{thm105} yields a complete result for all obtuse triangles. Therefore, the conjecture is completely resolved in the linear case, and in the semilinear case under condition \eqref{LY103}. The nonlinear obtuse case without \eqref{LY103} remains open.

\begin{cor} 
Let \({u} > 0\) be the first eigenfunction for a triangle with one Dirichlet side and two Neumann sides. Then we have
\begin{enumerate} [label = \rm(\arabic*)]
\item
\({u}\) is strictly decreasing in the inward normal direction to the longer Neumann side in \(\Omega\);
\item
the maximum point of \({u}\) is unique and lies on the longer Neumann side;
\item
the maximum is attained at the vertex opposite the Dirichlet side if and only if either the Neumann vertex is non-obtuse or the two Neumann sides have the same length.
\end{enumerate}
\end{cor}

It is noteworthy that for an obtuse Neumann vertex in a non-isosceles triangle the maximum of the eigenfunction is uniquely and exclusively located in the interior of the longer Neumann side. Moreover, several additional properties of the first mixed Laplace eigenfunction are established in \cite{Hat24}.

The proofs of \autoref{thm102}, \autoref{thm103}, and \autoref{thm104} rely on the application of the moving plane method \cite{GNN79} and certain versions of the maximum principle for mixed boundary problems, as presented in \cite{YCL18, YCG21}. The analysis of the local behavior near conical points \cite{Gri11}, as well as the structure of the nodal lines of directional derivatives \cite{HW53, HHT09}, play a pivotal role in the proofs of \autoref{thm105} and \autoref{thm104}.

The structure of the paper is as follows. In \autoref{Sec02}, we present preliminary results concerning the maximum principle and the monotonicity properties near the boundary. \autoref{Sec03acu} introduces the moving lines and moving domains -- key ingredients of our method -- and establishes \autoref{thm103} for acute triangles. In \autoref{Sec04iso}, we focus on isosceles triangles and provide the proof of \autoref{thm102}. \autoref{Sec05obt} is devoted to obtuse triangles, culminating in the proof of \autoref{thm104}. Finally, in \autoref{Sec06EF}, we consider the linear case and prove \autoref{thm105}.


\section{Some preliminaries} \label{Sec02}

In this section, we collect some preliminary results that will be used throughout the paper. These include a variant of the maximum principle for mixed boundary conditions, as well as local monotonicity properties of solutions near Dirichlet and Neumann boundaries.

\begin{lma} \label{lma201}
Let \(\Omega \subset \R^{n}\) be a bounded Lipschitz domain. Its boundary \(\partial\Omega\) is decomposed as the disjoint union \(\partial\Omega = \Gamma_{0} \cup \Gamma_{1}\), where \(\Gamma_{0}\) is a relatively closed subset of \(\partial\Omega\) and \(\Gamma_{1} = \partial\Omega \setminus \Gamma_{0}\). Denote by \(\nu({x})\) the unit outward normal vector on \(\Gamma_{1}\). Suppose that
\begin{equation} \label{LY201}
\begin{cases}
\Delta {v} + {c}( {x} ){v} \leq 0 & \text{in } \Omega,
\\
{v} \geq 0 & \text{on } \Gamma_{0},
\\
\frac{\partial {v}}{\partial \nu} \geq 0 & \text{on } \Gamma_{1}
\end{cases}
\end{equation}
with \(\sup_{{x} \in \Omega}|{c}({x})| < {c}_{0}\) for some \({c}_{0} \in (0, \infty)\). Assume further that the following geometric conditions are satisfied:
\begin{enumerate} [label = \rm(\arabic*)]
\item
\(\Omega \cup \Gamma_{1} \subset\{ {x} \in \R^{n}: \; 0 < x_{1}\}\);
\item
\(\nu \cdot \textit{e}_{1} \geq 0\) on \(\Gamma_{1}\), where \(\textit{e}_{1} = (1, 0, \ldots, 0)\) is the first coordinate unit vector;
\item
\(\Omega \subset \{ {x} \in \R^{n}: \; 0 < {x} \cdot \textit{e}_{1} < \eta\}\), where \(\eta = \pi/(2\sqrt{{c}_{0}})\).
\end{enumerate}
Then, \({v} \geq 0\) in \(\Omega\).
\end{lma}

\begin{proof}
We construct a positive supersolution associated with \eqref{LY201}:
\begin{equation*}
{g}({x}) = \sin \frac{\pi x_{1}}{2\eta}.
\end{equation*}
Note that this function vanishes on \({x}_{1} = 0\), is positive in \({x}_{1} > 0\), and satisfies
\begin{equation*}
\begin{cases}
\Delta {g} + {c}( {x} ){g} < 0 & \text{in } \Omega,
\\
\nu \cdot \nabla {g} \geq 0 & \text{on } \Gamma_{1}.
\end{cases}
\end{equation*}
Applying the maximum principle in \cite[Lemma 6 and Lemma 7]{YCL18}, we obtain \({v} \geq 0\) in \(\Omega\).
\end{proof}

\begin{rmk}
If the double (reflected) domain \(\tilde{\Omega}\), obtained by reflecting \(\Omega\) across the hyperplane \(\{ x_{1} = 0 \}\), is convex, then the condition \(\nu \cdot \textit{e}_{1} \geq 0\) holds automatically on \(\Gamma_{1}\).
\end{rmk}

Next, we recall two useful lemmas from \cite{GNN79} and \cite{BP89} concerning the monotonicity properties of solutions near Dirichlet and Neumann boundaries. Consider a solution \({u}({x})\) of the semilinear elliptic equation
\begin{equation} \label{LY203}
\Delta {u} + {f}({u}) = 0 \text{ in } \Omega,
\end{equation}
where \({f} \in \operatorname{Lip}_{\mathrm{loc}}(\R)\) and \(\Omega\) is a bounded domain.

\begin{lma} \label{lma203GNN}
Let \(\bar{x} \in \partial \Omega\), and let \(\gamma \in \R^{n}\) be a unit vector such that \(\nu(\bar{x}) \cdot \gamma > 0\), where \(\nu(\bar{x})\) denotes the outward unit normal vector at \(\bar{x}\). For some \(\varepsilon > 0\), assume that \({u} \in {C}^{2}(\overline{\Omega}_{\varepsilon})\) is a solution of \eqref{LY203}, and satisfies
\begin{equation*}
{u} \geq 0, \; {u} \not\equiv 0 \text{ in } \Omega_{\varepsilon} \quad \text{and} \quad {u} = 0 \text{ on } \Gamma_{\varepsilon},
\end{equation*}
where
\begin{equation*}
\Omega_{\varepsilon} = \{{x} \in \Omega: \; |{x} - \bar{x}| < \varepsilon\} \quad \text{and} \quad \Gamma_{\varepsilon} = \{{x} \in \partial\Omega: \; |{x} - \bar{x}| < \varepsilon\}.
\end{equation*}
Moreover, assume that the boundary \(\Gamma_{\varepsilon}\) is of class \({C}^{2}\). Then, there exists \(\delta > 0\) such that
\begin{equation*}
\frac{\partial {u}}{\partial\gamma} < 0 \text{ in } \Omega_{\delta} = \{{x} \in \Omega: \; |{x} - \bar{x}| < \delta\}.
\end{equation*}
\end{lma}

\begin{proof}
The proof is given in \cite[Lemma 2.1]{GNN79}.
\end{proof}

The above lemma establishes the strict inward monotonicity of solutions in a neighborhood of the Dirichlet boundary. We next recall a related monotonicity result near Neumann boundaries, which is instrumental in the application of the moving plane method.

\begin{lma} \label{lma205BP}
Let \(\bar{x} = (\bar{x}_{1}, \ldots, \bar{x}_{n}) \in \partial\Omega\). For some \(\varepsilon > 0\), suppose the domain \(\Omega_{\varepsilon, \bar{x}} = \{{x} \in \Omega: |{x} - \bar{x}| < \varepsilon \}\) coincides with the half-ball \({B}_{\varepsilon, \bar{x}}^{ + } = \{{x} \in \R^{n}: |{x} - \bar{x}| <\varepsilon, \; {x}_{1} > \bar{x}_{1} \}\).
Suppose that \({u}\) is a \(C^{2}\) solution of \eqref{LY203} in \(\overline{\Omega_{\epsilon, \bar{x}}}\), satisfying the Neumann boundary condition
\begin{equation*}
\partial_{{x}_{1}} {u} = 0 \text{ on } \{{x} \in \R^{n}: |{x} - \bar{x}| <\varepsilon, \; {x}_{1} = \bar{x}_{1} \},
\end{equation*}
and suppose that the following reflection condition holds:
\begin{equation*}
{u}({x}', x_{n}) < {u}({x}', 2\bar{x}_{n} - x_{n}) \; \text{ for all } {x} \in \Omega_{\varepsilon, \bar{x}} \text{ with } x_{n} > \bar{x}_{n},
\end{equation*}
where \({x}' = ({x}_{1}, \ldots, {x}_{n-1})\).
Then
\begin{equation*}
\partial_{x_{n}}{u}(\bar{x}) < 0.
\end{equation*}
\end{lma}

\begin{proof}
This result follows from Serrin's lemma \cite{Ser71, GNN79}. For further details, see the proof of \cite[Theorem 2.4]{BP89}.
\end{proof}


\section{The proof of monotonicity for non-obtuse Neumann vertex} \label{Sec03acu}

In this section, we establish the monotonicity result under the assumption that the two Neumann sides form an acute or right angle. We begin by introducing the geometric setting and notations concerning moving lines and moving domains, which will be used consistently throughout the paper. We adopt the following notation:
\begin{itemize}
\item
The symbol \(\mathbb{N}\) and \(\mathbb{N}^{ + }\) denote the sets of nonnegative and positive integers, respectively.
\item
The unit vector in the direction of an angle \({\theta}\) is denoted by \(\mathbf{e}_{\theta}: = (\cos\theta, \sin\theta)\).
\item
\(\mathrm{Int}( \cdot )\) denotes the interior of a domain or a curve.
\end{itemize}

For simplicity, we assume that \(\Omega\) is the open triangle with vertices \({z}_{0}\), \({z}_{1}\), and \({z}_{2}\). Denote its interior angles by
\begin{equation}
\angle {z}_{0}{z}_{1}{z}_{2} = \alpha, \quad \angle {z}_{0}{z}_{2}{z}_{1} = \beta, \quad \text{and} \quad \angle {z}_{1}{z}_{0}{z}_{2} = \gamma = \pi - \alpha - \beta.
\end{equation}
We place \({z}_{0}\) at the origin \((0, 0)\), and assume that \({z}_{1}\) and \({z}_{2}\) lie on the vertical line \(\{\, {x}_{1} = 1 \}\), with \({z}_{1}\) below \({z}_{2}\), see \autoref{fig31MMP}. Denote the side \({z}_{1}{z}_{2}\) by the Dirichlet boundary \(\Gamma_{D}\), and let the remaining two sides form the Neumann boundary, \(\Gamma_{N} = \partial\Omega \setminus \Gamma_{D}\). We further split \(\Gamma_{N}\) into its lower and upper Neumann boundaries,
\begin{equation*}
\Gamma_{N}^{ - } = \Gamma_{N} \cap \{{x}_{2} < 0\} \quad \text{and} \quad \Gamma_{N}^{ + } = \Gamma_{N} \cap \{{x}_{2} > 0\}.
\end{equation*}
We emphasize that \(\Gamma_{D}\) is a closed line segment, while both \(\Gamma_{N}^{-}\) and \(\Gamma_{N}^{+}\) are open line segments (excluding their endpoints). We refer to the vertex \({z}_{0}\) as the \textbf{Neumann vertex}, the vertex \({z}_{1}\) as the \textbf{lower mixed vertex}, and the vertex \({z}_{2}\) as the \textbf{upper mixed vertex}. Finally, the lengths of the lower and upper Neumann sides are given by
\begin{equation}
\Phi_{0} = \csc\alpha \quad \text{and} \quad \Psi_{0} = \csc\beta.
\end{equation}

Let \({P}_{\lambda} = (\lambda\sin\alpha, - \lambda\cos\alpha)\) be a point on the line containing the lower Neumann boundary \(\Gamma_{N}^{ - }\). For each angle \(\vartheta \in [0, \pi]\), we define the moving line \({T}_{\lambda, \vartheta}\) to be the line passing through \({P}_{\lambda}\) and forming an angle \(\vartheta\) with \(\Gamma_{N}^{ - }\), that is,
\begin{equation*}
{T}_{\lambda, \vartheta}: = \{{x} \in \R^{2}: \; ({x} - {P}_{\lambda}) \cdot \mathbf{e}_{\vartheta + \alpha - \pi} = 0\},
\end{equation*}
Equivalently, the moving line \({T}_{\lambda, \vartheta}\) can be written in coordinate form as
\begin{equation*}
{T}_{\lambda, \vartheta} = \{{x} \in \R^{2}: \; ({x}_{1} - \lambda\sin\alpha)\cos(\vartheta + \alpha) + ({x}_{2} + \lambda\cos\alpha)\sin(\vartheta + \alpha) = 0\}.
\end{equation*}
(see \autoref{fig31MMP}). We denote by \(\Omega_{\lambda, \vartheta}\) the right open cap of \(\Omega\) determined by \({T}_{\lambda, \vartheta}\), defined as
\begin{equation*}
\Omega_{\lambda, \vartheta}: = \{{x} \in \Omega: \; ({x} - {P}_{\lambda}) \cdot \mathbf{e}_{\vartheta + \alpha} < 0\}.
\end{equation*}
Moreover, we denote by \(\Gamma_{D}'\), \((\Gamma_{N})'\), \((\Gamma_{N}^{ - })'\), and \((\Gamma_{N}^{ + })'\) the reflections of \(\Gamma_{D}\), \(\Gamma_{N}\), \(\Gamma_{N}^{ - }\), and \(\Gamma_{N}^{ + }\) with respect to \({T}_{\lambda, \vartheta}\), respectively. Similarly, define \({Q}_{\lambda} = (\lambda\sin\beta, \lambda\cos\beta)\) and let \(\hat{{T}}_{\lambda, \vartheta}\) be the moving line passing through \({Q}_{\lambda}\) that forms an angle \(\vartheta \in [0, \pi]\) with the upper boundary \(\Gamma_{N}^{ + }\). Explicitly,
\begin{equation*}
\begin{aligned}
\hat{{T}}_{\lambda, \vartheta} & = \{{x} \in \R^{2}: \; ({x} - {Q}_{\lambda}) \cdot \mathbf{e}_{\pi - \beta - \vartheta} = 0\}
\\ & =
\{{x} \in \R^{2}: \; ({x}_{1} - \lambda\sin\beta)\cos(\vartheta + \beta) - ({x}_{2} - \lambda\cos\beta)\sin(\vartheta + \beta) = 0\}.
\end{aligned}
\end{equation*}

We aim to establish the directional monotonicity of \({u}\) along the lines \({T}_{\lambda, \vartheta}\) and \(\hat{T}_{\lambda, \vartheta}\):
\begin{gather}
\label{LY303a}
\nabla {u} \cdot \mathbf{e}_{\vartheta + \alpha} > 0 \text{ on } \Omega \cap {T}_{\lambda, \vartheta},
\\
\label{LY303b}
\nabla {u} \cdot \mathbf{e}_{ - \vartheta - \beta} > 0 \text{ on } \Omega \cap \hat{T}_{\lambda, \vartheta},
\end{gather}
for every \(\lambda > 0\) and for suitable choices of \(\vartheta \in (0, \pi)\). Since the proofs of \eqref{LY303a} and \eqref{LY303b} are analogous, it suffices to focus on establishing \eqref{LY303a}. In order to show the monotonicity properties \eqref{LY303a} and \eqref{LY303b}, we will employ the well-known moving plane method while using a new moving domain in place of the caps \(\Omega_{\lambda, \vartheta}\). So, for \(\lambda \geq 0\) and \(0 < \vartheta < \pi\), we define the family of moving domains
\begin{equation}
{D}_{\lambda, \vartheta} = \big\{ {x} \in \Omega: \; {x}^{\lambda, \vartheta} \in \Omega \; \text{ and } \; ({x} - {P}_{\lambda}) \cdot \mathbf{e}_{\vartheta + \alpha} < 0 \big\}.
\end{equation}
Here, \({x}^{\lambda, \vartheta}\) denotes the reflection of the point \({x} \in \R^{2}\) with respect to the line \({T}_{\lambda, \vartheta}\) (see the left picture in \autoref{fig31MMP} for the case \(\vartheta = \pi/2\)).

Our goal is to show that
\begin{equation}
\label{LY305a}
{w}^{\lambda, \vartheta} = {u}^{\lambda, \vartheta} - {u} > 0 \text{ in } {D}_{\lambda, \vartheta},
\end{equation}
where \({u}^{\lambda, \vartheta}({x}) = {u}({x}^{\lambda, \vartheta})\). Since \({w}^{\lambda, \vartheta}\) may not satisfy the desired \emph{a priori} boundary conditions on \(\partial {D}_{\lambda, \vartheta} \cap \Omega\), we will opt to work on a smaller subdomain of \({D}_{\lambda, \vartheta}\).

\begin{figure}[h]\centering
\begin{tikzpicture}[scale = 2]
\pgfmathsetmacro\xab{1.0}; \pgfmathsetmacro\AngleA{60}; \pgfmathsetmacro\AngleB{40};
\pgfmathsetmacro\AngleC{180 - \AngleA - \AngleB};
\pgfmathsetmacro\yA{ - \xab*tan(90 - \AngleA)}; \pgfmathsetmacro\yB{\xab*tan(90 - \AngleB)};
\pgfmathsetmacro\xA{\xab}; \pgfmathsetmacro\xB{\xab};
\pgfmathsetmacro\PHI{\xab/sin(\AngleA)}; \pgfmathsetmacro\PSI{\xab/sin(\AngleB)};
\pgfmathsetmacro\LAMBDA{\PHI*0.6}; \pgfmathsetmacro\THETA{90};
\pgfmathsetmacro\THETAa{0}; \pgfmathsetmacro\THETAc{2*\THETA - \THETAa};
\pgfmathsetmacro\xP{\LAMBDA*sin(\AngleA)}; \pgfmathsetmacro\yP{ - \LAMBDA*cos(\AngleA)};
\pgfmathsetmacro\xQ{\xab}; \pgfmathsetmacro\yQ{\yA + (\PHI - \LAMBDA)*sin(\THETA)/sin(\THETA + \AngleA)};
\pgfmathsetmacro\lengthB{(\PHI - \LAMBDA)*sin(\AngleA)/sin(\THETA + \AngleA)};
\pgfmathsetmacro\lengthA{(\PHI - \LAMBDA)*sin(\AngleA)/sin(\THETAa + \AngleA)};
\pgfmathsetmacro\lengthC{\lengthA};
\pgfmathsetmacro\xQ{\xP + \lengthB*cos(\THETA + \AngleA - 90)}; \pgfmathsetmacro\yQ{\yP + \lengthB*sin(\THETA + \AngleA - 90)};
\pgfmathsetmacro\xR{\xP + \lengthA*cos(\THETAa + \AngleA - 90)}; \pgfmathsetmacro\yR{\yP + \lengthA*sin(\THETAa + \AngleA - 90)};
\pgfmathsetmacro\xRR{\xP + \lengthC*cos(\THETAc + \AngleA - 90)}; \pgfmathsetmacro\yRR{\yP + \lengthC*sin(\THETAc + \AngleA - 90)};
\fill[gray, yellow, draw = black] (\xP, \yP) -- (\xQ, \yQ) -- (\xRR, \yRR) -- cycle;
\fill[gray, green, draw = black] (\xP, \yP) -- (\xQ, \yQ) -- (\xR, \yR) -- cycle;
\draw[thick] (0, 0) node [left] {${z}_{0}$} -- (\xA, \yA) node [below = 2pt, right] {${z}_{1}$} -- (\xB, \yB) node [above = 2pt, left] {${z}_{2}$} -- cycle;
\draw[red, thick]({\xP - 0.58*(\xQ - \xP)}, {\yP - 0.58*(\yQ - \yP)}) node [below] {${T}_{\lambda, \vartheta}$} -- ({\xP + 1.7*(\xQ - \xP)}, {\yP + 1.7*(\yQ - \yP)});
\fill (\xP, \yP) circle (0.5pt) node [right = 3pt, below] {\footnotesize ${P}_{\lambda}$};
\node at ({\xA + \xab*0.21*cos(90 + \AngleA/2)}, {\yA + \xab*0.21*sin(90 + \AngleA/2)}) {\small $\alpha$};
\node at ({\xB + \xab*0.23*cos( - 90 - \AngleB/2)}, {\yB + \xab*0.23*sin( - 90 - \AngleB/2)}) {\small $\beta$};
\node at ({\xab*0.16*cos(\AngleA/2 - \AngleB/2)}, {\xab*0.16*sin(\AngleA/2 - \AngleB/2)}) {\small $\gamma$};
\node[right] at ({\xA}, {0}) {\small $\Gamma_{D}$};
\node[left] at ({\xA*0.5}, {\yA*0.5}) {\small $\Gamma_{N}$};
\node[left] at ({\xB/2}, {\yB/2}) {\small $\Gamma_{N}$};
\end{tikzpicture}
\hspace{1ex}
\begin{tikzpicture}[scale = 2]
\pgfmathsetmacro\xab{1.0}; \pgfmathsetmacro\AngleA{60}; \pgfmathsetmacro\AngleB{40};
\pgfmathsetmacro\AngleC{180 - \AngleA - \AngleB};
\pgfmathsetmacro\yA{ - \xab*tan(90 - \AngleA)}; \pgfmathsetmacro\yB{\xab*tan(90 - \AngleB)};
\pgfmathsetmacro\xA{\xab}; \pgfmathsetmacro\xB{\xab};
\pgfmathsetmacro\PHI{\xab/sin(\AngleA)}; \pgfmathsetmacro\PSI{\xab/sin(\AngleB)};
\pgfmathsetmacro\LAMBDA{\PHI*0.4}; \pgfmathsetmacro\THETA{\AngleC*1.15};
\pgfmathsetmacro\THETAa{(90 - \AngleA)*1.1}; \pgfmathsetmacro\THETAc{2*\THETA - \THETAa};
\pgfmathsetmacro\xP{\LAMBDA*sin(\AngleA)}; \pgfmathsetmacro\yP{ - \LAMBDA*cos(\AngleA)};
\pgfmathsetmacro\llBa{(\PHI - \LAMBDA)*sin(\AngleA)/sin(\THETA + \AngleA)};
\pgfmathsetmacro\llBb{\LAMBDA*sin(\AngleC)/sin(\THETA - \AngleC)};
\pgfmathsetmacro\llB{min(\llBa, \llBb)}; 
\pgfmathsetmacro\llAa{(\PHI - \LAMBDA)*sin(\AngleA)/sin(\THETAa + \AngleA)};
\pgfmathsetmacro\llAb{\LAMBDA*sin(\AngleC)/sin(\THETAa - \AngleC)};
\pgfmathsetmacro\llAc{max(\llAa, \llAb)}; 
\pgfmathsetmacro\llCa{(\PHI - \LAMBDA)*sin(\AngleA)/sin(\THETAc + \AngleA)};
\pgfmathsetmacro\llCb{\LAMBDA*sin(\AngleC)/sin(\THETAc - \AngleC)};
\pgfmathsetmacro\llCc{max(\llCa, \llCb)}; 
\pgfmathsetmacro\llA{min(\llAc, \llCc)};
\pgfmathsetmacro\llC{min(\llAc, \llCc)};
\pgfmathsetmacro\xQ{\xP + \llB*cos(\THETA + \AngleA - 90)}; \pgfmathsetmacro\yQ{\yP + \llB*sin(\THETA + \AngleA - 90)};
\pgfmathsetmacro\xR{\xP + \llA*cos(\THETAa + \AngleA - 90)}; \pgfmathsetmacro\yR{\yP + \llA*sin(\THETAa + \AngleA - 90)};
\pgfmathsetmacro\xRR{\xP + \llC*cos(\THETAc + \AngleA - 90)}; \pgfmathsetmacro\yRR{\yP + \llC*sin(\THETAc + \AngleA - 90)};
\pgfmathsetmacro\llFa{\llB*sin(\THETA - \THETAa)/sin(\THETAa + \AngleA)};
\pgfmathsetmacro\llFt{\xab*(cot(\AngleA) + cot(\AngleB)) - (\PHI - \LAMBDA)*sin(\THETA)/sin(\THETA + \AngleA)};
\pgfmathsetmacro\llFb{\llFt*sin(\AngleB)/sin(180 - 2*\THETA - \AngleA + \AngleC)};
\pgfmathsetmacro\llF{min(\llFa, \llFb)};
\pgfmathsetmacro\xS{\xab}; \pgfmathsetmacro\yS{\yQ - \llF};
\pgfmathsetmacro\xSS{\xQ + \llF*cos(2*\THETA + 2*\AngleA - 90)}; \pgfmathsetmacro\ySS{\yQ + \llF*sin(2*\THETA + 2*\AngleA - 90)};
\fill[gray, yellow, draw = black] (\xP, \yP) -- (\xQ, \yQ) -- (\xSS, \ySS) -- (\xRR, \yRR) -- cycle;
\fill[gray, green, draw = black] (\xP, \yP) -- (\xQ, \yQ) -- (\xS, \yS) -- (\xR, \yR) -- cycle;
\draw[thick] (0, 0) node [left] {${z}_{0}$} -- (\xab, \yA) node [below = 2pt, right] {${z}_{1}$} -- (\xab, \yB) node [above = 2pt, left] {${z}_{2}$} -- cycle;
\fill (\xP, \yP) circle (0.5pt) node [left] {\footnotesize ${P}_{\lambda}$};
\draw[red, thick]({\xP - 0.4*(\xQ - \xP)}, {\yP - 0.4*(\yQ - \yP)}) node [left] {${T}_{\lambda, \vartheta}$} -- ({\xP + 1.41*(\xQ - \xP)}, {\yP + 1.41*(\yQ - \yP)});
\draw[]({\xP - 0.4*(\xR - \xP)}, {\yP - 0.4*(\yR - \yP)}) -- ({\xP + 1.8*(\xR - \xP)}, {\yP + 1.8*(\yR - \yP)}) node [right] {${T}_{\lambda, \vartheta_{1}}$};
\draw[]({\xP - 0.4*(\xRR - \xP)}, {\yP - 0.4*(\yRR - \yP)}) -- ({\xP + 1.8*(\xRR - \xP)}, {\yP + 1.8*(\yRR - \yP)}) node [left] {${T}_{\lambda, \vartheta_{3}}$};
\draw[thick]({\xS - 1.8*(\xR - \xS)}, {\yS - 1.8*(\yR - \yS)}) -- ({\xS + 2.4*(\xR - \xS)}, {\yS + 2.4*(\yR - \yS)}) node [below = - 7pt] {\small ${T}_{\check{\lambda}, \check{\vartheta}} = \hat{T}_{\hat{\lambda}, \hat{\vartheta}}$};
\end{tikzpicture}
\hspace{1ex}
\begin{tikzpicture}[scale = 2]
\pgfmathsetmacro\xab{1.0}; \pgfmathsetmacro\AngleA{60}; \pgfmathsetmacro\AngleB{40};
\pgfmathsetmacro\AngleC{180 - \AngleA - \AngleB};
\pgfmathsetmacro\yA{ - \xab*tan(90 - \AngleA)}; \pgfmathsetmacro\yB{\xab*tan(90 - \AngleB)};
\pgfmathsetmacro\xA{\xab}; \pgfmathsetmacro\xB{\xab};
\pgfmathsetmacro\PHI{\xab/sin(\AngleA)}; \pgfmathsetmacro\PSI{\xab/sin(\AngleB)};
\pgfmathsetmacro\LAMBDA{\PHI*0.5}; \pgfmathsetmacro\THETA{\AngleC*1.3};
\pgfmathsetmacro\THETAa{(90 - \AngleA)*1.5}; \pgfmathsetmacro\THETAc{2*\THETA - \THETAa};
\pgfmathsetmacro\xP{\LAMBDA*sin(\AngleA)}; \pgfmathsetmacro\yP{ - \LAMBDA*cos(\AngleA)};
\pgfmathsetmacro\llBa{(\PHI - \LAMBDA)*sin(\AngleA)/sin(\THETA + \AngleA)};
\pgfmathsetmacro\llBb{\LAMBDA*sin(\AngleC)/sin(\THETA - \AngleC)};
\pgfmathsetmacro\llB{min(\llBa, \llBb)}; 
\pgfmathsetmacro\llAa{(\PHI - \LAMBDA)*sin(\AngleA)/sin(\THETAa + \AngleA)};
\pgfmathsetmacro\llAb{\LAMBDA*sin(\AngleC)/sin(\THETAa - \AngleC)};
\pgfmathsetmacro\llAc{max(\llAa, \llAb)}; 
\pgfmathsetmacro\llCa{(\PHI - \LAMBDA)*sin(\AngleA)/sin(\THETAc + \AngleA)};
\pgfmathsetmacro\llCb{\LAMBDA*sin(\AngleC)/sin(\THETAc - \AngleC)};
\pgfmathsetmacro\llCc{max(\llCa, \llCb)}; 
\pgfmathsetmacro\llA{min(\llAc, \llCc)};
\pgfmathsetmacro\llC{min(\llAc, \llCc)};
\pgfmathsetmacro\xQ{\xP + \llB*cos(\THETA + \AngleA - 90)}; \pgfmathsetmacro\yQ{\yP + \llB*sin(\THETA + \AngleA - 90)};
\pgfmathsetmacro\xR{\xP + \llA*cos(\THETAa + \AngleA - 90)}; \pgfmathsetmacro\yR{\yP + \llA*sin(\THETAa + \AngleA - 90)};
\pgfmathsetmacro\xRR{\xP + \llC*cos(\THETAc + \AngleA - 90)}; \pgfmathsetmacro\yRR{\yP + \llC*sin(\THETAc + \AngleA - 90)};
\pgfmathsetmacro\llFt{\xab/sin(\AngleB) - \LAMBDA*sin(\THETA)/sin(\THETA - \AngleC)};
\pgfmathsetmacro\llFa{\llFt*sin(\AngleB)/sin(2*\THETA - 2*\AngleC - \AngleB)};
\pgfmathsetmacro\llFb{\llB*sin(\THETAc - \THETA)/sin(\THETAc - \AngleC)};
\pgfmathsetmacro\llF{min(\llFa, \llFb)};
\pgfmathsetmacro\xS{\xQ + \llF*cos(2*\THETA - 2*\AngleC - \AngleB - 90)}; \pgfmathsetmacro\yS{\yQ + \llF*sin(2*\THETA - 2*\AngleC - \AngleB - 90)};
\pgfmathsetmacro\xSS{\xQ - \llF*cos(90 - \AngleB)}; \pgfmathsetmacro\ySS{\yQ - \llF*sin(90 - \AngleB)};
\fill (\xP, \yP) circle (0.5pt) node [left] {\footnotesize ${P}_{\lambda}$};
\fill[gray, yellow, draw = black] (\xP, \yP) -- (\xQ, \yQ) -- (\xSS, \ySS) -- (\xRR, \yRR) -- cycle;
\fill[gray, green, draw = black] (\xP, \yP) -- (\xQ, \yQ) -- (\xS, \yS) -- (\xR, \yR) -- cycle;
\draw[thick] (0, 0) node [left] {${z}_{0}$} -- (\xab, \yA) node [below = 3pt, left] {${z}_{1}$} -- (\xab, \yB) node [above = 2pt, right] {${z}_{2}$} -- cycle;
\draw[red, thick]({\xP - 0.4*(\xQ - \xP)}, {\yP - 0.4*(\yQ - \yP)}) node [left] {${T}_{\lambda, \vartheta}$} -- ({\xP + 1.26*(\xQ - \xP)}, {\yP + 1.26*(\yQ - \yP)});
\draw[]({\xP - 0.4*(\xR - \xP)}, {\yP - 0.4*(\yR - \yP)}) -- ({\xP + 1.8*(\xR - \xP)}, {\yP + 1.8*(\yR - \yP)}) node [right] {${T}_{\lambda, \vartheta_{1}}$};
\draw[]({\xP - 0.4*(\xRR - \xP)}, {\yP - 0.4*(\yRR - \yP)}) -- ({\xP + 1.8*(\xRR - \xP)}, {\yP + 1.8*(\yRR - \yP)}) node [left] {${T}_{\lambda, \vartheta_{3}}$};
\draw[]({\xS - 1.25*(\xQ - \xS)}, {\yS - 1.25*(\yQ - \yS)}) node [below = - 7pt] {\small ${T}_{\check{\lambda}, \check{\vartheta}} = \hat{T}_{\hat{\lambda}, \hat{\vartheta}}$} -- ({\xS + 1.43*(\xQ - \xS)}, {\yS + 1.43*(\yQ - \yS)});
\end{tikzpicture} \vspace*{-2ex}
\caption{The moving lines $\color{red}{T}_{\lambda, \vartheta}$ and moving domains $\color{green}{D}_{\lambda, \vartheta, \vartheta_{1}}$}
\label{fig31MMP}
\end{figure}
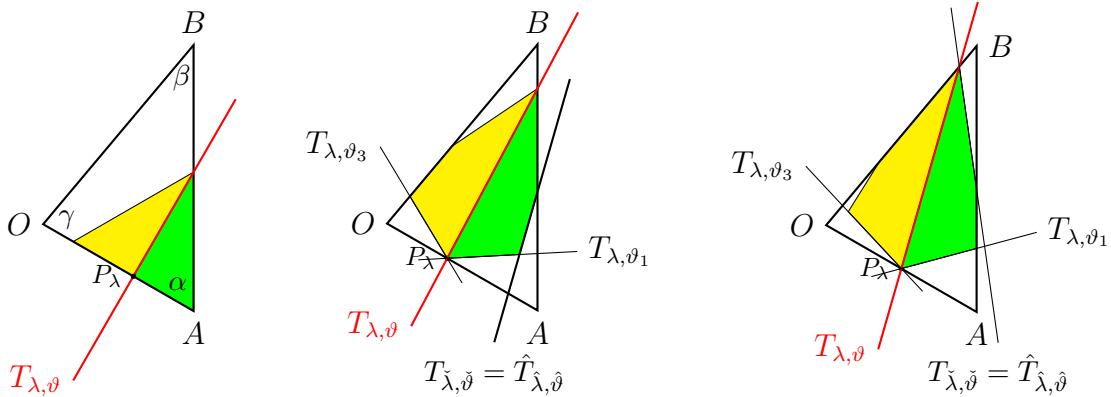

For \(\lambda \geq 0\), \(\vartheta \in (0, \pi)\), and \(0 \leq \vartheta_{1} < \vartheta_{3} \leq \pi\) with \(\vartheta_{1} + \vartheta_{3} = 2\vartheta\), we consider a family of moving domains \({D}_{\lambda, \vartheta, \vartheta_{1}}\) instead of the right cap \(\Omega_{\lambda, \vartheta}\) as follows:
\begin{equation}
{D}_{\lambda, \vartheta, \vartheta_{1}} = \big\{ {x} \in \Omega: \; {x}^{\lambda, \vartheta} \in \Omega \; \text{ and } \; ({x} - {P}_{\lambda}) \cdot \mathbf{e}_{\vartheta + \alpha} < 0 < ({x} - {P}_{\lambda}) \cdot \mathbf{e}_{\vartheta_{1} + \alpha} \big\}
\end{equation}
as illustrated in \autoref{fig31MMP}. Our goal is to prove that
\begin{equation}
\label{LY305b}
{w}^{\lambda, \vartheta} = {u}^{\lambda, \vartheta} - {u} > 0 \text{ in } {D}_{\lambda, \vartheta, \vartheta_{1}}.
\end{equation}
We now clarify some important features of these domains: (i) \({D}_{\lambda, \vartheta} \supset {D}_{\lambda, \vartheta, \vartheta_{1}}\), and both \({D}_{\lambda, \vartheta}\) and \({D}_{\lambda, \vartheta, \vartheta_{1}}\) are subsets of the right cap \(\Omega_{\lambda, \vartheta}\); (ii) Both \(\partial {D}_{\lambda, \vartheta}\) and \(\partial {D}_{\lambda, \vartheta, \vartheta_{1}}\) contain \({T}_{\lambda, \vartheta} \cap \Omega\); (iii) \({D}_{\lambda, \vartheta} = {D}_{\lambda, \vartheta, \vartheta_{1}}\) whenever \(\vartheta_{1} = \max\{0, 2\vartheta - \pi\}\).
The unit outward normal vector to \({D}_{\lambda, \vartheta, \vartheta_{1}}\) on its boundary \(\partial {D}_{\lambda, \vartheta, \vartheta_{1}}\) is still denoted by \({\nu}\). The boundary \(\partial {D}_{\lambda, \vartheta, \vartheta_{1}}\) consists of three parts:
\begin{enumerate}
\item[(1)]
\(\Gamma_{\lambda, \vartheta, \vartheta_{1}}^{0} = {T}_{\lambda, \vartheta} \cap \partial {D}_{\lambda, \vartheta, \vartheta_{1}}\). This set is nonempty whenever \({D}_{\lambda, \vartheta, \vartheta_{1}}\) is nonempty.
\item[(2)]
\(\Gamma_{\lambda, \vartheta, \vartheta_{1}}^{1} = \big( \Gamma_{D} \cup \Gamma_{D}' \big) \cap \big( \partial {D}_{\lambda, \vartheta, \vartheta_{1}} \setminus {T}_{\lambda, \vartheta} \big).\) This portion is contained in \(\Gamma_{D}\) since we always assume \(\lambda > 0\) and \(\vartheta \in [\pi/2 - \alpha, \pi)\).
\item[(3)]
\(\Gamma_{\lambda, \vartheta, \vartheta_{1}}^{2} = \partial {D}_{\lambda, \vartheta, \vartheta_{1}} \setminus \big( \Gamma_{\lambda, \vartheta, \vartheta_{1}}^{0} \cup \Gamma_{\lambda, \vartheta, \vartheta_{1}}^{1} \big).\) The component \(\Gamma_{\lambda, \vartheta, \vartheta_{1}}^{2}\) is associated with the boundaries corresponding to \(\Gamma_{N} \cup {T}_{\lambda, \vartheta_{1}}\) and its reflection, and it further decomposes into
\begin{equation*}
\Gamma_{\lambda, \vartheta, \vartheta_{1}}^{2A} = \Gamma_{\lambda, \vartheta, \vartheta_{1}}^{2} \cap \big( \Gamma_{N}^{ - } \cup (\Gamma_{N}^{ - })' \cup {T}_{\lambda, \vartheta_{1}} \big) \quad
\text{and} \quad
\Gamma_{\lambda, \vartheta, \vartheta_{1}}^{2B} = \Gamma_{\lambda, \vartheta, \vartheta_{1}}^{2} \cap \big( \Gamma_{N}^{ + } \cup (\Gamma_{N}^{ + })' \big).
\end{equation*}
\end{enumerate}
Moreover, \(\Gamma_{\lambda, \vartheta, \vartheta_{1}}^{2B}\) is a line segment and is contained in \({T}_{\check{\lambda}, \check{\vartheta}} = \hat{T}_{\hat{\lambda}, \hat{\vartheta}}\). Here, \({T}_{\check{\lambda}, \check{\vartheta}} = \hat{T}_{\hat{\lambda}, \hat{\vartheta}}\) stands for the line associated with the reflection of the upper boundary \(\Gamma_{N}^{ + }\) with respect to \({T}_{\lambda, \vartheta}\), where
\begin{equation*}
\check{\vartheta} = 2\vartheta - \gamma, \quad \hat{\vartheta} = \pi - 2\vartheta + 2\gamma,
\end{equation*}
and where \(\hat{\lambda} = \hat{\lambda}(\lambda, \vartheta)\) and \(\check{\lambda} = \check{\lambda}(\lambda, \vartheta)\) depend on \(\lambda\) and \(\vartheta\) as
\begin{equation}
\hat{\lambda} = \frac{\lambda \sin\vartheta}{\sin(\vartheta - \gamma)}, \quad \check{\lambda} = \lambda + \frac{\lambda \sin\gamma}{\sin(2\vartheta - \gamma)}.
\end{equation}
Observe that \({D}_{\lambda, \vartheta, \vartheta_{1}}\) is either a triangle or a quadrilateral. Moreover, for \(\lambda \geq 0\) and any angles satisfying \(\pi \geq \vartheta > \vartheta_{1}' > \vartheta_{1} \geq 0\), we have
\begin{equation*}
{D}_{\lambda, \vartheta, \vartheta_{1}} \supset {D}_{\lambda, \vartheta, \vartheta_{1}'} \quad \text{and} \quad \Gamma_{\lambda, \vartheta, \vartheta_{1}}^{2B} \supset \Gamma_{\lambda, \vartheta, \vartheta_{1}'}^{2B}.
\end{equation*}
For simplicity, in what follows we omit the subscripts \(\vartheta, \vartheta_{1}\) and denote the boundary parts by \(\Gamma_{\lambda}^{0}\), \(\Gamma_{\lambda}^{1}\), \(\Gamma_{\lambda}^{2}\), \(\Gamma_{\lambda}^{2A}\), and \(\Gamma_{\lambda}^{2B}\).
As a consequence, \({w}^{\lambda, \vartheta}\) satisfies
\begin{equation*} 
\begin{cases}
\Delta {w}^{\lambda, \vartheta} + {c}^{\lambda, \vartheta}({x}) {w}^{\lambda, \vartheta} = 0 & \text{in } {D}_{\lambda, \vartheta, \vartheta_{1}},
\\
{w}^{\lambda, \vartheta} = 0 & \text{on } \Gamma_{\lambda}^{0},
\\
{w}^{\lambda, \vartheta} > 0 & \text{on } \Gamma_{\lambda}^{1} \quad \text{when } \vartheta > \pi/2 - \alpha,
\\
{w}^{\lambda, \vartheta} = 0 & \text{on } \Gamma_{\lambda}^{1} \quad \text{when } \vartheta = \pi/2 - \alpha,
\end{cases}
\end{equation*}
for \(\lambda > 0\) and
\begin{equation*}
\max\big\{\frac{\pi}{2} - \alpha, 0\big\} \leq \vartheta < \pi.
\end{equation*}
Here,
\begin{equation*}
{c}^{\lambda, \vartheta}({x}) =
\begin{cases}
\frac{{f}({u}^{\lambda, \vartheta}({x})) - {f}({u}({x}))}{{u}^{\lambda, \vartheta}({x}) - {u}({x})} & \text{if } {w}^{\lambda, \vartheta}({x}) \neq 0,
\\
0 & \text{if } {w}^{\lambda, \vartheta}({x}) = 0,
\end{cases}
\end{equation*}
is uniformly bounded with respect to \(\lambda\) and \(\vartheta\); that is, \(|{c}^{\lambda, \vartheta}({x})| < {c}_{0}\) for some constant \({c}_{0} > 0\).

Obtaining an \emph{a priori} boundary condition for \({w}^{\lambda, \vartheta}\) on \(\Gamma_{\lambda}^{2}\) is notably challenging and constitutes the main difficulty in proving the positivity of \({w}^{\lambda, \vartheta}\). The following lemma illustrates that the moving plane method can be initiated by establishing an \emph{a priori} boundary condition for \({w}^{\lambda, \vartheta}\) on \(\Gamma_{\lambda}^{2}\).

\begin{lma} \label{lma301}
Let \(\vartheta \in (0, \pi)\) be fixed with \(\vartheta \geq \pi/2 - \alpha\). Choose \(\vartheta_{1}\) such that \(\max\{2\vartheta - \pi, 0\} \leq \vartheta_{1} < \vartheta\), and let \(\Lambda \in (0, \lambda_{M}(\vartheta))\), where
\begin{equation*}
\lambda_{M}(\vartheta) = \sup\{\lambda \in \R: \; {T}_{\lambda, \vartheta} \cap \Omega = \emptyset\}.
\end{equation*}
Assume that, for every \(\lambda \geq \Lambda\), the following conditions hold:
\begin{subequations} \label{LY316}
\begin{align}
\label{LY316a}
\Delta {w}^{\lambda, \vartheta} + {c}^{\lambda, \vartheta}{w}^{\lambda, \vartheta} = 0 & \text{ in } {D}_{\lambda, \vartheta, \vartheta_{1}}
\\ \label{LY316b}
{w}^{\lambda, \vartheta} = 0 & \text{ on } \Gamma_{\lambda, \vartheta, \vartheta_{1}}^{0},
\\ \label{LY316c}
{w}^{\lambda, \vartheta} \geq 0 & \text{ on } \Gamma_{\lambda, \vartheta, \vartheta_{1}}^{1},
\\ \label{LY316d}
\nabla {w}^{\lambda, \vartheta} \cdot \nu \geq 0 & \text{ on } \Gamma_{\lambda, \vartheta, \vartheta_{1}}^{2},
\\ \label{LY316e}
{w}^{\lambda, \vartheta}\not \equiv 0 & \text{ in }
{D}_{\lambda, \vartheta, \vartheta_{1}}.
\end{align}
\end{subequations}
Then, \eqref{LY305b} and \eqref{LY303a} hold for every \(\lambda \geq \Lambda\).
\end{lma}

\begin{proof}
By the definition of the moving domain \({D}_{\lambda, \vartheta, \vartheta_{1}}\), the closure of the union of \({D}_{\lambda, \vartheta, \vartheta_{1}}\) and its reflection \(({D}_{\lambda, \vartheta, \vartheta_{1}})'\) (with respect to \({T}_{\lambda, \vartheta}\)) is convex. Hence, the second condition in \autoref{lma201} is satisfied. Let \(\eta > 0\) be the constant provided by \autoref{lma201}, so that the maximum principle holds for mixed boundary value problems in domains of sufficiently small width.

\textbf{Step 1: Initiation of the Moving Plane Method}.
By the definition of \({D}_{\lambda, \vartheta, \vartheta_{1}}\), there exists a constant \(\varepsilon_{1} > 0\) such that for every \({\lambda} \in ({\lambda}_{M}(\vartheta) - \varepsilon_{1}, {\lambda}_{M}(\vartheta) )\) we have
\begin{equation*}
{D}_{\lambda, \vartheta, \vartheta_{1}}\subset \big\{{x} \in \R^{2}: \; \operatorname{dist}({x}, {T}_{\lambda, \vartheta}) < {\eta}\big\}.
\end{equation*}
Then by applying the maximum principle with mixed boundary conditions (see \autoref{lma201}), one deduces that \({w}^{\lambda, \vartheta}\) is strictly positive in \({D}_{\lambda, \vartheta, \vartheta_{1}}\). Moreover, applying the Hopf boundary lemma to \({w}^{\lambda, \vartheta}\) yields \eqref{LY303a}. Therefore, \eqref{LY305b} and \eqref{LY303a} hold for all \({\lambda}\) in the range
\({\lambda}_{M}(\vartheta) - \varepsilon_{1} < {\lambda} < {\lambda}_{M}(\vartheta)\).

\textbf{Step 2: \eqref{LY305b} and \eqref{LY303a} hold for every \({\lambda} \geq \Lambda\)}.
We argue by contradiction. Define
\begin{equation*}
\bar{\lambda} = \inf\{\lambda' > 0: \; \text{\eqref{LY305b} and \eqref{LY303a} hold for every }\lambda \geq \lambda'\}.
\end{equation*}
Suppose that \(\bar{\lambda} > \Lambda\). By continuity, we have
\begin{equation*}
{w}^{\bar{\lambda}, \vartheta} \geq 0 \text{ in } {D}_{\bar{\lambda}, \vartheta, \vartheta_{1}}.
\end{equation*}
Then, by the strong maximum principle and the Hopf boundary lemma, we deduce that \({w}^{\bar{\lambda}, \vartheta}\) is strictly positive in \({D}_{\bar{\lambda}, \vartheta, \vartheta_{1}}\) and on the smooth portion of \(\Gamma_{\bar{\lambda}, \vartheta, \vartheta_{1}}^{2}\). Even though \(\Gamma_{\lambda, \vartheta, \vartheta_{1}}^{2}\) may contain non-smooth boundary points (namely, corners formed by two lines), one may verify (see \cite[Remark 1]{YCG21}) that \({w}^{\lambda, \vartheta}\) is also positive at these points. Hence, for every \({\lambda} \geq \Lambda\) one obtains
\begin{equation} \label{LY317a}
\begin{aligned}
{w}^{\lambda, \vartheta} > 0 & \text{ in } \overline{{D}_{\lambda, \vartheta, \vartheta_{1}}} \setminus {T}_{\lambda, \vartheta} \text{ if }\vartheta \in (\pi/2 - \alpha, \pi),
\\
{w}^{\lambda, \vartheta} > 0 & \text{ in }
\overline{{D}_{\lambda, \vartheta, \vartheta_{1}}} \setminus ({T}_{\lambda, \vartheta} \cup \Gamma_{\lambda, \vartheta, \vartheta_{1}}^{1}) \text{ if }\vartheta = \pi/2 - \alpha,
\end{aligned}
\end{equation}
and
\begin{equation*} 
\nabla {u} \cdot \mathbf{e}_{\vartheta + \alpha} > 0 \text{ on } {T}_{\bar{\lambda}, \vartheta} \cap \Omega.
\end{equation*}

\begin{figure}[h]\centering
\begin{tikzpicture}[scale = 1.6]
\pgfmathsetmacro\xab{3}; \pgfmathsetmacro\AngleA{60}; \pgfmathsetmacro\AngleB{45};
\pgfmathsetmacro\AngleC{180 - \AngleA - \AngleB};
\pgfmathsetmacro\yA{ - \xab*tan(90 - \AngleA)}; \pgfmathsetmacro\yB{\xab*tan(90 - \AngleB)};
\pgfmathsetmacro\xA{\xab}; \pgfmathsetmacro\xB{\xab};
\pgfmathsetmacro\PHI{\xab/sin(\AngleA)}; \pgfmathsetmacro\PSI{\xab/sin(\AngleB)};
\pgfmathsetmacro\LAMBDA{\PHI*0.1}; \pgfmathsetmacro\THETA{\AngleC*0.80};
\pgfmathsetmacro\THETAa{(90 - \AngleA)*1.2}; \pgfmathsetmacro\THETAc{2*\THETA - \THETAa};
\pgfmathsetmacro\xP{\LAMBDA*sin(\AngleA)}; \pgfmathsetmacro\yP{ - \LAMBDA*cos(\AngleA)};
\pgfmathsetmacro\llBa{(\PHI - \LAMBDA)*sin(\AngleA)/sin(\THETA + \AngleA)};
\pgfmathsetmacro\llBb{\LAMBDA*sin(\AngleC)/sin(\THETA - \AngleC)};
\pgfmathsetmacro\llB{max(\llBa, \llBb)}; 
\pgfmathsetmacro\xQ{\xP + \llB*cos(\THETA + \AngleA - 90)}; \pgfmathsetmacro\yQ{\yP + \llB*sin(\THETA + \AngleA - 90)};
\pgfmathsetmacro\llAa{(\PHI - \LAMBDA)*sin(\AngleA)/sin(\THETAa + \AngleA)};
\pgfmathsetmacro\llAb{\LAMBDA*sin(\AngleC)/sin(\THETAa - \AngleC)};
\pgfmathsetmacro\llAc{max(\llAa, \llAb)}; 
\pgfmathsetmacro\llCa{(\PHI - \LAMBDA)*sin(\AngleA)/sin(\THETAc + \AngleA)};
\pgfmathsetmacro\llCb{\LAMBDA*sin(\AngleC)/sin(\THETAc - \AngleC)};
\pgfmathsetmacro\llCc{min(\llCa, \llCb)}; 
\pgfmathsetmacro\llA{min(\llAc, \llCc)}; \pgfmathsetmacro\llC{min(\llAc, \llCc)};
\pgfmathsetmacro\xR{\xP + \llA*cos(\THETAa + \AngleA - 90)}; \pgfmathsetmacro\yR{\yP + \llA*sin(\THETAa + \AngleA - 90)};
\pgfmathsetmacro\xRR{\xP + \llC*cos(\THETAc + \AngleA - 90)}; \pgfmathsetmacro\yRR{\yP + \llC*sin(\THETAc + \AngleA - 90)};
\pgfmathsetmacro\llFa{\llB*sin(\THETA - \THETAa)/sin(\THETAa + \AngleA)};
\pgfmathsetmacro\llFt{\xab*(cot(\AngleA) + cot(\AngleB)) - (\PHI - \LAMBDA)*sin(\THETA)/sin(\THETA + \AngleA)};
\pgfmathsetmacro\llFb{\llFt*sin(\AngleB)/sin(180 - 2*\THETA - \AngleA + \AngleC)};
\pgfmathsetmacro\llF{min(\llFa, \llFb)};
\pgfmathsetmacro\xS{\xab}; \pgfmathsetmacro\yS{\yQ - \llF};
\pgfmathsetmacro\xSS{\xQ + \llF*cos(2*\THETA + 2*\AngleA - 90)}; \pgfmathsetmacro\ySS{\yQ + \llF*sin(2*\THETA + 2*\AngleA - 90)};
\fill[gray, green, draw = black, thin] (\xP, \yP) -- (\xQ, \yQ) -- (\xS, \yS) -- (\xR, \yR) -- cycle;
\draw[thin] (0, 0) node [left] {${z}_{0}$} -- (\xab, \yA) node [above right] {${z}_{1}$} -- (\xab, \yB) node [below right] {${z}_{2}$} -- cycle;
\draw[dotted] (\xP, \yP) -- ({\xP + 1.6*(\xR - \xP)}, {\yP + 1.6*(\yR - \yP)}) node [below] {\footnotesize $ {T}_{\bar{\lambda}, \vartheta_{1}}$};
\draw[dotted] (\xR, \yR) -- ({\xR + 1.9*(\xS - \xR)}, {\yR + 1.9*(\yS - \yR)}) node [above] {\footnotesize ${T}_{\check{\lambda}(\bar{\lambda}, \vartheta), \check{\vartheta}}$};
\draw[red!70]({\xP - 0.2*(\xQ - \xP)}, {\yP - 0.2*(\yQ - \yP)}) node [below] {\tiny ${T}_{\bar{\lambda}, \vartheta}$} -- ({\xP + 1.2*(\xQ - \xP)}, {\yP + 1.2*(\yQ - \yP)});
\pgfmathsetmacro\Dela{\xab*0.2}; \pgfmathsetmacro\Delb{\Dela*0.2};
\pgfmathsetmacro\Delc{(\Dela + \Delb)/2/sin(\THETA + \AngleA)};
\pgfmathsetmacro\Deld{(\Dela/sin(\THETA + \AngleA) - \Delc)*tan(180 - \THETA - \AngleA)*1/3};
\pgfmathsetmacro\xMc{\xQ - \Delb}; \pgfmathsetmacro\yMc{\yQ - \Delb*tan(\THETA + \AngleA - 90)};
\pgfmathsetmacro\xMca{\xMc}; \pgfmathsetmacro\yMca{\yMc - \Deld/sin(\THETA + \AngleA)};
\pgfmathsetmacro\xMcb{\xP + \Deld/sin(\THETA)*cos(\AngleA - 90)};
\pgfmathsetmacro\yMcb{\yP + \Deld/sin(\THETA)*sin(\AngleA - 90)};
\pgfmathsetmacro\xMcc{\xMc}; \pgfmathsetmacro\yMcc{ - \xMcc*cot(\AngleA)};
\fill[pattern color = orange!50, pattern = dots, draw = orange, dashed] (\xMca, \yMca) -- (\xMcb, \yMcb) -- (\xMcc, \yMcc) -- cycle;
\pgfmathsetmacro\xMb{\xQ - \Dela}; \pgfmathsetmacro\yMb{\yQ - \Dela*tan(\THETA + \AngleA - 90)};
\pgfmathsetmacro\xMba{\xMb}; \pgfmathsetmacro\yMba{\yMb + \Deld/sin(\THETA + \AngleA)};
\pgfmathsetmacro\xMbb{\xQ}; \pgfmathsetmacro\yMbb{\yQ + \Deld/sin(\THETA + \AngleA)};
\pgfmathsetmacro\xMbc{\xMb}; \pgfmathsetmacro\yMbc{ - \xMbc*cot(\AngleA)};
\pgfmathsetmacro\xMbd{\xA}; \pgfmathsetmacro\yMbd{\yA};
\draw[cyan, pattern color = cyan!60, pattern = north west lines] (\xMba, \yMba) -- (\xMbb, \yMbb) -- (\xMbd, \yMbd) -- (\xMbc, \yMbc) -- cycle;
\pgfmathsetmacro\xMa{\xQ - \Delc*sin(\THETA + \AngleA)}; \pgfmathsetmacro\yMa{\yQ + \Delc*cos(\THETA + \AngleA)};
\pgfmathsetmacro\xMaa{\xMa + \Deld*cos(\THETA + \AngleA)};
\pgfmathsetmacro\yMaa{\yMa + \Deld*sin(\THETA + \AngleA)};
\pgfmathsetmacro\xMab{\xMa - \Deld*cos(\THETA + \AngleA)};
\pgfmathsetmacro\yMab{\yMa - \Deld*sin(\THETA + \AngleA)};
\pgfmathsetmacro\xMac{\xP - \Deld/sin(\THETA)*cos(\AngleA - 90)};
\pgfmathsetmacro\yMac{\yP - \Deld/sin(\THETA)*sin(\AngleA - 90)};
\pgfmathsetmacro\xMad{\xP + \Deld/sin(\THETA)*cos(\AngleA - 90)};
\pgfmathsetmacro\yMad{\yP + \Deld/sin(\THETA)*sin(\AngleA - 90)};
\draw (\xMaa, \yMaa) -- (\xMab, \yMab);
\draw[very thin] ({\xMaa - 4.2*(\xMab - \xMaa)}, {\yMaa - 4.2*(\yMab - \yMaa)}) node[above] {\footnotesize ${L}_{0}$} -- ({\xMaa + 4.2*(\xMab - \xMaa)}, {\yMaa + 4.2*(\yMab - \yMaa)});
\fill[pattern color = blue!60, pattern = north east lines, draw = blue] (\xMaa, \yMaa) -- (\xMab, \yMab) -- (\xMad, \yMad) -- (\xMac, \yMac) -- cycle;
\pgfmathsetmacro\xPA{\xP - \Deld/2/sin(\THETA)*cos(\AngleA - 90)};
\pgfmathsetmacro\yPA{\yP - \Deld/2/sin(\THETA)*sin(\AngleA - 90)};
\pgfmathsetmacro\xQA{\xQ};
\pgfmathsetmacro\yQA{\yQ + \Deld/2/sin(\THETA + \AngleA)};
\draw[red, thick]({\xPA - 0.2*(\xQA - \xPA)}, {\yPA - 0.2*(\yQA - \yPA)}) node [above] {\tiny ${T}_{\lambda, \vartheta}$} -- ({\xPA + 1.2*(\xQA - \xPA)}, {\yPA + 1.2*(\yQA - \yPA)});
\node[orange, fill = white] at ({(\xMca + \xMcb + \xMcc)/3}, {(\yMca + \yMcb + \yMcc)/3}) [fill = white, rotate = 0] {\tiny $\mathcal{D}_{3}$ };
\node[cyan, fill = white] at ({(\xMba + \xMbd)/2}, {(\yMba + \yMbd)/2}) [fill = white, rotate = 0] {\tiny ${D}_{2}$ };
\node[blue, rotate = {\THETA + \AngleA - 90}, fill = white] at ({(\xMaa + \xMad)/2}, {(\yMaa + \yMad)/2}) {\tiny $\mathcal{D}_{1}$ };
\fill (\xQ, \yQ) circle (0.65pt) node [below = 2pt, right] {\footnotesize $\bar{x}$};
\draw[thin, ->] ({\xQ*1.2}, {\yQ*1.2}) -- ++ ({\THETA + \AngleA - 90}: {\xab*0.2}) node [above = 5pt, right = - 3ex] {\footnotesize $\mathbf{e}_{\vartheta + \alpha - \pi/2}$}; 
\draw[thin, ->] ({\xQ*1.2}, {\yQ*1.2}) -- ++ ({0}: {\xab*0.2}) node [right] {\footnotesize $\mathbf{e}_{0}$};
\draw[thin, ->] ({\xQ*1.2}, {\yQ*1.2}) -- ++ ({\THETA + \AngleA}: {\xab*0.2}) node [above = - 2pt] {\footnotesize $\mathbf{e}_{\vartheta + \alpha}$}; 
\end{tikzpicture} \vspace*{-1ex}
\caption{The domains $\color{blue}\mathcal{D}_{1}$, $\color{cyan}\mathcal{D}_{2}$ and $\color{orange}\mathcal{D}_{3}$ }
\label{fig32MMP}
\end{figure}

Assume that the intersection \(\Gamma_{D} \cap {T}_{\bar{\lambda}, \vartheta}\) is nonempty (the case \(\Gamma_{D} \cap {T}_{\bar{\lambda}, \vartheta} = \emptyset\) can be treated similarly). Define \({\bar{x}} = (\bar{x}_{1}, \bar{x}_{2})\) as the unique point in \(\Gamma_{D} \cap {T}_{\bar{\lambda}, \vartheta}\subset \partial {D}_{\bar{\lambda}, \vartheta, \vartheta_{1}}\), and set \({P}_{\bar{\lambda}} = (\bar{\lambda}\sin\alpha, - \bar{\lambda}\cos\alpha)\).
Fix any \(\delta_{1}, \delta_{2} \in (0, {\eta})\) with \(\delta_{2} < \delta_{1}\), and choose \(\delta_{3} > 0\) such that the intersection \({T}_{\bar{\lambda}, \vartheta} \cap {L}_{0}\) is contained in \(\Omega \cap \{{x}: \; - \delta_{1} < ({x} - {\bar{x}}) \cdot \mathbf{e}_{0} < - \delta_{2}\}\), where the line
\begin{equation*}
{L}_{0} = \{{x}: \; ({x} - {\bar{x}}) \cdot \mathbf{e}_{\vartheta + \alpha - \pi/2} = - \delta_{3} \}
\end{equation*}
is defined to be perpendicular to \({T}_{\bar{\lambda}, \vartheta}\). Since \eqref{LY303a} holds for \({\lambda} \geq \bar{\lambda}\), the continuity of \(\nabla {u}\) implies the existence of a small constant \(\delta_{4} \in (0, {\eta}/3)\) such that
\begin{equation*}
\nabla {u} \cdot \mathbf{e}_{\vartheta + \alpha} > 0 \text{ on } \big\{{x} \in {L}_{0}: \; |({x} - \bar{x}) \cdot \mathbf{e}_{\vartheta + \alpha}| \leq 3\delta_{4}\big\},
\end{equation*}
and hence,
\begin{equation} \label{LY318a}
{w}^{\lambda, \vartheta} > 0 \text{ in } {D}_{\lambda, \vartheta, \vartheta_{1}} \cap \big\{{x} \in {L}_{0}: \; |({x} - \bar{x}) \cdot \mathbf{e}_{\vartheta + \alpha}| \leq \delta_{4}\big\} \; \text{ for } |\lambda - \bar{\lambda}| < \delta_{4}.
\end{equation}
Set
\begin{align*}
\mathcal{D}_{1} & = \big\{{x} \in \Omega: \; \big|({x} - \bar{x}) \cdot \mathbf{e}_{\vartheta + \alpha}\big| \leq \delta_{4}, \quad ({x} - \bar{x}) \cdot \mathbf{e}_{\vartheta + \alpha - \pi/2} \leq - \delta_{3}\big\},
\\
\mathcal{D}_{2} & = \big\{{x} \in \Omega: \; ({x} - \bar{x}) \cdot \mathbf{e}_{\vartheta + \alpha} \leq \delta_{4}, \quad ({x} - \bar{x}) \cdot \mathbf{e}_{0} \geq - \delta_{1}\big\},
\\
\mathcal{D}_{3} & = \big\{{x} \in \Omega: \; ({x} - \bar{x}) \cdot \mathbf{e}_{\vartheta + \alpha} \leq - \delta_{4}, \quad ({x} - \bar{x}) \cdot \mathbf{e}_{0} \leq - \delta_{2}\big\}.
\end{align*}
It is clear that \(\mathcal{D}_{1}\cup\mathcal{D}_{2}\cup\mathcal{D}_{3} = \{{x} \in \Omega: \; ({x} - {P}_{\bar{\lambda}}) \cdot \mathbf{e}_{\vartheta + \alpha} \leq \delta_{4} \}\); see \autoref{fig32MMP}. In particular, from \eqref{LY317a} the function \({w}^{\bar{\lambda}, \vartheta}\) is positive in the compact set \(\overline{\mathcal{D}_{3}} \cap \overline{{D}_{\bar{\lambda}, \vartheta, \vartheta_{1}}}\). By continuity, there exists a small constant \(\varepsilon_{2} \in (0, \delta_{4})\) such that
\begin{equation} \label{LY318b}
{w}^{\lambda, \vartheta} > 0 \text{ in } \overline{{D}_{\lambda, \vartheta, \vartheta_{1}}} \cap \overline{\mathcal{D}_{3}} \quad \text{if } {\lambda} \in (\bar{\lambda} - \varepsilon_{2}, \bar{\lambda}).
\end{equation}
Now, let \({\lambda} \in (\bar{\lambda} - \varepsilon_{2}, \bar{\lambda})\). From \eqref{LY318a} and \eqref{LY318b}, the function \({w}^{\lambda, \vartheta}\) satisfies
\begin{equation} \label{LY319}
\begin{cases}
\Delta {w}^{\lambda, \vartheta} + {c}^{\lambda, \vartheta} {w}^{\lambda, \vartheta} = 0 & \text{in } {D}_{\lambda, \vartheta, \vartheta_{1}} \cap \mathcal{D}_{j},
\\
{w}^{\lambda, \vartheta} \geq 0, \quad {w}^{\lambda, \vartheta}\not\equiv 0 & \text{on } \partial ({D}_{\lambda, \vartheta, \vartheta_{1}} \cap \mathcal{D}_{j} ) \setminus \Gamma_{\lambda}^{2},
\\
\nabla {w}^{\lambda, \vartheta} \cdot {\nu} \geq 0 & \text{on } \partial ({D}_{\lambda, \vartheta, \vartheta_{1}} \cap \mathcal{D}_{j} ) \cap \Gamma_{\lambda}^{2},
\end{cases}
\end{equation}
for \(j = 1\). By applying the maximum principle (see \autoref{lma201}), we deduce that
\begin{equation} \label{LY318c}
{w}^{\lambda, \vartheta} > 0 \text{ in } {D}_{\lambda, \vartheta, \vartheta_{1}} \cap \mathcal{D}_{j},
\end{equation}
for \(j = 1\). Using this result together with \eqref{LY318a} and \eqref{LY318b}, one shows that \({w}^{\lambda, \vartheta}\) satisfies \eqref{LY319} for \(j = 2\). Hence, by the maximum principle from \autoref{lma201} we conclude that \eqref{LY318c} holds for \(j = 2\). Furthermore, the Hopf lemma implies that \eqref{LY305b} and \eqref{LY303a} remain valid for \({\lambda} \in (\bar{\lambda} - \varepsilon_{2}, \bar{\lambda})\).

This contradicts the definition of \(\bar{\lambda}\). Hence, we conclude that \(\bar{\lambda} \leq \Lambda\) and consequently, \eqref{LY305b} and \eqref{LY303a} hold for every \({\lambda} > \Lambda\). Moreover, by referring back to the argument at the beginning of Step 2 one deduces that \eqref{LY305b} and \eqref{LY303a} also hold for \({\lambda} = \Lambda\). This completes the proof.
\end{proof}

\begin{rmks}
The proof of the above lemma relies on the maximum principle for narrow domains, as stated in \autoref{lma201}.
However, a simpler argument can be obtained by applying the maximum principle established in \cite[Lemma 2.5]{Yao26} in domains of sufficiently small measure whose Neumann boundary lies in two hyperplanes.
Indeed, let \(\eta > 0\) be the small constant from \cite[Lemma 2.5]{Yao26}. Let \({K}\) be a fixed compact subset, and choose \(\varepsilon_{1} > 0\) such that
\begin{equation*}
{D}_{\lambda, \vartheta, \vartheta_{1}} \supset {K} \quad \text{and} \quad \big|{D}_{\lambda, \vartheta, \vartheta_{1}} \setminus {K}\big| < \eta \gamma \quad \text{for} \quad |\lambda - \bar{\lambda}| < \varepsilon_{1}.
\end{equation*}
By the positivity of \({w}^{\bar{\lambda}, \vartheta}\), there exists \(\varepsilon_{2} \in (0, \varepsilon_{1})\) such that
\begin{equation*}
{w}^{\lambda, \vartheta} > 0 \text{ in } {K} \quad \text{for every } \lambda \in [\bar{\lambda} - \varepsilon_{2}, \bar{\lambda} + \varepsilon_{2}].
\end{equation*}
Now, in the remaining subdomain \(\mathcal{D}: \; = {D}_{\lambda, \vartheta, \vartheta_{1}} \setminus {K}\), the function \({w}^{\lambda, \vartheta}\) satisfies
\begin{equation*}
\Delta {w}^{\lambda, \vartheta} + {c}^{\lambda, \vartheta} {w}^{\lambda, \vartheta} = 0 \text{ in } \mathcal{D}, \;
\partial_{\nu} {w}^{\lambda, \vartheta} \geq 0 \text{ on } \Gamma_{\lambda, \vartheta, \vartheta_{1}}^{2} \subset \partial \mathcal{D}, \;
{w}^{\lambda, \vartheta} \geq, \not\equiv 0 \text{ on } \partial \mathcal{D} \setminus \Gamma_{\lambda, \vartheta, \vartheta_{1}}^{2}.
\end{equation*}
Applying the maximum principle from \cite[Lemma 2.5]{Yao26}, we conclude that \({w}^{\lambda, \vartheta} \geq 0\) in \(\mathcal{D}\), and hence throughout the entire domain \({D}_{\lambda, \vartheta, \vartheta_{1}}\). Finally, by the strong maximum principle, it follows that both \eqref{LY305b} and \eqref{LY303a} hold for all \(\lambda \in [\bar{\lambda} - \varepsilon_{2}, \bar{\lambda}]\).
\end{rmks}

\begin{lma} \label{lma302}
Let
\begin{equation} \label{LY323a}
\Phi_{1} = \max\{\frac{1}{2}\Phi_{0}, \; \Psi_{0}\cos\gamma\} \quad \text{and} \quad \Psi_{1} = \max\{\frac{1}{2}\Psi_{0}, \; \Phi_{0}\cos\gamma\}.
\end{equation}
If \(\alpha \in (0, \pi/2)\), then \eqref{LY305a}, \eqref{LY303a} and
\begin{equation} \label{LY322a}
\nabla {u} \cdot \mathbf{e}_{\alpha - \pi/2} < 0 \text{ on } \Gamma_{N}^{ - } \cap {T}_{\lambda, \pi/2}
\end{equation}
hold for \(\lambda \geq \Phi_{1}\) and \(\vartheta = \pi/2\).
If \(\beta \in (0, \pi/2)\), then \eqref{LY303b} and
\begin{equation} \label{LY322b}
\nabla {u} \cdot \mathbf{e}_{\pi/2 - \beta} < 0 \text{ on } \Gamma_{N}^{ + } \cap \hat{T}_{\lambda, \pi/2}
\end{equation}
hold for \(\vartheta = \pi/2\) and \(\lambda \geq \Psi_{1}\).
\end{lma}

\begin{proof}
We only prove the case \(\alpha \in (0, \pi/2)\); the case \(\beta \in (0, \pi/2)\) follows analogously.

\textbf{Step 1: Monotonicity of \({u}\) in the interior of \(\Omega\)}.
Let \(\vartheta = \pi/2\) and set \(\vartheta_{1} = 0\) (see the left picture in \autoref{fig31MMP}). By the definition of \({\Phi}_{1}\) in \eqref{LY323a}, for every \({\lambda} \geq {\Phi}_{1}\) the set \(\Gamma_{\lambda, \vartheta, \vartheta_{1}}^{2}\) and its reflection
\((\Gamma_{\lambda, \vartheta, \vartheta_{1}}^{2} )'\) (with respect to \({T}_{\lambda, \vartheta}\)) are contained in \(\Gamma_{N}^{ - }\). Moreover,
\({w}^{\lambda, \vartheta} > 0\) on \(\Gamma_{\lambda, \vartheta, \vartheta_{1}}^{1}\). This implies that the conditions in \eqref{LY316} are fulfilled for \({\lambda} \geq {\Phi}_{1}\). Hence, by applying \autoref{lma301} we conclude that \eqref{LY305a} and \eqref{LY303a} hold for \(\vartheta = \pi/2\) and \({\lambda} \geq {\Phi}_{1}\).

\textbf{Step 2: Monotonicity of \({u}\) along the lower Neumann side}.
Since \({w}^{\lambda, \pi/2} > 0\), Serrin’s boundary point lemma yields the non-vanishing of the tangential derivative of \({u}\) along \(\Gamma_{N}^{ - }\); see also \autoref{lma205BP}. Hence, \eqref{LY322a} holds for every \(\lambda \geq \Phi_{1}\).
\end{proof}

For \(\alpha \in (0, \pi/2)\), it is straightforward to verify that
\begin{equation} \label{LY324a}
\Gamma_{\lambda, \vartheta, 0}^{2B} = \emptyset \quad \text{whenever } \vartheta \in [\pi/2 - \alpha, \alpha_{*}], \; \lambda \geq \Phi_{2},
\end{equation}
where the constants \(\alpha_{*}\) and \(\Phi_{2}\) are given by
\begin{align*}
& \alpha_{*} = \frac{\pi}{2} \text{ for } \alpha \in [\frac{\pi}{4}, \frac{\pi}{2}), \quad
\alpha_{*} = \frac{3\pi}{4} \text{ for } \alpha \in [\frac{\pi}{8}, \frac{\pi}{4}), \quad
\alpha_{*} = \pi - 2\alpha \text{ for } \alpha \in (0, \frac{\pi}{8}),
\\
& \Phi_{2} = \max\Big\{ \tfrac{\sin(\alpha_{*} - \gamma)}{\sin\alpha_{*}}\Psi_{0}, \quad \tfrac{1}{1 + \sin\gamma}\Phi_{0} \Big\}.
\end{align*}
Note that \(\pi - 2\alpha \leq \alpha_{*} < \pi - \alpha\) and \(0 < \Phi_{2} < \Phi_{0}\). Analogously, for \(\beta \in (0, \pi/2)\), we define:
\begin{align*}
& \beta_{*} = \frac{\pi}{2} \text{ for } \beta \in [\frac{\pi}{4}, \frac{\pi}{2}), \quad
\beta_{*} = \frac{3\pi}{4} \text{ for } \beta \in [\frac{\pi}{8}, \frac{\pi}{4}), \quad
\beta_{*} = \pi - 2\beta \text{ for } \beta \in (0, \frac{\pi}{8}),
\\
& \Psi_{2} = \max\Big\{\tfrac{\sin(\beta_{*} - \gamma)}{\sin\beta_{*}}\Phi_{0}, \quad \tfrac{1}{1 + \sin\gamma}\Psi_{0} \Big\}.
\end{align*}

\begin{lma} \label{lma303}
If \(\alpha \in (0, \pi/2)\), then \eqref{LY303a} holds for all \(\vartheta \in [\pi/2 - \alpha, \alpha_{*}]\) and \(\lambda \geq \Phi_{2}\).
Similarly, if \(\beta \in (0, \pi/2)\), then \eqref{LY303b} holds for all \(\vartheta \in [\pi/2 - \beta, \beta_{*}]\) and \(\lambda \geq \Psi_{2}\).
\end{lma}

\begin{proof}
We focus on the case \(\alpha < \pi/2\); the case \(\beta < \pi/2\) follows analogously.

\textbf{Step 1}.
We first claim that \eqref{LY303a} holds for \(\lambda \geq \Phi_{2}\) and \(\vartheta = \alpha_{*}\).
Indeed, by \autoref{lma302} and the definition of \(\alpha_{*}, \Phi_{2}\), the claim is already valid for \(\alpha \in [\pi/4, \pi/2)\). It thus remains to consider \(\alpha < \pi/4\).

\textbf{Step 1.1}: In the case \(\alpha < \pi/4\), \eqref{LY303a} holds for \(\lambda \geq \Phi_{2}\) and \(\vartheta \in [\pi/2, 3\pi/4]\).
Define
\begin{equation*}
{J}_{k}: = \Big\{ \frac{\pi}{2} + \frac{j\pi}{2^{k + 1}}: \; j = 0, 1, \dotsc, 2^{k - 1} \Big\}, \quad {J}_{ \infty}: \; = \bigcup_{k = 1}^{ \infty} {J}_{k}.
\end{equation*}
Let \(\vartheta = 3\pi/4\), with \(\vartheta_1 = \pi/2\), \(\vartheta_3 = \pi\). By \autoref{lma302} and \eqref{LY324a}, the boundary conditions \eqref{LY316} are satisfied for \(\lambda \geq \Phi_2\), and \autoref{lma301} applies. Therefore, \eqref{LY303a} holds for all \(\vartheta \in {J}_1\). Assume \eqref{LY303a} holds for all \(\vartheta \in {J}_{k}\). Take any \(\vartheta \in {J}_{k + 1} \setminus {J}_{k}\), and write
\begin{equation*}
\vartheta = \frac{\pi}{2} + (2j + 1)\frac{\pi}{2^{k + 2}}, \quad \text{with } \vartheta_1 = \frac{\pi}{2} + \frac{j\pi}{2^{k + 1}}, \quad \vartheta_3 = \frac{\pi}{2} + \frac{(j + 1)\pi}{2^{k + 1}}.
\end{equation*}
Then \(\vartheta_1, \vartheta_3 \in {J}_{k}\) and \(\vartheta_3 - \vartheta = \vartheta - \vartheta_1 > 0\). Hence, \eqref{LY316} is satisfied and by \autoref{lma301}, \eqref{LY303a} holds for \(\vartheta \in {J}_{k + 1}\). By induction, \eqref{LY303a} holds for all \(\vartheta \in {J}_{ \infty}\).

Since \({J}_{ \infty}\) is dense in the interval \([\pi/2, 3\pi/4]\) and \(\nabla {u}\) is continuous, it follows that
\begin{equation} \label{LY326}
\nabla {u} \cdot \mathbf{e}_{\vartheta + \alpha} \geq 0 \text{ on } \Omega \cap {T}_{\lambda, \vartheta}
\end{equation}
holds for all \(\vartheta \in [\pi/2, 3\pi/4]\) and \({\lambda} \geq {\Phi}_{2}\). To obtain strict positivity, fix any \(\vartheta \in (\pi/2, 3\pi/4)\), and choose \(\vartheta_{1} \in [\pi/2, 3\pi/4]\), \(\vartheta_{3} \in {J}_{ \infty}\) with \(\vartheta_{3} - \vartheta = \vartheta - \vartheta_{1} > 0\). Then \eqref{LY316} holds and \autoref{lma301} yields strict monotonicity. Hence, \eqref{LY303a} holds for all \(\vartheta \in [\pi/2, 3\pi/4]\), concluding Step 1.1. In particular, this completes Step 1 for the case \(\alpha \in [\pi/8, \pi/4)\).

\textbf{Step 1.2}: In the case \(\alpha \in (0, \pi/8)\), \eqref{LY303a} holds for \(\lambda \geq \Phi_{2}\) and \(\pi/2 \leq \vartheta \leq \pi - 2\alpha\). For this purpose, define for \(k \geq 1\) the sets
\begin{equation*}
\tilde{J}_{k} = \big[\frac{\pi}{2}, \pi - \frac{\pi}{2^{k + 1}}\big] \cap \big[\frac{\pi}{2}, \pi - 2\alpha\big].
\end{equation*}
From Step 1.1, we have already established \eqref{LY303a} for \(\vartheta \in \tilde{J}_{1}\). Assume that \eqref{LY303a} holds for all \(\vartheta \in \tilde{J}_{k}\) for some \(k\ge1\). Let \(\vartheta \in \tilde{J}_{k + 1} \setminus \tilde{J}_{k}\) and set \(\vartheta_{3} = \pi\) and \(\vartheta_{1} = 2\vartheta - \vartheta_{3}\). Then \(\vartheta_{1} \in \tilde{J}_{k}\), and \eqref{LY316} is valid for \(\lambda \geq \Phi_{2}\). By \autoref{lma301} it follows that \eqref{LY305a} and \eqref{LY303a} \(\lambda \geq \Phi_{2}\). By induction, we deduce that \eqref{LY303a} holds for all \(\vartheta \in \cup _{k = 1}^{ \infty}\tilde{J}_{k} = [\pi/2, \pi - 2\alpha]\). This concludes Step 1.

\textbf{Step 2}.
We claim that \eqref{LY303a} holds for \(\lambda \geq \Phi_{2}\) and \(\vartheta \in [\pi/2 - \alpha, \alpha_{*}]\).

\textbf{Step 2.1}: This step holds for \(\vartheta = \alpha_{*}/2\). Since \(\alpha_{*} \geq \pi - 2\alpha\), we have \(\alpha_{*}/2 \geq \pi/2 - \alpha\). For \(\vartheta = \alpha_{*}/2\), set \(\vartheta_{1} = 0\) and \(\vartheta_{3} = \alpha_{*}\). Then, for \({\lambda} \geq {\Phi}_{2}\), by \eqref{LY324a} the function \({w}^{\lambda, \vartheta}\) satisfies \eqref{LY316}. Applying \autoref{lma301}, we conclude that \eqref{LY305a} and \eqref{LY303a} hold for \(\vartheta = \alpha_{*}/2\) and \(\lambda \geq \Phi_{2}\).

\textbf{Step 2.2}: This step holds for \(\vartheta \in [\alpha_{*}/2, \alpha_{*}]\). The argument proceeds analogously to Step 1.1, so we omit further details.

\textbf{Step 2.3}: This step holds for \(\vartheta \in [\pi/2 - \alpha, \alpha_{*}]\). This case is handled similarly to Step 1.2; details are omitted for brevity.
This completes the proof.
\end{proof}

\begin{figure}[h]\centering
\begin{tikzpicture}[scale = 2]
\pgfmathsetmacro\AngleA{60}; \pgfmathsetmacro\AngleB{40}; \pgfmathsetmacro\xab{1.0};
\pgfmathsetmacro\AngleC{180 - \AngleA - \AngleB}; \pgfmathsetmacro\yA{ - \xab*tan(90 - \AngleA)}; \pgfmathsetmacro\yB{\xab*tan(90 - \AngleB)}
\pgfmathsetmacro\PHI{\xab/sin(\AngleA)}; \pgfmathsetmacro\PSI{\xab/sin(\AngleB)};
\pgfmathsetmacro\THETA{90 + 0.8*(90 - \AngleA)};
\pgfmathsetmacro\LAMBDA{\PHI + 0.80*(\PSI*cos(\AngleC) - \PHI)};
\pgfmathsetmacro\xP{\LAMBDA*sin(\AngleA)}; \pgfmathsetmacro\yP{ - \LAMBDA*cos(\AngleA)};
\pgfmathsetmacro\xQ{\xab}; \pgfmathsetmacro\yQ{\yA + (\PHI - \LAMBDA)/cos(\AngleA)};
\pgfmathsetmacro\xPP{\xP + (\PHI - \LAMBDA)*tan(\AngleA)/tan(180 - \THETA)*cos(\AngleA - 90)};
\pgfmathsetmacro\yPP{\yP + (\PHI - \LAMBDA)*tan(\AngleA)/tan(180 - \THETA)*sin(\AngleA - 90)};
\pgfmathsetmacro\THETAa{\THETA*2 - 180 + \AngleA};
\pgfmathsetmacro\xR{\xP + (\PHI - \LAMBDA)*tan(\AngleA)/tan(180 - \THETAa)*cos(\AngleA - 90)};
\pgfmathsetmacro\yR{\yP + (\PHI - \LAMBDA)*tan(\AngleA)/tan(180 - \THETAa)*sin(\AngleA - 90)};
\pgfmathsetmacro\xS{\xab};
\pgfmathsetmacro\yS{\yQ - (\PHI - \LAMBDA)*tan(\AngleA)/sin(180 - \THETAa)};
\fill[gray, yellow, draw = black] (\xPP, \yPP) -- (\xQ, \yQ) -- (\xR, \yR) -- cycle;
\fill[gray, green, draw = black] (\xPP, \yPP) -- (\xQ, \yQ) -- (\xS, \yS) -- cycle;
\draw[] (0, 0) node [left] {${z}_{0}$} -- (\xab, \yA) node [right] {${z}_{1}$} -- (\xab, \yB) node [right] {${z}_{2}$} -- cycle;
\draw[red] ({\xQ + 0.2*(\xQ - \xP)}, {\yQ + 0.2*(\yQ - \yP)}) -- ({\xP - 0.2*(\xQ - \xP)}, {\yP - 0.2*(\yQ - \yP)}) node [below] {${T}_{\phi, \pi/2}$};
\draw[red, very thick] ({\xQ + 0.2*(\xQ - \xPP)}, {\yQ + 0.2*(\yQ - \yPP)}) -- ({\xPP - 0.15*(\xQ - \xPP)}, {\yPP - 0.15*(\yQ - \yPP)}) node [below = -2pt] {${T}_{\Upsilon, \vartheta}$};
\end{tikzpicture} \vspace*{-2ex}
\caption{The moving domain ${D}_{\Upsilon, \vartheta}$ when the slope of ${T}_{\Upsilon, \vartheta}$ $\vartheta$ is positive and large}
\label{fig33theta0}
\end{figure}
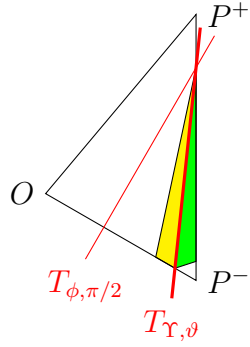

A straightforward computation shows that the three lines \({T}_{\Upsilon, \vartheta}\), \({T}_{\phi, \pi/2}\) and \({T}_{\Phi_{0}, \pi - \alpha} = \{{x}: \; x_{1} = 1\}\) intersect at a common point if and only if
\begin{equation} \label{LY328}
\Upsilon = \Upsilon_{\phi, \vartheta} = \phi + (\Phi_{0} - \phi)\tan\alpha\cot(\pi - \vartheta).
\end{equation}

\begin{lma} \label{lma304}
(1) Let \(\alpha \in (0, \pi/2)\). Suppose that \eqref{LY303a} holds for \(\vartheta = \pi/2\) and \(\lambda \geq \Phi\) for some \(\Phi \in [\max\{\Psi_{0} \cos\gamma, 0\}, \Phi_{0})\). Then \eqref{LY303a} also holds for
\begin{equation*}
\vartheta \in [\pi/2, \pi - \alpha) \quad \text{and} \quad \lambda \geq \Upsilon_{\Phi, \vartheta},
\end{equation*}
where \(\Upsilon_{\phi, \vartheta}\) is defined in \eqref{LY328}.

(2) Let \(\beta \in (0, \pi/2)\). Suppose that \eqref{LY303b} holds for \(\vartheta = \pi/2\) and \(\lambda \geq \Psi\) for some \(\Psi \in [\max\{\Phi_{0} \cos\gamma, 0\}, \Psi_{0})\). Then \eqref{LY303b} also holds for
\begin{equation*}
\vartheta \in [\pi/2, \pi - \beta) \quad \text{and} \quad \lambda \geq \Psi + (\Psi_{0} - \Psi)\tan\beta\cot(\pi-\vartheta).
\end{equation*}
\end{lma}

\begin{proof}
We focus on the first assertion; the second one can be proved in the same way. Note that \({T}_{\lambda, \pi/2} \cap \Gamma_{D} \neq \emptyset\) if and only if \(\lambda \in [\Psi_{0} \cos\gamma, \Phi_{0}]\).  

\textbf{Step 1}.
We first claim that \eqref{LY305a} and \eqref{LY303a} hold for all \((\vartheta, {\lambda})\) satisfying
\begin{equation} \label{LY329a}
\vartheta \in [\bar{\vartheta}, \pi - \alpha ) \quad \text{and} \quad \lambda \geq {\Psi}_{0}\sin\bigl(\vartheta - \gamma\bigr)\csc\vartheta,
\end{equation}
where \(\bar{\vartheta} < \pi - \alpha\) is any fixed number satisfying
\begin{equation} \label{LY329b}
2\bar{\vartheta} - \gamma \geq \pi - \alpha, \quad 2\bar{\vartheta} - \pi \geq \pi/2 - \alpha, \quad \text{and} \quad {\Psi}_{0}\sin\bigl(\bar{\vartheta} - \gamma\bigr)
\csc\bar{\vartheta} \geq {\Phi}_{2}.
\end{equation}
Under the condition \eqref{LY329a} (with the choice \(\vartheta_{3} = \pi\)) the definition \eqref{LY324a} implies that
\begin{equation*}
\Gamma_{\lambda, \vartheta, \vartheta_{1}}^{2B} = \emptyset \quad \text{and} \quad
\Gamma_{\lambda, \vartheta, \vartheta_{1}}^{2A}\subset {T}_{\lambda, 2\vartheta - \pi},
\end{equation*}
with
\begin{equation*}
2\vartheta - \pi \in [\pi/2 - \alpha, \alpha_{*}] \quad \text{and} \quad \lambda \geq {\Phi}_{2}.
\end{equation*}
See \autoref{fig33theta0}. Hence, by applying \autoref{lma303} we deduce that \eqref{LY305a} and \eqref{LY303a} hold under \eqref{LY329a}.

\textbf{Step 2}.
For any fixed \({\bar{\phi}} \in \bigl[{\Phi}, {\Phi}_{0}\bigr)\), let
\({\bar{x}} = ({\bar{x}}_{1}, {\bar{x}}_{2})\) be the unique intersection point of the line \({T}_{{\bar{\phi}}, \pi/2}\) and \(\Gamma_{D}\). We aim to establish that
\begin{equation} \label{LY330}
\nabla {u} \cdot \mathbf{e}_{\vartheta + \alpha} > 0 \text{ on } \Omega \cap \{ {x}: \; ({x} - {\bar{x}}) \cdot \mathbf{e}_{\vartheta + \alpha} = 0 \}
\end{equation}
for every \(\vartheta \in [\pi/2, \bar{\vartheta}]\). To prove this, consider the angular derivative of \({u}\) about \({\bar{x}}\), defined by
\begin{equation} \label{LY331a}
({R}_{\bar{x}}{u})({x}) = ({x}_{1} - \bar{x}_{1}) \partial_{x_{2}}{u}({x}) - ({x}_{2} - \bar{x}_{2}) \partial_{x_{1}}{u}({x}),
\end{equation}
in the domain
\begin{equation*}
\mathcal{D} = \bigl\{{x} \in \Omega: \; ({x} - {\bar{x}}) \cdot \mathbf{e}_{\pi/2 + \alpha} < 0, \quad ({x} - {\bar{x}}) \cdot \mathbf{e}_{\bar{\vartheta} + \alpha} > 0\bigr\}.
\end{equation*}
Then \({R}_{\bar{x}}{u} \in {C}^{1}(\overline{\mathcal{D}})\) satisfies the corresponding linearized equation in the weak sense:
\begin{equation} \label{LY331b}
[\Delta + {f}'({u})] {R}_{\bar{x}}{u} = 0 \text{ in } \mathcal{D} \quad \text{and} \quad {R}_{\bar{x}}{u} \leq, \not \equiv 0 \text{ on } \partial\mathcal{D},
\end{equation}
where \({f}'({u})\) is defined a.e. in \(\Omega\) and belongs to \(L^{\infty}(\Omega)\) (since \({f} \in \mathrm{Lip}_{loc}(\mathbb R)\) and \({u}\) is bounded), and the boundary inequality in \eqref{LY331b} follows from Step~1 and \autoref{lma302}. Furthermore, since
\begin{equation*}
[\Delta + {f}'({u})] (\nabla {u} \cdot \mathbf{e}_{\pi/2 + \alpha}) = 0 \text{ in } \mathcal{D} \quad \text{and} \quad \nabla {u} \cdot \mathbf{e}_{\pi/2 + \alpha} > 0 \text{ in } \overline{\mathcal{D}} \setminus \partial\Omega,
\end{equation*}
we apply the maximum principle (see, e.g., \cite{BNV94}) to \eqref{LY331b} and deduce that \({R}_{\bar{x}}{u} < 0\) in \(\mathcal{D}\), hence establishing \eqref{LY330}.

Combining the results of Step 1 and Step 2, we conclude that \eqref{LY330} holds for any \(\vartheta \in [\pi/2, \pi - \alpha)\) and for any \({\bar{x}} \in \bigl( \Gamma_{D} \cap \{ {x}: \; {x} \cdot \mathbf{e}_{\alpha - \pi/2} \geq {\Phi}\}\bigr)\). In particular, this implies that \eqref{LY303a} holds for
\(\vartheta \in [\pi/2, \pi - \alpha)\) and \(\lambda \geq \Upsilon_{\Phi, \vartheta}\).
This completes the proof.
\end{proof}

\begin{lma} \label{lma305}
Let \(\gamma \in (0, \pi/2]\).

(i)
If \(\alpha \in (0, \pi/2)\), then \eqref{LY305a} and \eqref{LY303a} hold for \(\vartheta = \pi/2\) and all \(\lambda \geq \Psi_{0}\cos\gamma\). Moreover, \eqref{LY303b} holds for \(\vartheta \in (\pi - \beta, \pi/2 + \gamma]\) and \(\lambda \geq \Psi_{0}\).

(ii)
If \(\beta \in (0, \pi/2)\), then \eqref{LY303b} holds for \(\vartheta = \pi/2\) and all \(\lambda \geq \Phi_{0}\cos\gamma\). Moreover, \eqref{LY303a} holds for
\(\vartheta \in (\pi - \alpha, \pi/2 + \gamma]\) and \(\lambda \geq \Phi_{0}\).
\end{lma}

\begin{proof}
We focus on the case \(\alpha < \pi/2\); the case \(\beta < \pi/2\) follows similarly. Suppose first that \(\gamma \leq \alpha\). Then from \eqref{LY323a}, we have \(\Phi_{1} = \Psi_{0}\cos\gamma\). Thus, the conclusion follows from \autoref{lma302} and \autoref{lma304}.
We mainly consider \(\gamma > \alpha\). Then from \eqref{LY323a}, we have \(\Phi_{1} = \Phi_{0}/2\). Define \(\Lambda_{0}: = \Phi_{0}\) and
\begin{equation*}
\Lambda_{j + 1}: \; = \frac{1}{2} \Upsilon_{\Lambda_{j}, \pi - \gamma} = \frac{1}{2} \left( (\Phi_{0} - \Lambda_{j}) \tan\alpha \cot\gamma + \Lambda_{j} \right) \; \text{ for } j \in \mathbb{N}.
\end{equation*}
It can be shown that \(\Lambda_{1} = \Phi_{1}\), and that the sequence \(\{ \Lambda_{j} \}\) is strictly decreasing and converges to
\begin{equation*}
\Lambda_{ \infty} = \Psi_{0} \cos\gamma.
\end{equation*}
By \autoref{lma302} and \autoref{lma304}, \eqref{LY303a} holds for \(\vartheta \in [\pi/2, \pi - \alpha )\) and
\(\lambda \geq \Upsilon_{\Lambda_{1}, \vartheta}\). Now, assume that the claim holds for some \(k\ge1\); that is, \eqref{LY303a} holds for all
\(\vartheta \in [\pi/2, \pi - \alpha )\) and \(\lambda \geq \Upsilon_{\Lambda_{k}, \vartheta}\).
Let
\begin{equation} \label{LY332b}
\vartheta = \pi/2 \; \text{ and } \; \lambda \geq \Lambda_{k + 1},
\end{equation}
Then \(\Gamma_{\lambda, \vartheta, \vartheta_{1}}^{2B}\subset{T}_{2\lambda, \pi - \gamma}\)
with \(\pi - \gamma \in [\pi/2, \pi)\) and \(2\lambda \geq 2\Lambda_{k + 1} = \Upsilon_{\Lambda_{k}, \pi - \gamma}\). Hence, \eqref{LY316} is satisfied.
According to \autoref{lma301}, we conclude that \eqref{LY305a} and \eqref{LY303a} hold under \eqref{LY332b}. Combining this with \autoref{lma304}, \eqref{LY303a} holds for \(\vartheta \in [\pi/2, \pi - \alpha )\) and
\(\lambda \geq \Upsilon_{\Lambda_{k + 1}, \vartheta}\). By mathematical induction, the claim is valid for every \(j\), implying that \eqref{LY303a} holds for \(\vartheta \in [\pi/2, \pi - \alpha)\) and
\(\lambda > \Upsilon_{\Lambda_{ \infty}, \vartheta}\).

Thus, we obtain that \eqref{LY303a} holds for \(\vartheta \in [\pi/2, \pi - \alpha)\) and \(\lambda > \Psi_{0}\csc\vartheta\sin(\vartheta - \gamma)\). It remains to show the case \(\lambda = \Psi_{0}\csc\vartheta\sin(\vartheta - \gamma)\).
By continuity, \eqref{LY326} holds \(\vartheta \in [\pi/2, \pi - \alpha)\) and \(\lambda = \Psi_{0}\csc\vartheta\sin(\vartheta - \gamma)\).
Let \({R}_{\bar{x}}{u}\) be the angular derivative of \({u}\) about \(\bar{x} = (1, \cot\beta)\) (see \eqref{LY331a}), and define
\begin{equation*}
\mathcal{D} = \{ {x} \in \Omega: \;
({x} - \bar{x}) \cdot \mathbf{e}_{\pi/2 + \alpha} < 0\}.
\end{equation*}
Since \({R}_{\bar{x}}{u} \leq, \not \equiv 0\) on \(\mathcal{D}\) and satisfies the linear equation \eqref{LY331b}, the strong maximum principle yields \({R}_{\bar{x}}{u} < 0\) in \(\mathcal{D}\). Hence, \eqref{LY303a} holds for
\begin{equation*}
\vartheta \in (\pi/2, \pi - \alpha) \hspace{1mm}\text{and} \hspace{1mm}\lambda = \Psi_{0}\csc\vartheta\sin(\vartheta - \gamma).
\end{equation*}
In particular, \eqref{LY305a} and \eqref{LY303a} hold for \(\vartheta = \pi/2\) and \(\lambda \geq \Psi_{0}\cos\gamma\); moreover, by \autoref{lma304}, \eqref{LY303b} holds for \(\vartheta \in (\pi - \beta, \pi/2 + \gamma]\) and \(\lambda \geq \Psi_{0}\).
This completes the proof.
\end{proof}

As a direct consequence, the following holds.

\begin{cor} \label{cor306}
Let \(\gamma = \pi - \alpha - \beta = \pi/2\). Then \(\partial_{{x}_{1}}{u} < 0\) in \(\Omega\). Moreover, \eqref{LY303a} and \eqref{LY303b} hold for all \(\lambda > 0\) and \(\vartheta \in [\pi/2, \pi)\).
\end{cor}

\begin{lma} \label{lma307}
Let \(\beta \leq \pi/2\), and let \(\lambda > 0\) be fixed. Suppose that \eqref{LY303a} holds for all \(\vartheta \in [\gamma, \pi/2 + \gamma]\), and that \eqref{LY305a} holds at the endpoints \(\vartheta \in \{\gamma, \pi/2 + \gamma\}\). Then \eqref{LY305a} holds for all \(\vartheta \in [\gamma, \pi/2 + \gamma]\).
\end{lma}

\begin{proof}
Note that the function \({w}^{\lambda, \vartheta}\) may fail to satisfy the Neumann boundary condition on \(\Gamma_{\lambda, \vartheta}^{2B}\) for certain values of \(\vartheta\) (for example, when \(\vartheta\) is slightly less than \(\pi/2 + \gamma\)).
To establish
\begin{equation*} 
{w}^{\lambda, \vartheta}( {x} ) = {u}({x}^{\lambda, \vartheta}) - {u}( {x} ) > 0 \text{ for } {x} \in \overline{ {D}_{\lambda, \vartheta} } \setminus {T}_{\lambda, \vartheta},
\end{equation*}
we follow the approach in \cite[Lemma 13]{YCG21}. From the assumptions of this lemma, we have
\begin{align} \label{LY334a}
{w}^{\lambda, \pi/2 + \gamma} > 0 & \; \text{ in } \overline{ {D}_{\lambda, \pi/2 + \gamma} } \setminus {T}_{\lambda, \pi/2 + \gamma},
\\ \label{LY334b}
{w}^{\lambda, \gamma} > 0 & \; \text{ in } \overline{ {D}_{\lambda, \gamma} } \setminus {T}_{\lambda, \gamma},
\end{align}
and
\begin{equation} \label{LY334c} \normalsize
\text{ \({u}( {x} )\) is increasing as the angle of \(\overrightarrow{{P}_{\lambda}{x}}\) and \(\overrightarrow{{P}_{\lambda}{z}_{0}}\) decreases along each arc \({S}({P}_{\lambda}, {r})\)},
\end{equation}
where \({P}_{\lambda} = (\lambda\sin\alpha, - \lambda\cos\alpha)\) and
\begin{equation*}
{S}({P}_{\lambda}, {r}) = \{ {x} \in \overline{\Omega}: \; |{x} - {P}_{\lambda}| = {r}, \; ({x} - {P}_{\lambda}) \cdot \mathbf{e}_{\pi/2 - \beta} \geq 0 \geq
({x} - {P}_{\lambda}) \cdot \mathbf{e}_{ - \beta} \}.
\end{equation*}
The condition \(\beta \leq \pi/2\) guarantees that
\(\gamma \geq \pi/2 - \alpha\), so that \({S}({P}_{\lambda}, {r})\) is connected and lies on a circular arc. Fix any \(\vartheta \in (0, \pi/2 + \gamma)\) and let \({x}\) be an arbitrary point in
\(\overline{ {D}_{\lambda, \vartheta} } \setminus {T}_{\lambda, \vartheta}\). Then there exist two angles \(\psi_{1}, \psi_{2} \in [0, \pi]\) such that
\begin{equation}
{x} \in {T}_{\lambda, \psi_{1}} \quad \text{and} \quad {y}: = {x}^{\lambda, \vartheta} \in {T}_{\lambda, \psi_{2}},
\end{equation}
where we note that \(\psi_{1} = 2\vartheta - \psi_{2} < \vartheta\). We consider the following four cases.

\textbf{Case 1}:
\(\psi_{2}, \psi_{1} \in [\gamma, \pi/2 + \gamma]\).
By the monotonicity in \eqref{LY334c}, we directly obtain:
\begin{equation*}
{u}( {x} ) < {u}({y}).
\end{equation*}

\textbf{Case 2}:
\(\psi_{2} \in [\gamma, \pi/2 + \gamma]\) and \(\psi_{1}\not \in [\gamma, \pi/2 + \gamma]\).
Reflect \({x}\) with respect to the line \({T}_{\lambda, \gamma}\) to obtain a point \({x}^{\lambda, \gamma}\). Then \({x}^{\lambda, \gamma} \in {T}_{\lambda, \psi_{1}'}\) with \(\psi_{1}' = 2\gamma - \psi_{1}\). Since \(\gamma > \pi/2 - \alpha\), we deduce that \({x}^{\lambda, \gamma} \subset\{{x}_{1} < 1\}\).
Since \(\psi_{1}' - \psi_{2} = 2(\gamma - \vartheta) < 0\) and \(\psi_{2} \leq \pi/2 + \gamma\), we deduce that \({x}^{\lambda, \gamma}\) lies below the line containing \(\Gamma_{N}^{ - }\). Thus, \({x}^{\lambda, \gamma} \in \overline{\Omega}\). Then from \eqref{LY334b} and \eqref{LY334c}, we obtain:
\begin{equation*}
{u}( {x} ) < {u}({x}^{\lambda, \gamma}) \quad \text{and} \quad {u}({x}^{\lambda, \gamma}) < {u}({y}).
\end{equation*}

\textbf{Case 3}:
\(\psi_{2}\not \in [\gamma, \pi/2 + \gamma]\) and \(\psi_{1} \in [\gamma, \pi/2 + \gamma]\).
Reflect \({y}\) with respect to the line \({T}_{\lambda, \pi/2 + \gamma}\) to obtain a point \({y}^{\lambda, \pi/2 + \gamma}\). Then \({y}^{\lambda, \pi/2 + \gamma} \in {T}_{\lambda, \psi_{2}'}\) with \(\psi_{2}' = \pi + 2\gamma - \psi_{2}\). It follows that \({y}^{\lambda, \pi/2 + \gamma}\) lies on or below the line containing \(\Gamma_{N}^{ + }\). Moreover, since \(\psi_{2}' - \psi_{1} = 2(\pi/2 + \gamma - \vartheta) > 0\) and \(\psi_{1} \geq \pi/2 - \alpha\), we deduce that \({y}^{\lambda, \pi/2 + \gamma} \in \overline{\Omega}\). Thus, from \eqref{LY334c} and \eqref{LY334a}, we obtain
\begin{equation*}
{u}( {x} ) < {u}({y}^{\lambda, \pi/2 + \gamma}) \quad \text{and} \quad {u}({y}^{\lambda, \pi/2 + \gamma}) < {u}({y}).
\end{equation*}

\textbf{Case 4}:
\(\psi_{2}, \psi_{1}\not \in [\gamma, \pi/2 + \gamma]\).
In this case, we have \({y}^{\lambda, \pi/2 + \gamma} \in {T}_{\lambda, \psi_{2}'}\) and \({x}^{\lambda, \gamma} \in {T}_{\lambda, \psi_{1}'}\).
Since \(0 \leq \psi_{1} < \psi_{2} \leq \pi\), we deduce that \(\psi_{1} \geq \max\{2\vartheta - \pi, 0\}\) and hence
\(\psi_{2}' - \psi_{1}' = 2(\pi/2 + \psi_{1} - \vartheta) \geq 0\). Thus, one obtains
\begin{equation*}
0 \leq \psi_{1} < \gamma < \psi_{1}' \leq \psi_{2}' < \pi/2 + \gamma < \psi_{2} \leq \pi.
\end{equation*}
It follows that both \({x}^{\lambda, \gamma}\) and \(({x}^{\lambda, \vartheta})^{\lambda, \pi/2 + \gamma}\) belong to \(\overline{\Omega}\). Consequently, from \eqref{LY334b}, \eqref{LY334c} and \eqref{LY334a} we have
\begin{equation*}
{u}( {x} ) < {u}({x}^{\lambda, \gamma}), \quad
{u}({x}^{\lambda, \gamma}) \leq {u}({y}^{\lambda, \pi/2 + \gamma}), \quad \text{and} \quad {u}({y}^{\lambda, \pi/2 + \gamma}) < {u}({y}).
\end{equation*}
In all cases, we have shown that \({u}( {x} ) < {u}({y})\). This concludes the proof.
\end{proof}

\begin{lma} \label{lma308}
Let \(\alpha\), \(\beta\), and \(\gamma\) be acute angles.
Then there exists a small constant \(\varepsilon > 0\) such that
\begin{enumerate}
\item[\rm(i)]
\eqref{LY303a} is valid for \(\vartheta \in [\pi/2 - \alpha, \pi/2 + \gamma]\) and \(\lambda > \Phi_{0} - \varepsilon\),
\item[\rm(ii)]
\eqref{LY303b} is valid for \(\vartheta \in [\pi/2 - \beta, \pi/2 + \gamma]\) and \(\lambda > \Psi_{0} - \varepsilon\).
\end{enumerate}
\end{lma}

\begin{proof}
We focus on the first assertion; the second one can be proved in the same way.

\textbf{Step 1}.
Assume that both \(\alpha\) and \(\beta\) are acute. Then \eqref{LY303a} holds for
\begin{equation*} 
\lambda \geq \Phi_{0} - \varepsilon_{1} \quad \text{and} \quad \pi/2 - \alpha \leq \vartheta \leq \pi - \alpha,
\end{equation*}
where \(\varepsilon_{1}\in(0, \Phi_{0} - \Phi_{2})\) is chosen so that, setting \(\Phi_{3}: = \Phi_{0} - \varepsilon_{1}\) and
\(\Psi_{3}: = \Phi_{3}\sin\alpha\, \csc\beta\), the following hold:
\begin{subequations}\label{LY335cd}
\begin{align}
\label{LY335c}
&\text{\eqref{LY303a} holds for \(\vartheta \in [\tfrac{\pi}{2} - \alpha, \alpha_{**}]\) and \(\lambda \geq \Phi_{3}\), }
\\
\label{LY335d}
&\text{\eqref{LY303b} holds for \(\vartheta \in [\tfrac{\pi}{2} - \beta, \beta_{**}]\) and \(\lambda \geq \Psi_{3}\).}
\end{align}
\end{subequations}
The existence of such an \(\varepsilon_{1}\) follows from \autoref{lma303} and \autoref{lma304}, together with the fixed constants \(\alpha_{**}\) and \(\beta_{**}\) satisfying
\begin{equation*}
\max\{\beta/2 + \gamma, \ \pi - 2\alpha\} < \alpha_{**} < \pi - \alpha,
\quad
\pi - 2\beta < \beta_{**} = \pi - 2\alpha_{**} + 2\gamma < \pi - \beta.
\end{equation*}

For \(\vartheta = \pi - \alpha\) and \(\lambda \geq \Phi_{3}\), we have \(\Gamma_{\lambda, \vartheta}^{2A} \subset {T}_{\lambda, \pi - 2\alpha}\) and \(\Gamma_{\lambda, \vartheta}^{2B} \subset {T}_{\hat{\lambda}, \pi - 2\beta}\),
where \(\hat{\lambda} = \lambda\sin\alpha\csc\beta \geq \Psi_{3}\).
By \eqref{LY335cd}, \(w^{\lambda, \vartheta}\) satisfies the strict boundary condition in \eqref{LY316} on \(\Gamma_{\lambda, \vartheta}^{2}\).
Hence, by \autoref{lma301}, \eqref{LY305a} and \eqref{LY303a} hold for \(\vartheta = \pi - \alpha\) and \(\lambda \geq \Phi_{3}\).

It remains to consider \(\vartheta\in(\alpha_{**}, \pi - \alpha)\) with \(\lambda \geq \Phi_{3}\).
In this case,
\(\Gamma_{\lambda, \vartheta}^{2B}\subset {T}_{\hat{\lambda}, \hat{\vartheta}}\), where
\begin{equation*}
\hat{\lambda} = \hat{\lambda}(\lambda, \vartheta) > \Psi_{3}
\quad \text{and} \quad
\hat{\vartheta} = \pi - 2\vartheta + 2\gamma \in [\pi - 2\beta, \beta_{**}].
\end{equation*}
It then follows from \eqref{LY335d} that \(w^{\lambda, \vartheta}\) satisfies the strict boundary condition in \eqref{LY316} on \(\Gamma_{\lambda, \vartheta}^{2B}\).
Therefore, arguing as in Step 1.1 of \autoref{lma303}, we deduce that \eqref{LY305a} and \eqref{LY303a} hold for \(\vartheta\in(\alpha_{**}, \pi - \alpha)\) and \(\lambda \geq \Phi_{3}\).
This completes Step~1.

\textbf{Step 2}.
We claim that \eqref{LY303a} holds when \(\vartheta = \tilde{\vartheta}_{*}\) and \(\Phi_{0} - \lambda\) is positive and sufficiently small.
Here, we define \(\tilde{\vartheta}_{*} = \tilde{\vartheta}_{\kappa}\), where the integer \(\kappa\) is chosen such that \(\tilde{\vartheta}_{\kappa} \geq 2\gamma > \tilde{\vartheta}_{\kappa - 1}\), i.e.,
\(\tilde{\vartheta}_{\kappa + 1} \geq \gamma + \pi/2 > \tilde{\vartheta}_{\kappa }\), and
\begin{equation*}
\tilde{\vartheta}_{0} = \pi - \alpha, \quad
\tilde{\vartheta}_{j} = (\pi + \tilde{\vartheta}_{j - 1})/2, \; j \in \mathbb{N}^{+}.
\end{equation*}

By Step 1, the claim is valid when \(\tilde{\vartheta}_{*} \leq \pi - \alpha\), that is, when \(2\gamma \leq \pi - \alpha\) or equivalently \(\gamma \leq \beta\).
Now we assume that \(\gamma > \beta\) (hence \(\gamma > \pi/4\)).
Applying \autoref{lma307}, we obtain
\begin{equation} \label{LY337a}
{w}^{\lambda, \vartheta} > 0 \text{ in } \overline{{D}_{\lambda, \vartheta, 2\vartheta - \pi}} \setminus {T}_{\lambda, \vartheta}
\end{equation}
for \(\vartheta \in [\pi - \alpha/2, \pi/2 + \gamma]\) and \(\lambda \geq \Phi_{0}\).
Based on this and the monotonicity \eqref{LY322b} on the upper Neumann boundary \(\Gamma_{N}^{ + }\) (see details in \autoref{lma305} and Step 2 of \autoref{lma302}),
it follows by continuity that there exists \(\varepsilon_{2} \in (0, \varepsilon_{1})\) such that
\begin{equation*} 
{w}^{\lambda, \vartheta} > 0 \text{ on } \Gamma_{\lambda, \vartheta, 2\vartheta - \pi}^{2B} \; \text{ for } \lambda > \Phi_{0} - \varepsilon_{2} \; \text{ and } \; \vartheta \in \{\tilde{\vartheta}_{1}, \tilde{\vartheta}_{2}, \ldots, \tilde{\vartheta}_{\kappa}\}.
\end{equation*}

Now, set \(\vartheta = \tilde{\vartheta}_{1}\) and \(\vartheta_{1} = 2\vartheta - \pi\).
Then \(\vartheta_{1} = \tilde{\vartheta}_{0}\) and
\begin{equation*}
\begin{cases}
\Delta {w}^{\lambda, \vartheta} + {c}^{\lambda, \vartheta} {w}^{\lambda, \vartheta} = 0 & \text{in } {D}_{\lambda, \vartheta},
\\
{w}^{\lambda, \vartheta} \geq 0 & \text{on } \partial {D}_{\lambda, \vartheta} \setminus \Gamma_{\lambda, \vartheta}^{2A},
\\
\partial_{\nu}{w}^{\lambda, \vartheta} \geq, \not \equiv 0 & \text{on } \Gamma_{\lambda, \vartheta}^{2A},
\end{cases}
\end{equation*}
which is similar to \eqref{LY316}. Following an argument similar to that in \autoref{lma301}, we conclude that \eqref{LY337a} and \eqref{LY303a} hold for \(\lambda > \Phi_{0} - \varepsilon_{2}\) with \(\vartheta = \tilde{\vartheta}_{1}\).
By mathematical induction, we conclude that \eqref{LY337a} and \eqref{LY303a} hold for all \(\lambda > \Phi_{0} - \varepsilon_{2}\) and \(\vartheta \in \{\tilde{\vartheta}_{1}, \tilde{\vartheta}_{2}, \ldots, \tilde{\vartheta}_{\kappa}\}\).

\textbf{Step 3}.
We claim that \eqref{LY303a} holds for \(\vartheta \in [\pi - \alpha, \pi/2 + \gamma]\) and \(\lambda > \Phi_{0} - \varepsilon_{2}\).
Indeed, this follows by an argument analogous to Step 2 in the proof of \autoref{lma304}.
For any fixed \(\lambda \in [\Phi_{0} - \varepsilon_{2}, \Phi_{0})\) and \(\bar{x} = (\bar{x}_{1}, \bar{x}_{2})\) with
\(\bar{x}_{1} = \lambda\sin\alpha\), \(\bar{x}_{2} = - \lambda\cos\alpha\), define
\begin{equation*}
\mathcal{D} = \{ {x} \in \Omega: \;
({x} - \bar{x}) \cdot \mathbf{e}_{\tilde{\vartheta}_{*} + \alpha} < 0, \; ({x} - \bar{x}) \cdot \mathbf{e}_{\tilde{\vartheta}_{0} + \alpha} > 0\}.
\end{equation*}
Then the angular derivative \({R}_{\bar{x}}{u}\), introduced in \eqref{LY331a}, belongs to \({C}^{1}(\overline{\mathcal{D}})\) and satisfies
\begin{equation} \label{LY338}
[\Delta + {f}'({u})] {R}_{\bar{x}}{u} = 0 \text{ in } \mathcal{D} \quad \text{and} \quad {R}_{\bar{x}}{u} \geq, \not \equiv 0 \text{ on } \partial\mathcal{D}
\end{equation}
as established in Steps 1 and 2. Recalling that
\begin{equation*}
[\Delta + {f}'({u})] (\nabla {u} \cdot \mathbf{e}_{\tilde{\vartheta}_{*} + \alpha}) = 0 \text{ in } \mathcal{D} \quad \text{and} \quad (\nabla {u} \cdot \mathbf{e}_{\tilde{\vartheta}_{*} + \alpha}) < 0 \text{ in } \overline{\mathcal{D}} \setminus \partial\Omega,
\end{equation*}
we apply the maximum principle (see, e.g., \cite{BNV94}) to conclude that
\({R}_{\bar{x}}{u} > 0\). In particular, this establishes \eqref{LY303a} for \(\vartheta \in [\pi/2 - \alpha, \tilde{\vartheta}_{*}]\) and \(\lambda \geq \Phi_{0} - \varepsilon_{2}\).

Now, let \(\vartheta = \pi/2 + \gamma\), \(\vartheta_{1} = 2\vartheta - \pi = 2\gamma\) with \(\lambda \geq \Phi_{0} - \varepsilon_{2}\). Observe that \(\Gamma_{\lambda, \vartheta}^{2B}\) and its reflection \((\Gamma_{\lambda, \vartheta}^{2B})'\) belong to \(\Gamma_{N}^{ + }\).
By the definition of \(\tilde{\vartheta}_{*}\), we derive \(2\gamma \in [\pi/2 - \alpha, \tilde{\vartheta}_{*}]\), which ensures that \({w}^{\lambda, \vartheta}\) satisfies \eqref{LY316}. Applying \autoref{lma301}, we conclude that both \eqref{LY305a} and \eqref{LY303a} hold for
\(\vartheta = \pi/2 + \gamma\) and \(\lambda \geq \Phi_{0} - \varepsilon_{2}\).

Finally, as argued at the beginning of this step, it follows that
\eqref{LY303a} remains valid for \(\vartheta \in [\pi - \alpha, \pi/2 + \gamma]\) with \(\lambda \geq \Phi_{0} - \varepsilon_{2}\).
This completes the proof.
\end{proof}

\begin{thm} \label{thm309}
Let \({u}\) be a positive solution of the semilinear mixed boundary value problem \eqref{LY101} in a planar triangle \(\Omega\), with the notations as in \autoref{Sec03acu}. Suppose that
\begin{equation*}
\alpha \in (0, \pi/2), \quad \beta \in (0, \pi/2), \quad \gamma \in (0, \pi/2].
\end{equation*}
Then \(\partial_{{x}_{1}}{u} < 0\) in \(\Omega\). Moreover, for every \(\lambda > 0\) and \(\vartheta \in [\gamma, \pi)\), the following inequalities hold:
\begin{gather}
\label{LY333a}
\nabla {u} \cdot \mathbf{e}_{\vartheta + \alpha} > 0 \text{ on } (\Omega \cup \Gamma_{N}^{ - }) \cap {T}_{\lambda, \vartheta},
\\
\label{LY333b}
\nabla {u} \cdot \mathbf{e}_{ - \vartheta - \beta} > 0 \text{ on } (\Omega \cup \Gamma_{N}^{ + }) \cap \hat{T}_{\lambda, \vartheta}.
\end{gather}
\end{thm}

\begin{proof}
The case \(\gamma = \pi/2\) is already established in \autoref{cor306}; we thus assume \(\gamma < \pi/2\). Let \((\bar{t}, \infty)\) denote the largest open interval of positive values \({t}\) for which \eqref{LY303a} holds for \(\vartheta \in [\gamma, \pi/2 + \gamma]\) and \(\lambda \geq {t} \Phi_{0}\), and for which \eqref{LY303b} holds for \(\vartheta \in [\gamma, \pi/2 + \gamma]\) and \(\lambda \geq {t} \Psi_{0}\). By \autoref{lma308}, \(\bar{t}\) is well defined and \(\bar{t} < 1\). Assume by contradiction that \(\bar{t} > 0\), and define \(\bar{\Phi}: = \bar{t}\Phi_{0}\).

\textbf{Step 1}.
We first establish that \eqref{LY303a} holds for \(\vartheta \in [\gamma, \pi/2 + \gamma]\) and \(\lambda \geq \bar{t} \Phi_{0}\), while \eqref{LY303b} remains valid for \(\vartheta \in [\gamma, \pi/2 + \gamma]\) and \(\lambda \geq \bar{t} \Psi_{0}\).
In fact, we set
\begin{equation*}
\bar{x}: = (\bar{x}_{1}, \bar{x}_{2}), \quad \text{where} \quad \bar{x}_{1}: = \bar{\Phi}\sin\alpha, \quad \bar{x}_{2}: = - \bar{\Phi}\cos\alpha,
\end{equation*}
and
\begin{equation*}
\mathcal{D}: = \{ {x} \in \Omega: \; ({x} - \bar{x}) \cdot \mathbf{e}_{\pi/2 - \beta} > 0, \quad ({x} - \bar{x}) \cdot \mathbf{e}_{ - \beta} < 0 \}.
\end{equation*}
By the monotonicity near the Dirichlet boundary (see \autoref{lma203GNN}), the angular derivative \({R}_{\bar{x}}{u}\), defined in \eqref{LY331a}, does not vanish identically in \(\mathcal{D}\). From the definition of \(\bar{t}\) and continuity, we know that the nonstrict inequality \eqref{LY326} holds for \(\vartheta \in [\gamma, \pi/2 + \gamma]\) and \(\lambda \geq \bar{\Phi} = \bar{t} \Phi_{0}\), and hence \({R}_{\bar{x}}{u} \geq 0\) in \(\mathcal{D}\).
Applying the strong maximum principle to the linear elliptic equation satisfied by the nonnegative function \({R}_{\bar{x}}{u}\) (see \eqref{LY338}), we conclude that \({R}_{\bar{x}}{u} > 0\) in \(\mathcal{D}\), implying that \eqref{LY303a} holds strictly for \(\vartheta \in (\gamma, \pi/2 + \gamma)\) and \(\lambda = \bar{\Phi}\).

For \(\vartheta = \gamma\) and \(\lambda \geq \bar{\Phi}\), we observe that \(\Gamma_{\lambda, \gamma}^{2B}\subset {T}_{2\lambda, \gamma}\), and
\begin{equation*}
\Gamma_{\lambda, \gamma}^{2A} \subset {T}_{\lambda, 0}, \quad (\Gamma_{\lambda, \gamma}^{2A})'\subset {T}_{\lambda, 2\gamma} \quad \text{with} \quad 2\gamma \in [\gamma, \pi/2 + \gamma].
\end{equation*}
Thus, \eqref{LY305a} and \eqref{LY303a} hold for \(\vartheta = \gamma\) and \(\lambda \geq \bar{\Phi}\).

Similarly, for \(\vartheta = \pi/2 + \gamma\) and \(\lambda \geq \bar{\Phi}\), we have \(\Gamma_{\lambda, \gamma}^{2B} \subset \Gamma_{N}^{ + }\), and
\begin{equation*}
(\Gamma_{\lambda, \pi/2 + \gamma}^{2A})'\subset {T}_{\lambda, \pi}, \quad \Gamma_{\lambda, \pi/2 + \gamma}^{2A}\subset {T}_{\lambda, 2\gamma} \quad \text{with} \quad 2\gamma \in [\gamma, \pi/2 + \gamma].
\end{equation*}
Thus, \eqref{LY305a} and \eqref{LY303a} hold for \(\vartheta = \pi/2 + \gamma\) and \(\lambda \geq \bar{\Phi}\).
Furthermore, by Step 2 of \autoref{lma302}, \({u}\) is strictly monotone along the upper Neumann boundary \(\Gamma_{N}^{ + }\); in particular, \eqref{LY322b} holds for \(\lambda \geq \bar{\Phi}\cos\gamma\). In summary, we have:
\begin{itemize}
\item
\eqref{LY303a} holds for \(\vartheta \in [\gamma, \pi/2 + \gamma]\) and \(\lambda \geq \bar{\Phi}\), and \eqref{LY322b} is true for \(\lambda \geq \bar{\Phi}\cos\gamma\).
\item
Similarly, \eqref{LY303b} holds for \(\vartheta \in [\gamma, \pi/2 + \gamma]\) and \(\lambda \geq \bar{t} \Psi_{0}\), and \eqref{LY322a} is true for \(\lambda \geq \bar{t} \Psi_{0}\cos\gamma\).
\end{itemize}

\textbf{Step 2}.
We next show that \eqref{LY303a} and \eqref{LY303b} remain valid for \(\vartheta \in [\gamma, \pi/2 + \gamma]\) and \(0 < \bar{t}\Phi_{0} - \lambda\ll1\), similarly, \eqref{LY303b} holds for \(\vartheta \in [\gamma, \pi/2 + \gamma]\) and \(0 < \bar{t}\Psi_{0} - \lambda\ll1\).
Indeed, let \(\bar{\vartheta}\) be the unique angle such that \({T}_{\bar{\Phi}, \bar{\vartheta}}\) passes through the upper mixed boundary point. Since \(\bar{\Phi} < \Phi_{0}\), it follows that \(\bar{\vartheta} \in (\gamma, \pi - \alpha)\). Choose \(\delta > 0\) sufficiently small so that
\begin{equation} \label{LY339a}
0 < \delta < \min\big\{ \tfrac{\bar{\vartheta} - \gamma}{3}, \; \tfrac{\pi/2 + \gamma - \bar{\vartheta}}{5}, \; \tfrac{\pi/2 + 2\gamma + \beta - 2\bar{\vartheta}}{2} \big\} \quad \text{and} \quad \hat{\lambda}(\bar{\Phi}, \bar{\vartheta} + \delta) > \Psi_{2}.
\end{equation}
Then there exists \(\varepsilon_{1} > 0\) small enough so that
\begin{equation} \label{LY339b}
\hat{\lambda}(\lambda, \vartheta) > \Psi_{2} \quad \text{and} \quad \check{\lambda}(\lambda, \vartheta) > \bar{\Phi}
\quad \text{for all } |\vartheta - \bar{\vartheta}| \leq \delta \; \text{ and } \; \lambda \geq \bar{\Phi} - \varepsilon_{1}.
\end{equation}
Note that \({u}\) is strictly monotone along the Neumann boundary (see \eqref{LY322a} for \(\lambda \geq \bar{\Phi}\) and \eqref{LY322b} for \(\lambda \geq \bar{\Phi}\cos\gamma\)) as well as near the Dirichlet boundary (see \autoref{lma203GNN}). By continuity, there exists a small constant \(\varepsilon_{2} \in (0, \varepsilon_{1})\) such that \eqref{LY303a} holds for
\begin{equation*}
\vartheta \in [\gamma, \bar{\vartheta} - \delta] \cup [\bar{\vartheta} + \delta, \pi/2 + \gamma] \quad \text{and} \quad \lambda \geq \bar{\Phi} - \varepsilon_{2}.
\end{equation*}
Now consider the case where \(|\vartheta - \bar{\vartheta}| < \delta\) and \(\lambda \geq \bar{\Phi} - \varepsilon_{2}\). Define \(\vartheta_{1} = \bar{\vartheta} - 3\delta\) and \(\vartheta_{3} = 2\vartheta - \vartheta_{1}\). By \eqref{LY339a}, both \(\vartheta_{1}\) and \(\vartheta_{3}\) belong to \([\gamma, \bar{\vartheta} - \delta] \cup [\bar{\vartheta} + \delta, \pi/2 + \gamma]\),
ensuring that \({w}^{\lambda, \vartheta}\) satisfies the strict boundary condition on \(\Gamma_{\lambda, \vartheta, \vartheta_{1}}^{2A}\).
Furthermore, \(\Gamma_{\lambda, \vartheta}^{2B}\subset \hat{T}_{\hat{\lambda}, \hat{\vartheta}} = {T}_{\check{\lambda}, \check{\vartheta}}\) with
\begin{align*}
& \hat{\lambda} > \Psi_{2} \quad \text{and} \quad \hat{\vartheta} = \pi - 2\vartheta + 2\gamma \in [\pi/2 - \beta, \pi/2] \quad \text{if } \vartheta \geq \pi/4 + \gamma,
\\
& \check{\lambda} > \bar{\Phi} \quad \text{and} \quad \check{\vartheta} = 2\vartheta - \gamma \in [\gamma, \pi/2 + \gamma] \quad \text{if } \vartheta < \pi/4 + \gamma,
\end{align*}
where \eqref{LY339b} is used. Therefore, \({w}^{\lambda, \vartheta}\) satisfies the strict boundary condition on \(\Gamma_{\lambda, \vartheta, \vartheta_{1}}^{2B}\). Consequently, \({w}^{\lambda, \vartheta}\) satisfies \eqref{LY316}. Applying \autoref{lma301}, we deduce that \eqref{LY305b} and \eqref{LY303a} remain valid for all \(\vartheta \in [\bar{\vartheta} - \delta, \bar{\vartheta} + \delta]\), with \(\vartheta_{1} = \bar{\vartheta} - 3\delta\) and for \(\lambda \geq \bar{\Phi} - \varepsilon_{2}\). Thus, \eqref{LY303a} holds for every \(\vartheta \in [\gamma, \pi/2 + \gamma]\) and \(\lambda \geq \bar{\Phi} - \varepsilon_{2}\). A similar argument applies to \eqref{LY303b}, ensuring its validity for all \(\vartheta \in [\gamma, \pi/2 + \gamma]\) when \(0 < \bar{t} \Psi_{0} - \lambda \ll 1\).

This contradicts the definition of \(\bar{t}\), implying that \(\bar{t} = 0\). The proof is complete.
\end{proof}

\begin{thm} \label{thm310}
Let \({u}\) be a positive solution of the semilinear mixed boundary value problem \eqref{LY101} in a planar triangle \(\Omega\), with the notations as in \autoref{Sec03acu}. Suppose that
\begin{equation*}
\max\{\alpha, \beta\} \geq \pi/2.
\end{equation*}
Then \(\partial_{{x}_{1}}{u} < 0\) in \(\Omega\). Furthermore, \({u}\) has no non-vertex critical point, and
\(\nabla {u} \cdot \mathbf{e}_{\theta} < 0\) in \(\Omega \) when \( - \min\{\pi/2, \beta\} \leq \theta \leq \min\{\pi/2, \alpha\}\).
\end{thm}

\begin{proof}
Without loss of generality, we assume that \(\alpha \geq \pi/2\). Set
\begin{equation*}
\bar{\Phi}: = \inf\{\Phi > 0: \; \eqref{LY303a} \text{ holds for every } \vartheta \in (0, \pi/2 + \gamma] \; \text{ and } \; \lambda \geq \Phi\}.
\end{equation*}
We argue by contradiction and suppose that \(\bar{\Phi} > 0\). By \autoref{lma305}, this infimum is well-defined and satisfies \(\bar{\Phi} \leq \Phi_{0}\). Using the same reasoning as in Step 1 of \autoref{thm309}, it follows that \eqref{LY303a} holds for all \(\vartheta \in (0, \pi/2 + \gamma]\) and \(\lambda \geq \bar{\Phi}\). More precisely, for each \(\lambda \geq \bar{\Phi}\), we have:
\begin{equation*} 
{w}^{\lambda, \pi/2 + \gamma} > 0 \text{ in } \overline{ {D}_{\lambda, \pi/2 + \gamma} } \setminus {T}_{\lambda, \pi/2 + \gamma},
\end{equation*}
and
\begin{equation*} 
\text{ \({u}\) is increasing as the angle between \(\overrightarrow{ {P}_{\lambda} {x} }\) and \(\overrightarrow{ {P}_{\lambda} {z}_{0}}\) decreases on each arc \({S}({P}_{\lambda}, {r})\)},
\end{equation*}
where \({S}({P}_{\lambda}, {r})\) is a connected arc of a circle defined by
\begin{equation*}
{S}({P}_{\lambda}, {r}) = \{ {x} \in \overline{\Omega}: \; |{x} - {P}_{\lambda}| = {r}, \; ({x} - {P}_{\lambda}) \cdot \mathbf{e}_{\pi/2 - \beta} \geq 0 \}.
\end{equation*}

\textbf{Step 1}.
For \(\vartheta \in (0, \pi/2 + \gamma]\) and \(\lambda \geq \bar{\Phi}\), the inequality
\begin{equation} \label{LY343}
{w}^{\lambda, \vartheta} > 0 \text{ in } \overline{ {D}_{\lambda, \vartheta} } \setminus {T}_{\lambda, \vartheta}
\end{equation}
holds. Details are omitted as the argument parallels that in \autoref{lma307}.

\textbf{Step 2}.
Both \eqref{LY305a} and \eqref{LY303a} hold at \(\vartheta = \pi/2\) for all \(\lambda \in (\bar{\Phi} - \varepsilon_{1}, \bar{\Phi}]\) for some small constant \(\varepsilon_{1} > 0\). This follows from the same argument as in Step 2 of the proof of \autoref{lma308}.

\textbf{Step 3}.
We now show that \eqref{LY303a} holds for
\begin{equation} \label{LY345}
0 < \vartheta \leq \tilde{\vartheta}_{0} \; \text{ and } \; \lambda \geq \bar{\Phi} - \varepsilon_{2},
\end{equation}
where \(\tilde{\vartheta}_{0} = \max\{\pi/4 + \gamma, \bar{\vartheta} + \delta\}\) and \(\bar{\vartheta}\) is the unique angle such that \({T}_{\bar{\Phi}, \bar{\vartheta}}\) contains the upper mixed boundary point. The constants \(\delta > 0\) and \(\varepsilon_{2} \in (0, \varepsilon_{1})\) are chosen to satisfy:
\begin{equation} \label{LY346}
\bar{\vartheta} + \delta < \pi/4 + \gamma + \beta/2, \quad
\hat{\lambda}(\bar{\Phi} - \varepsilon_{2}, \bar{\vartheta} + \delta) > \Psi_{2},
\quad (\bar{\Phi} - \varepsilon_{2})(1 + \sin\gamma) > \bar{\Phi}.
\end{equation}
Note that the existence of such \(\delta\) and \(\varepsilon_{2}\) is ensured by the fact that \(\bar{\vartheta} \leq \pi - \alpha < \pi/4 + \gamma + \beta/2\).

Step 3.1.
We verify that the strict boundary condition of \({w}^{\lambda, \vartheta}\) holds on \(\Gamma_{\lambda, \vartheta}^{2B}\) under the assumptions in \eqref{LY345}. In fact, we have:
\begin{enumerate}
\item[\rm(1)]
If \(\vartheta \in (0, \gamma/2]\), then \(\Gamma_{\lambda, \vartheta}^{2B} = \emptyset\) is always valid.
\item[\rm(2)]
If \(\vartheta \in (\gamma/2, \pi/4 + \gamma]\), then \(\Gamma_{\lambda, \vartheta}^{2B} = {T}_{\check{\lambda}, \check{\vartheta}}\) with
\begin{equation*}
\check{\vartheta} = 2\vartheta - \gamma \in (0, \pi/2 + \gamma] \; \text{ and } \; \check{\lambda} = \lambda + \frac{\lambda\sin\gamma}{\sin(2\vartheta - \gamma)} > \bar{\Phi}.
\end{equation*}
\item[\rm(3)]
If \(\vartheta \in (\pi/4 + \gamma, \pi/4 + \gamma + \beta/2]\), then \(\Gamma_{\lambda, \vartheta}^{2B}\subset\hat{T}_{\hat{\lambda}, \hat{\vartheta}}\) with
\begin{equation*}
\hat{\vartheta} = \pi - 2\vartheta + 2\gamma \in [\pi/2 - \beta, \pi/2) \; \text{ and } \; \hat{\lambda} = \frac{\lambda\sin\vartheta}{\sin(\vartheta - \gamma)} > \Psi_{2}.
\end{equation*}
\end{enumerate}
By the definition of \(\bar{\Phi}\) and applying \autoref{lma303}, we conclude that the boundary condition of \({w}^{\lambda, \vartheta}\) is satisfied on \(\Gamma_{\lambda, \vartheta}^{2B}\).

Step 3.2.
\eqref{LY343} and \eqref{LY303a} hold for \(\vartheta \in (0, \pi/4]\) and \(\lambda \geq \bar{\Phi} - \varepsilon_{2}\). This follows directly from the argument used in Step 1.1 of \autoref{lma303}.

Step 3.3.
\eqref{LY343} and \eqref{LY303a} to \(\vartheta \in (0, \min\{\pi/2, \tilde{\vartheta}_{0}\}]\) and \(\lambda \geq \bar{\Phi} - \varepsilon_{2}\). This is achieved by the same reasoning as in Step 1.2 of \autoref{lma303}.

Step 3.4.
Finally, \eqref{LY343} and \eqref{LY303a} hold for \(\vartheta \in (0, \tilde{\vartheta}_{0}]\) and \(\lambda \geq \bar{\Phi} - \varepsilon_{2}\). Again, this follows from arguments similar to those in Step 1.2 of \autoref{lma303}.

\textbf{Step 4}.
We show that \eqref{LY303a} holds for \(\vartheta = \tilde{\vartheta}_{*}\) and \(\lambda \geq \bar{\Phi} - \varepsilon_{3}\) for some \(\varepsilon_{3} \in (0, \varepsilon_{2})\).
Here, we define \(\tilde{\vartheta}_{*} = \tilde{\vartheta}_{\kappa}\) where the integer \(\kappa\) is chosen such that \(\tilde{\vartheta}_{\kappa} \geq 2\gamma > \tilde{\vartheta}_{\kappa - 1}\), i.e.,
\(\tilde{\vartheta}_{\kappa + 1} \geq \gamma + \pi/2 > \tilde{\vartheta}_{\kappa }\), and
\begin{equation*}
\tilde{\vartheta}_{0} = \max\{\pi/4 + \gamma, \bar{\vartheta} + \delta\}, \quad
\tilde{\vartheta}_{j} = (\pi + \tilde{\vartheta}_{j - 1})/2, \; j \in \mathbb{N}^{+}.
\end{equation*}
Since the proof follows the same argument as in Step 2 of \autoref{lma308}, the details are omitted.

\textbf{Step 5}.
We conclude that \eqref{LY303a} holds for \(\vartheta \in (0, \pi/2 + \gamma]\) and \(\lambda \geq \bar{\Phi} - \varepsilon_{3}\).
Since this follows the same reasoning as Step 3 in \autoref{lma308}, we omit the details.

By combining these steps, we obtain a contradiction to the definition of \(\bar{\Phi}\). Therefore, we must have \(\bar{\Phi} = 0\). This completes the proof.
\end{proof}

We conclude this section by noting that \autoref{thm103} follows from \autoref{thm309} and \autoref{thm310}.


\section{The symmetry property in an isosceles triangle} \label{Sec04iso}

In this section, we consider the isosceles case and present a refined version of \autoref{thm102}.

\begin{thm} \label{thm401}
Let \({u}\) be a positive solution of \eqref{LY101} in an isosceles triangle \(\Omega\) such that \(\Gamma_{N}\) consists of the two equal sides (i.e., \(\alpha = \beta\)). Then \({u}\) is symmetric with respect to the \({x}_{1}\)-axis. More precisely,
\begin{equation*}
\partial_{x_{1}}{u} < 0 \text{ in } \Omega \cup ( \Gamma_{N}^{ + } \cup \Gamma_{N}^{ - } ), \quad \text{and} \quad x_{2} \partial_{x_{2}}{u} < 0 \text{ in } \Omega \cap \{{x}_{2} \neq 0\}.
\end{equation*}
\end{thm}

\begin{proof}
Define
\begin{equation*}
\bar{\Lambda}: = \inf\{\Lambda > 0: \; \text{\eqref{LY303a} and \eqref{LY303b} hold for } \vartheta \in [\gamma/2, \pi/2 + \gamma/2], \;
\lambda > \Lambda\}.
\end{equation*}
By Step 1 of \autoref{lma308}, this infimum is well defined and satisfies \(\bar{\Lambda} < \Phi_{0}\). To prove the theorem, it suffices to show that \(\bar{\Lambda} = 0\). We argue by contradiction and assume that \(\bar{\Lambda} > 0\).
By continuity, for all \(\vartheta \in [\gamma/2, \pi/2 + \gamma/2]\) and \(\lambda \geq \bar{\Lambda}\), we have
\begin{equation*}
\nabla {u} \cdot \mathbf{e}_{\vartheta + \alpha} \geq 0 \text{ on } \Omega \cap {T}_{\lambda, \vartheta}, \quad \text{and} \quad
\nabla {u} \cdot \mathbf{e}_{ - \vartheta - \beta} \geq 0 \text{ on } \Omega \cap \hat{T}_{\lambda, \vartheta}.
\end{equation*}
Moreover,
\begin{equation*}
{w}^{\lambda, \gamma/2} > 0 \text{ in } {D}_{\lambda, \gamma/2} \quad \text{and} \quad
{w}^{\lambda, \pi/2 + \gamma/2} > 0 \text{ in } {D}_{\lambda, \pi/2 + \gamma/2},
\end{equation*}
Thus, \eqref{LY303a} and \eqref{LY303b} hold for \(\vartheta \in [\gamma/2, \pi/2 + \gamma/2]\) and \(\lambda \geq \bar{\Lambda}\). By \autoref{lma307}, we obtain
\begin{equation*}
{w}^{\lambda, \pi/2} > 0 \text{ in } \overline{{D}_{\lambda, \pi/2}} \setminus {T}_{\lambda, \pi/2} \quad \text{for all } \lambda \geq \bar{\Lambda}.
\end{equation*}

Let \(\bar{\vartheta}\) be the angle such that \({T}_{\bar{\Lambda}, \bar{\vartheta}}\) passes through the upper mixed boundary point. Since \(\bar{\Lambda} < \Phi_{0}\), we have \(\bar{\vartheta} \in (\gamma, \pi - \alpha)\). Choose a small constant \(\delta > 0\) such that
\begin{equation} \label{LY402a}
0 < \delta < \min\big\{ \tfrac{\bar{\vartheta} - \gamma}{3}, \; \tfrac{\pi + \gamma - 2\bar{\vartheta}}{10}, \; \tfrac{2\pi + 3\gamma - 4\bar{\vartheta}}{4} \big\}.
\end{equation}
We then choose \(\varepsilon_{1} > 0\) small enough so that
\begin{equation} \label{LY402b}
\hat{\lambda}(\lambda, \vartheta) > \bar{\Lambda} \; \text{ and } \; \check{\lambda}(\lambda, \vartheta) > \bar{\Lambda} \;
\text{ for } |\vartheta - \bar{\vartheta}| \leq \delta, \; \lambda \geq \bar{\Lambda} - \varepsilon_{1}.
\end{equation}
By continuity, there exists a small constant \(\varepsilon_{2} \in (0, \varepsilon_{1})\) such that \eqref{LY303a} holds for
\begin{equation*}
\vartheta \in [\gamma, \bar{\vartheta} - \delta] \cup [\bar{\vartheta} + \delta, \pi/2 + \gamma/2] \; \text{ and } \; \lambda \geq \bar{\Lambda} - \varepsilon_{2}.
\end{equation*}
Furthermore, applying the same reasoning as in Step 2 of \autoref{lma303}, we conclude that \eqref{LY303a} also holds for \(\vartheta \in [\gamma/2, \gamma]\) and \(\lambda \geq \bar{\Lambda} - \varepsilon_{2}\).
Now for \(|\vartheta - \bar{\vartheta}| < \delta\) and \(\lambda \geq \bar{\Lambda} - \varepsilon_{2}\), take \(\vartheta_{1} = \bar{\vartheta} - 3\delta\) and \(\vartheta_{3} = 2\vartheta - \vartheta_{1}\). The condition \eqref{LY402a} ensures that
\begin{equation*}
\vartheta_{1} \in [\gamma, \bar{\vartheta} - \delta] \quad \text{and} \quad \vartheta_{3} \in [\bar{\vartheta} + \delta, \pi/2 + \gamma/2].
\end{equation*}
From \eqref{LY402b}, we find that \(\Gamma_{\lambda, \vartheta}^{2B}\subset \hat{T}_{\hat{\lambda}, \hat{\vartheta}} = {T}_{\check{\lambda}, \check{\vartheta}}\) with
\begin{align*}
\hat{\lambda} > \bar{\Lambda} & \; \text{ and } \; \hat{\vartheta} = \pi - 2\vartheta + 2\gamma \in [\gamma/2, \pi/2 + \gamma/2] \text{ if } \vartheta \geq \pi/4 + 3\gamma/4,
\\
\check{\lambda} > \bar{\Lambda} & \; \text{ and } \; \check{\vartheta} = 2\vartheta - \gamma \in [\gamma/2, \pi/2 + \gamma/2] \text{ if } \vartheta < \pi/4 + 3\gamma/4.
\end{align*}
Hence, \({w}^{\lambda, \vartheta}\) satisfies the strict boundary conditions on both \(\Gamma_{\lambda, \vartheta, \vartheta_{1}}^{2A}\) and \(\Gamma_{\lambda, \vartheta, \vartheta_{1}}^{2B}\), and satisfies \eqref{LY316}. Applying \autoref{lma301}, we conclude that \eqref{LY305b} and \eqref{LY303a} hold for all \(\lambda \geq \bar{\Lambda} - \varepsilon_{2}\), \(\vartheta \in (\bar{\vartheta} - \delta, \bar{\vartheta} + \delta)\) and \(\vartheta_{1} = \bar{\vartheta} - 3\delta\).

In summary, \eqref{LY303a} holds for \(\vartheta \in [\gamma/2, \pi/2 + \gamma/2]\) and
\(\lambda \geq \bar{\Lambda} - \varepsilon_{2}\).
Similarly, \eqref{LY303b} remains valid for \(\vartheta \in [\gamma/2, \pi/2 + \gamma/2]\) and \(\bar{\Lambda} - \lambda\ll1\). This contradicts the minimality of \(\bar{\Lambda}\), hence \(\bar{\Lambda} = 0\). Consequently, both \eqref{LY303a} and \eqref{LY303b} hold for \(\vartheta \in [\gamma/2, \pi/2 + \gamma/2]\) and
\(\lambda > 0\).

To conclude the symmetry, observe that
\begin{equation*}
{w}^{\lambda, \gamma/2}( {x} ) = {u}({x}_{1}, - 2\lambda\sin(\gamma/2) - x_{2}) - {u}({x}_{1}, x_{2}) > 0 \text{ for } {x} \in \Omega \text{ with } x_{2} < - \lambda\sin(\gamma/2).
\end{equation*}
Similarly, one can deduce that
\begin{equation*}
{u}({x}_{1}, 2\lambda\sin(\gamma/2) - x_{2}) - {u}({x}_{1}, x_{2}) > 0 \text{ for } {x} \in \Omega \text{ with } x_{2} > \lambda\sin(\gamma/2).
\end{equation*}
Taking the limit as \(\lambda \to 0^{ + }\), we obtain
\begin{equation*}
{u}({x}_{1}, - x_{2}) - {u}({x}_{1}, x_{2}) = 0 \; \text{ for } {x} \in \Omega,
\end{equation*}
which proves the symmetry of \({u}\) with respect to the \(x_{1}\)-axis.
\end{proof}


\section{The proof of monotonicity for an obtuse Neumann vertex} \label{Sec05obt}

In this section, we establish the monotonicity property of solutions when the two Neumann boundary components meet at an obtuse angle.

\begin{lma} \label{lma501}
Let \(\gamma > \pi/2\) and
\begin{equation*} 
\max\{\alpha, \beta\} \geq \pi/4.
\end{equation*}
Then the solution \({u}\) is strictly decreasing in the inward normal direction to the longer Neumann side, and \({u}\) has a nonzero tangential derivative in the interior of the shorter Neumann boundary.
\end{lma}

\begin{proof}
Without loss of generality, we assume that \(\alpha \geq \beta\), which implies
\begin{equation} \label{LY501b}
\alpha \geq \pi/4.
\end{equation}

\textbf{Part 1}.
We claim that \eqref{LY303a} holds for \(\vartheta \in [\pi/2 - \alpha, \pi/2]\) and \(\lambda \geq 0\).
To prove this, we argue by contradiction and assume that \(\bar{\Phi} > 0\), where we set
\begin{equation*} 
\bar{\Phi}: = \inf \{\Phi > 0: \; \eqref{LY303a} \text{ holds for } \vartheta \in [\pi/2 - \alpha, \pi/2] \text{ and } \lambda > \Phi\}.
\end{equation*}
By Step 1 of \autoref{lma308}, this infimum is well-defined and satisfies \(\bar{\Phi} < \Phi_{0}\). For \(\vartheta = \pi/2\) and \(\lambda \geq \bar{\Phi}/2\), we have \(\Gamma_{\lambda, \pi/2}^{2B}\subset {T}_{2\lambda, \pi - \gamma}\) and \(\pi - \gamma \in [\pi/2 - \alpha, \pi/2]\). Therefore, \({w}^{\lambda, \pi/2}\) satisfies the boundary condition \eqref{LY316}. By \autoref{lma301}, both \eqref{LY305a} and \eqref{LY303a} hold for \(\vartheta = \pi/2\) and \(\lambda \geq \bar{\Phi}/2\).
Moreover, the condition \eqref{LY501b} implies that \(\pi/4 \geq \pi/2 - \alpha\). We can apply the arguments from Step 1.1 and Step 2 of \autoref{lma303} to deduce that \eqref{LY303a} with \(\lambda \geq \bar{\Phi}/2\) holds for \(\vartheta \in [\pi/4, \pi/2]\) and for \(\vartheta \in [\pi/2 - \alpha, \pi/4]\). This contradicts the definition of \(\bar{\Phi}\). Hence, \(\bar{\Phi} = 0\), and \eqref{LY303a} holds for all \(\vartheta \in [\pi/2 - \alpha, \pi/2]\) and \(\lambda \geq 0\). This completes Part 1.

\begin{figure}[h]\centering
\begin{tikzpicture}[scale = 1.6]
\pgfmathsetmacro\AngleA{48}; \pgfmathsetmacro\AngleB{33}; \pgfmathsetmacro\xab{1.0};
\pgfmathsetmacro\AngleC{180 - \AngleA - \AngleB};
\pgfmathsetmacro\xA{\xab}; \pgfmathsetmacro\yA{ - \xab*tan(90 - \AngleA)};
\pgfmathsetmacro\xB{\xab}; \pgfmathsetmacro\yB{\xab*tan(90 - \AngleB)};
\pgfmathsetmacro\PHI{\xab/sin(\AngleA)}; \pgfmathsetmacro\PSI{\xab/sin(\AngleB)};
\pgfmathsetmacro\xBB{\PSI*cos(90 - \AngleB - 2*\AngleC)}; \pgfmathsetmacro\yBB{\PSI*sin(90 - \AngleB - 2*\AngleC)};
\pgfmathsetmacro\THETA{\AngleC}; \pgfmathsetmacro\LAMBDA{\PHI*0.6};
\pgfmathsetmacro\xP{\LAMBDA*sin(\AngleA)}; \pgfmathsetmacro\yP{ - \LAMBDA*cos(\AngleA)};
\pgfmathsetmacro\xQ{\xab}; \pgfmathsetmacro\yQ{\yA + (\PHI - \LAMBDA)/cos(\AngleA)};
\pgfmathsetmacro\llm{(\PHI - \LAMBDA)*sin(\THETA)/sin(\THETA + \AngleA)};
\pgfmathsetmacro\xM{\xab}; \pgfmathsetmacro\yM{\yA + \llm};
\pgfmathsetmacro\xAA{\xM + \llm*cos(2*\THETA + 2*\AngleA - 90)};
\pgfmathsetmacro\yAA{\yM + \llm*sin(2*\THETA + 2*\AngleA - 90)};
\pgfmathsetmacro\lln{(\PHI - \LAMBDA)*sin(\THETA)/sin(\THETA - \AngleA)};
\pgfmathsetmacro\xN{\xA + \lln*cos(90 + 2*\AngleA)}; \pgfmathsetmacro\yN{\yA + \lln*sin(90 + 2*\AngleA)};
\fill[gray, yellow, draw = black] (\xM, \yM) -- (\xN, \yN) -- (\xAA, \yAA) -- cycle;
\fill[gray, green, draw = black] (\xM, \yM) -- (\xN, \yN) -- (\xA, \yA) -- cycle;
\draw[dashed] (\xab, \yB) -- ( - \xab*0.85, - \yB*0.85) node [right] {\small $\hat{T}_{0, 0}$};
\draw[thick] (0, 0) node [above left] {${z}_{0}$} -- (\xab, \yA) node [right] {${z}_{1}$} -- (\xab, \yB) node [below right] {${z}_{2}$} -- cycle;
\draw[] (0, 0) -- (\xBB, \yBB) -- (\xab, \yA);
\draw[red, very thick] ({\xM + 0.2*(\xM - \xN)}, {\yM + 0.2*(\yM - \yN)}) -- ({\xM - 1.25*(\xM - \xN)}, {\yM - 1.25*(\yM - \yN)}) node [right] {\small ${T}_{\lambda, \gamma}$};
\end{tikzpicture}
\hspace*{2em}
\begin{tikzpicture}[scale = 1.6]
\pgfmathsetmacro\AngleA{48}; \pgfmathsetmacro\AngleB{33}; \pgfmathsetmacro\xab{1.0};
\pgfmathsetmacro\AngleC{180 - \AngleA - \AngleB};
\pgfmathsetmacro\yA{ - \xab*tan(90 - \AngleA)}; \pgfmathsetmacro\yB{\xab*tan(90 - \AngleB)};
\pgfmathsetmacro\xA{\xab}; \pgfmathsetmacro\xB{\xab};
\pgfmathsetmacro\PHI{\xab/sin(\AngleA)}; \pgfmathsetmacro\PSI{\xab/sin(\AngleB)};
\pgfmathsetmacro\xBB{\PSI*cos(90 - \AngleB - 2*\AngleC)}; \pgfmathsetmacro\yBB{\PSI*sin(90 - \AngleB - 2*\AngleC)};
\pgfmathsetmacro\THETA{\AngleC}; \pgfmathsetmacro\LAMBDA{\PHI*0.35};
\pgfmathsetmacro\xP{\LAMBDA*sin(\AngleA)}; \pgfmathsetmacro\yP{ - \LAMBDA*cos(\AngleA)};
\pgfmathsetmacro\xQ{\xab}; \pgfmathsetmacro\yQ{\yA + (\PHI - \LAMBDA)/cos(\AngleA)};
\pgfmathsetmacro\llm{(\PHI - \LAMBDA)*sin(\THETA)/sin(\THETA + \AngleA)};
\pgfmathsetmacro\xM{\xab}; \pgfmathsetmacro\yM{\yA + \llm};
\pgfmathsetmacro\xAA{\xM + \llm*cos(2*\THETA + 2*\AngleA - 90)};
\pgfmathsetmacro\yAA{\yM + \llm*sin(2*\THETA + 2*\AngleA - 90)};
\pgfmathsetmacro\lln{(\PHI - \LAMBDA)*sin(\THETA)/sin(\THETA - \AngleA)};
\pgfmathsetmacro\xN{\xA + \lln*cos(90 + 2*\AngleA)}; \pgfmathsetmacro\yN{\yA + \lln*sin(90 + 2*\AngleA)};
\pgfmathsetmacro\llr{(\xab*tan(90 - \AngleA) + \xab*tan(90 - \AngleB) - \llm)*sin(\AngleB)/sin(360 - 2*\THETA - 2*\AngleA - \AngleB)};
\pgfmathsetmacro\xR{\xab}; \pgfmathsetmacro\yR{\yM - \llr};
\pgfmathsetmacro\xRR{\xM + \llr*cos(2*\THETA + 2*\AngleA - 90)}; \pgfmathsetmacro\yRR{\xRR*tan(90 - \AngleB)};
\pgfmathsetmacro\lls{(\xab/sin(\AngleA)*sin(\AngleC)/sin(\AngleC - \AngleA) - \lln)*sin(\AngleC - \AngleA)/sin(2*\THETA - \AngleC - \AngleA)};
\pgfmathsetmacro\xS{\xN + \lls*cos(2*\AngleA - 90)}; \pgfmathsetmacro\yS{\yN + \lls*sin(2*\AngleA - 90)};
\pgfmathsetmacro\xSS{\xN + \lls*cos(2*\THETA - 90)}; \pgfmathsetmacro\ySS{\xSS*tan(90 - \AngleB)};
\fill[gray, yellow, draw = black] (\xM, \yM) -- (\xN, \yN) -- (\xSS, \ySS) -- (\xRR, \yRR) -- cycle;
\fill[gray, green, draw = black] (\xM, \yM) -- (\xN, \yN) -- (\xS, \yS) -- (\xR, \yR) -- cycle;
\draw[dashed] (\xab, \yB) -- ( - \xab*0.85, - \yB*0.85) node [right] {\small $\hat{T}_{0, 0}$};
\draw[thick] (0, 0) node [above left] {${z}_{0}$} -- (\xab, \yA) node [right] {${z}_{1}$} -- (\xab, \yB) node [below right] {${z}_{2}$} -- cycle;
\draw[] (0, 0) -- (\xBB, \yBB) -- (\xab, \yA);
\draw[red, very thick] ({\xM + 0.2*(\xM - \xN)}, {\yM + 0.2*(\yM - \yN)}) -- ({\xM - 1.16*(\xM - \xN)}, {\yM - 1.16*(\yM - \yN)}) node [right] {\small ${T}_{\lambda, \gamma}$};
\end{tikzpicture} \vspace*{ - 2ex}
\caption{The case for $\gamma > \pi/2$ and $\alpha \geq \pi/4$}
\label{fig37gamma}
\end{figure}
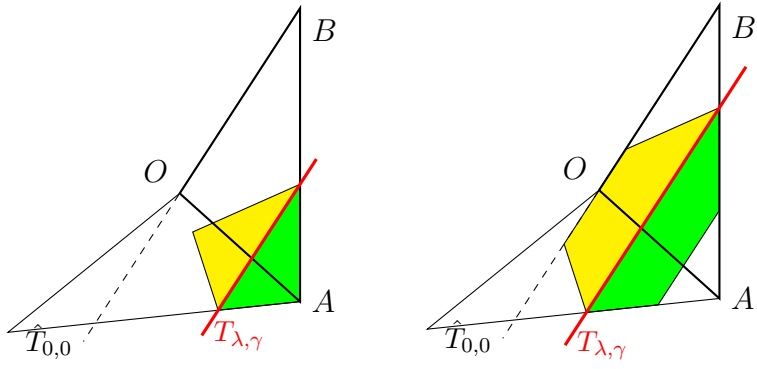

\textbf{Part 2}.
We now claim that \eqref{LY303a} holds for \(\vartheta = \gamma\) and \(\lambda > 0\).
To prove this, define \(\tilde{u}\) by reflecting \({u}\) across the line \({T}_{0, 0}\), which contains the lower Neumann boundary \(\Gamma_{N}^{ - }\), as follows:
\begin{equation} \label{LY504}
\tilde{u}( {x} ) =
\begin{cases}
{u}( {x} ), & \text{ if } {x} \in \overline{\Omega},
\\
{u}({x}^{0, 0}), & \text{ if }{x}^{0, 0} \in \overline{\Omega},
\end{cases}
\end{equation}
where \({x}^{0, 0}\) denotes the reflection of \({x}\) with respect to \(\Gamma_{N}^{ - }\). Then \(\tilde{u}\) is a positive solution of \eqref{LY101} in the doubled domain \(\tilde{\Omega}: = \Omega \cup \Omega' \cup \Gamma_{N}^{ - }\), where \(\Omega'\) is the reflection of \(\Omega\) with respect to \(\Gamma_{N}^{ - }\); see \autoref{fig37gamma}.
From the assumptions \(\gamma > \pi/2\) and \(\alpha \geq \pi/4\), it follows that \(\pi - \gamma \in [\pi/2 - \alpha, \pi/2]\). Thus, by Part 1, we obtain that \(\nabla {u} \cdot \mathbf{e}_{\alpha - \gamma} < 0\) on \(\Omega \cap {T}_{0, \pi - \gamma}\), that is,
\begin{equation*}
\nabla \tilde{u} \cdot \mathbf{e}_{ - \beta} < 0 \text{ on } \tilde{\Omega} \cap \hat{T}_{0, 0},
\end{equation*}
where \(\hat{T}_{0, 0}\) is the line containing \(\Gamma_{N}^{ + }\). Consequently, \(\tilde{u}\) satisfies
\begin{equation*}
\begin{cases}
\Delta \tilde{u} + {f}(\tilde{u}) = 0 & \text{ in } \tilde{\Omega} \cap \{{x} \cdot \mathbf{e}_{ - \beta} > 0\},
\\
\tilde{u} = 0 & \text{ on } \partial\tilde{\Omega} \cap \{{x} \cdot \mathbf{e}_{ - \beta} > 0\},
\\
\nabla\tilde{u} \cdot \mathbf{e}_{ - \beta} \leq 0 & \text{ on } \overline{\tilde{\Omega}} \cap \{{x} \cdot \mathbf{e}_{ - \beta} = 0\}.
\end{cases}
\end{equation*}
Applying the moving plane method in the direction \(\mathbf{e}_{ - \beta}\), it is easy to derive that
\begin{gather*}
\tilde{u}({x}^{\lambda, \gamma}) - \tilde{u}( {x} ) > 0 \text{ for } {x} \in \tilde{\Omega} \text{ with } \lambda\sin\gamma < {x} \cdot \mathbf{e}_{ - \beta} \leq 2\lambda\sin\gamma,
\\
\nabla \tilde{u} \cdot \mathbf{e}_{ - \beta} < 0 \text{ on } \tilde{\Omega} \cap {T}_{\lambda, \gamma}
\end{gather*}
for every \(\lambda > 0\). In particular, this implies that \(\nabla {u} \cdot \mathbf{e}_{ - \beta} < 0 \) in \(\Omega\). Therefore, \eqref{LY303a} holds for all \(\vartheta \in [\pi/2 - \alpha, \gamma]\) and \(\lambda > 0\), and \eqref{LY303b} holds for all \(\vartheta \in [\pi - \beta, \pi]\) and \(\lambda \geq \Psi_{0}\). This completes the proof.
\end{proof}

\begin{lma} \label{lma502}
Let \(\gamma > \pi/2\) and
\begin{equation*} 
\gamma - \pi/2 \geq \min\{\alpha, \beta\}.
\end{equation*}
Then the solution \({u}\) is strictly decreasing in the inward normal direction to the longer Neumann side, and \({u}\) has a nonzero tangential derivative in the interior of the shorter Neumann boundary.

In particular, the conclusion holds if \(\gamma \geq 2\pi/3\).
\end{lma}

\begin{proof}
The isosceles case (i.e., \(\alpha = \beta\)) has been established in \autoref{thm401}. Without loss of generality, we assume that \(\alpha > \beta\), which implies
\begin{equation} \label{LY506b}
\gamma \geq \pi/2 + \beta.
\end{equation}
To prove the result, we define
\begin{equation*} 
\bar{\Phi}: = \inf \{\Phi > 0: \; \text{ \eqref{LY303a} holds for \(\vartheta \in [\pi/2 - \alpha, \gamma]\) and \(\lambda > \Phi\)} \}.
\end{equation*}
By Step 1 of \autoref{lma308}, this infimum is well-defined and satisfies \(\bar{\Phi} < \Phi_{0}\). Suppose, for contradiction, that \(\bar{\Phi} > 0\).

\textbf{Part 1}.
We first show that \eqref{LY303a} holds for all \(\vartheta \in [\pi/2 - \alpha, \gamma]\) and \(\lambda \geq \bar{\Phi}\).
In fact, by continuity and the strong maximum principle applied to the angular derivative of \({u}\), we obtain \eqref{LY303a} for all \(\vartheta \in (\pi/2 - \alpha, \gamma)\) and \(\lambda \geq \bar{\Phi}\); see Step 1 of \autoref{thm309}. For the endpoint \(\vartheta = \pi/2 - \alpha\), we take \(\vartheta_{1} = 0\), \(\vartheta_{3} = \pi - 2\alpha\). Note that \eqref{LY506b} implies that \(\vartheta_{3} \in [\pi/2 - \alpha, \gamma]\) and \(\Gamma_{\lambda, \vartheta}^{2B} = \emptyset\). Since \(\lambda \geq \bar{\Phi}\), we have \({w}^{\lambda, \vartheta}\) satisfies \eqref{LY316}. Hence, both \eqref{LY305a} and \eqref{LY303a} hold for \(\vartheta = \pi/2 - \alpha\) and \(\lambda \geq \bar{\Phi}\).
For the other endpoint \(\vartheta = \gamma\), we take \(\vartheta_{3} = \pi\), \(\vartheta_{1} = 2\gamma - \pi\), so that \(\Gamma_{\lambda, \gamma}^{2B}\subset {T}_{2\lambda, \gamma}\) and \(\Gamma_{\lambda, \gamma}^{2A}\subset {T}_{\lambda, 2\gamma - \pi}\). By \eqref{LY506b}, \(2\gamma - \pi \in [\pi/2 - \alpha, \gamma]\), and so \({w}^{\lambda, \vartheta}\) satisfies \eqref{LY316}. Thus, \eqref{LY305a} and \eqref{LY303a} hold for \(\vartheta = \gamma\) and \(\lambda \geq \bar{\Phi}\).

\textbf{Part 2}.
We prove that for \(\lambda \geq \bar{\Phi}\),
\begin{equation} \label{LY508}
\nabla {u} \cdot \mathbf{e}_{\alpha - \pi/2} < 0 \text{ on } \Gamma_{N}^{ - } \cap {T}_{\lambda, \pi/2}.
\end{equation}
Note that it is difficult to derive the positivity of the function \({w}^{\lambda, \pi/2}\), the argument employed in Step 2 of \autoref{lma302} is no longer applicable. We therefore adopt a different strategy based on the local asymptotic analysis of the angular derivative developed in \cite[Lemma 3.7]{Yao26}. The required details are presented below.
Let \(\lambda \in [\bar{\Phi}, \Phi_{0})\), and define
\begin{equation*}
\bar{x}: = (\bar{x}_{1}, \bar{x}_{2}), \quad \text{with } \; \bar{x}_{1} = \lambda\sin\alpha \text{ and } \bar{x}_{2} = - \lambda\cos\alpha.
\end{equation*}
The angular derivative \({R}_{\bar{x}}{u}\) about \(\bar{x}\), defined in \eqref{LY331a}, satisfies
\begin{equation*}
\Delta {R}_{\bar{x}}{u} + {f}'({u}) {R}_{\bar{x}}{u} = 0
\text{ in } \Omega, \quad \text{and} \quad {R}_{\bar{x}}{u} = 0 \text{ on } \Gamma_{N}^{ - }.
\end{equation*}
Let \(({r}, \vartheta)\) denote polar coordinates centered at \(\bar{x}\), where \(\vartheta\) is the polar angle measured from the ray \(\bar{x}{z}_{1}\) to the ray \(\bar{x}{x}\). From \cite{HW53, HHHO99}, there exist an integer \({l} \geq 1\) and a constant \({C}_{0} \in \R \setminus \{0\}\), such that
\begin{equation} \label{LY509a}
({R}_{\bar{x}}{u})({r}, \vartheta) = {C}_{0}{r}^{l}\sin(l\vartheta) + O({r}^{l + 1}),
\end{equation}
with \({R}_{\bar{x}}{u}\) has exactly \({l} - 1\) nodal domains near \(\bar{x}\) in \(\Omega\). Here, the term \(O({r}^{l + 1})/{r}^{l + 1}\) is bounded as \({r}\to0\).
By Part 1, we have that
\begin{equation*} 
({R}_{\bar{x}}{u})({r}, \vartheta) > 0 \; \text{ for } \vartheta \in [\pi/2 - \alpha, \gamma],
\end{equation*}
so combining with \eqref{LY509a} yields
\begin{equation} \label{LY509c}
{C}_{0}\sin({l}\vartheta) > 0 \; \text{ for } \vartheta \in (\pi/2 - \alpha, \gamma).
\end{equation}
Therefore, the interval \((\pi/2 - \alpha, \gamma)\) must be contained in a nodal interval of the function \(\sin({l}\vartheta)\), whose length is at most \(\pi/{l}\). Combining this with the fact that
\begin{equation} \label{LY506c}
\gamma > \pi/2 \quad \text{and} \quad \alpha > \beta,
\end{equation}
we deduce that \(\beta < (\alpha + \beta)/2 = (\pi - \gamma)/2 < \pi/4\), and
\begin{equation*}
{l} \leq \frac{\pi}{\gamma - (\pi/2 - \alpha)} = \frac{\pi}{\pi/2 - \beta} < 4.
\end{equation*}
Since \(\pi/2 \in (\pi/2 - \alpha, \gamma)\), it follows from \eqref{LY509c} that \({C}_{0}\sin({l}\pi/2) > 0\), so \({l}\) is odd.

We claim that \({l} \neq 3\). Indeed, if \({l} = 3\), then \eqref{LY509c} becomes \({C}_{0}\sin (3\vartheta) > 0\) for \(\vartheta \in (\pi/2 - \alpha, \gamma)\).
Combining this with the fact that the interval \((3(\pi/2 - \alpha), 3\gamma)\) contains \(3\pi/2\), we conclude that \({C}_{0} < 0\) and
\begin{equation*}
3(\pi/2 - \alpha) \geq \pi, \quad 3\gamma \leq 2\pi,
\end{equation*}
that is, \(\alpha \leq \pi/6\) and \(\gamma \leq 2\pi/3\). From this and the condition \(\alpha > \beta\) in \eqref{LY506c}, we obtain \(\alpha > (\alpha + \beta)/2 = (\pi - \gamma)/2 > \pi/6\), which yields a contradiction.

Since \({l} < 4\), \({l}\) is odd, and \({l}\neq3\), we conclude \({l} = 1\) and \({C}_{0} > 0\). Therefore, \({R}_{\bar{x}}{u} > 0\) near \(\bar{x}\) in \(\Omega\). Applying Hopf lemma to \({R}_{\bar{x}}{u}\), it follows that the outward normal derivative of \({R}_{\bar{x}}{u}\) at \(\bar{x}\) is negative, i.e., \(\partial_{\nu}{R}_{\bar{x}}{u}(\bar{x}) < 0\). Therefore, \(\eqref{LY508}\) holds.

\textbf{Part 3}.
We show that \eqref{LY303a} holds for \(\vartheta \in [\pi/2 - \alpha, \gamma]\) and \(\bar{\Phi} - \lambda \ll1\).
Since \({u}\) is strictly monotone near the lower Neumann boundary (by \(\eqref{LY508}\)), near the Dirichlet boundary (see \autoref{lma203GNN}), and in the interior of \(\Omega\) (by Part 1), it follows by continuity that \(\eqref{LY303a}\) holds for
\begin{equation*}
\vartheta \in [\pi - 2\alpha, \gamma] \quad \text{and} \quad \lambda \geq \bar{\Phi} - \varepsilon
\end{equation*}
for some small constant \(\varepsilon > 0\). Using the same argument as in Step 2 of \autoref{lma303}, we deduce that \eqref{LY303a} holds for \(\vartheta \in [\pi/2 - \alpha, \pi - 2\alpha]\) and \(\lambda \geq \bar{\Phi} - \varepsilon\). This contradicts the definition of \(\bar{\Phi}\), and we conclude \(\bar{\Phi} = 0\). This completes the proof.
\end{proof}

We now show that the global maximum of \({u}\) may not be attained at the vertex of the domain \(\Omega\).

\begin{lma} \label{lma503}
Let \(\gamma > \pi/2\), and suppose that \(\alpha > \beta\) and
\begin{equation} \label{LY511}
\tilde{v}: = \nabla {u} \cdot \mathbf{e}_{ - \beta} \leq 0 \text{ in } \Omega.
\end{equation}
Then the global maximum of \({u}\) is not attained at the vertices of the triangle \(\Omega\).
\end{lma}

\begin{proof}
Since the nonpositive function \(\tilde{v}\) is negative near \(z_{1}\) (e.g., \autoref{lma303}) and satisfies \([\Delta + {f}'({u})] \tilde{v} = 0\) in \(\Omega\), the strong maximum principle implies that \(\tilde{v} < 0\) in \(\Omega\).
The proof is divided into two parts.

\textbf{Part 1}.
We first show that assumption \eqref{LY511} implies that \eqref{LY303a} holds for every \(\lambda \geq 0\) and for all \(\vartheta \in [\pi/2 - \alpha, \gamma]\). Moreover, it follows that
\begin{equation} \label{LY512}
\nabla {u} \cdot \mathbf{e}_{\alpha - \pi/2} < 0 \text{ on } \Gamma_{N}^{ - }.
\end{equation}
Indeed, since \(\alpha > \beta\), it follows that \(\gamma/2 > \pi/2 - \alpha\). For \(\vartheta = \gamma/2\), take \(\vartheta_{1} = 0\), \(\vartheta_{3} = \gamma\). Then, by \eqref{LY511} and for any \(\lambda \geq 0\), the function \({w}^{\lambda, \vartheta}\) satisfies the boundary condition \eqref{LY316}. Applying \autoref{lma301}, we deduce that both \eqref{LY305b} and \eqref{LY303a} hold. Following Step 1.1 of \autoref{lma303}, we conclude that \eqref{LY303a} holds for all \(\lambda \geq 0\) and \(\vartheta \in [\gamma/2, \gamma]\). A similar argument from Step 1.2 of \autoref{lma303} shows that \eqref{LY303a} remains valid for
\begin{equation*}
\vartheta \in \cup_{k = 1}^{ \infty}[\min\{2^{ - k}\gamma, \; \pi/2 - \alpha\}, \; \gamma].
\end{equation*}
In particular, \eqref{LY303a} holds for \(\lambda \geq 0\) and for \(\vartheta = \pi/2\), this is,
\begin{equation*}
\nabla {u}({x}) \cdot \mathbf{e}_{\alpha - \pi/2} < 0 \; \text{ for } {x} \in \Omega \text{ with } {x} \cdot \mathbf{e}_{\alpha - \pi/2} \geq 0.
\end{equation*}
Consequently, the even extension \(\tilde{u}\) of \({u}\) with respect to the Neumann boundary \(\Gamma_{N}^{ - }\), defined in \eqref{LY504}, satisfies
\begin{equation*}
\nabla \tilde{u}({x}) \cdot \mathbf{e}_{\alpha - \pi/2} \leq, \not\equiv 0 \; \text{ for } {x} \in \tilde{\Omega} \text{ with } {x} \cdot \mathbf{e}_{\alpha - \pi/2} \geq 0,
\end{equation*}
where \(\tilde{\Omega}\) denotes the reflected (or doubled) domain obtained by reflecting \(\Omega\) across the lower Neumann boundary \(\Gamma_{N}^{ - }\); see \autoref{fig37gamma}. Since the function \(\nabla \tilde{u} \cdot \mathbf{e}_{\alpha - \pi/2}\) satisfies the linear equation
\begin{equation*}
[\Delta + {f}'(\tilde{u})] (\nabla \tilde{u} \cdot \mathbf{e}_{\alpha - \pi/2}) = 0 \text{ in } \tilde{\Omega},
\end{equation*}
the strong maximum principle guarantees that \(\nabla \tilde{u} \cdot \mathbf{e}_{\alpha - \pi/2}\) must be negative in \(\tilde{\Omega} \cap \{{x} \cdot \mathbf{e}_{\alpha - \pi/2} \geq 0\}\). This proves \eqref{LY512}.

\textbf{Part 2}.
We now show that the origin \({z}_{0}\) is not an extremum point of \({u}\).

Let \(({r}, \theta)\) denote the standard polar coordinates centered at the origin. By classical regularity theory for solutions in domains with conical points (e.g., Theorem  6.4.2.5 in \cite{Gri11}), the expansion of \({u}\) near the origin is
\begin{equation} \label{LY515u}
{u}({r}, \theta) = {c}_{0} - {c}_{1} {r}^{\omega}\cos\big(\omega(\theta - \alpha + \pi/2)\big) - \tfrac{1}{4}{c}_{2} {r}^{2} + o({r}^{2})
\end{equation}
where \(\omega = \pi/\gamma \in (1, 2)\), \({c}_{0} = {u}({z}_{0}) > 0\), and \({c}_{2} = {f}({c}_{0})\). Define the angular derivative
\begin{equation}
{v}( {x} ): = x_{1}\partial_{x_{2}}{u}( {x} ) - x_{2}\partial_{x_{1}}{u}( {x} ).
\end{equation}
From Part 1, we know that \eqref{LY303a} holds for \(\lambda = 0\) and \(\vartheta \in [\pi/2 - \alpha, \gamma]\). That is
\begin{equation} \label{LY514}
{v}( {x} ) > 0 \; \text{ for } {x} \in \Omega \text{ with } 0 \leq x_{2} < x_{1}\cot\beta.
\end{equation}
Note that \({v}\) satisfies linear equation
\begin{equation*}
[\Delta + {f}'({u})] {v} = 0 \text{ in } \Omega \quad \text{and} \quad {v} = 0 \text{ on } \Gamma_{N},
\end{equation*}
and \({v}\not\equiv 0\) in \(\Omega\) (in any neighborhood of \({z}_{0}\)).
By \cite{HW53, HHT09}, there exist an integer \({l} \geq 1\) and a constant \(\tilde{c}_{l} \neq 0\) such that
\begin{equation*}
{v}({r}, \theta) = {l}\omega \tilde{c}_{l} {r}^{l\omega}\sin\big({l}\omega(\theta - \alpha + \pi/2)\big) + o({r}^{{l}\omega}).
\end{equation*}
Moreover, the nodal set \(\mathcal{Z}({v}) = \overline{\{{x} \in \Omega: \; {v}({x}) = 0\}}\) has exactly \({l} - 1\) branches near the origin, each tangent to the line
\begin{equation*}
x_{1} \sin(\frac{j\gamma}{l} + \alpha) - x_{2}\cos(\frac{j\gamma}{l} + \alpha) = 0
\end{equation*}
at the origin \({z}_{0}\) for \(j = 1, \ldots, {l} - 1\). Combining this with \eqref{LY514}, we obtain \(\gamma/{l} \geq \pi/2 - \beta\), hence \({l} \leq \gamma/(\pi/2 - \beta)\). Since \(\alpha > \beta\) implies \(\gamma = \pi - \alpha - \beta < \pi - 2\beta\), we conclude \({l} < 2\), thus \({l} = 1\). Hence,
\begin{equation*}
{v}({r}, \theta) = \tilde{c}_{1}\omega {r}^{\omega}\sin\big(\omega(\theta - \alpha + \pi/2)\big) + o({r}^{2}).
\end{equation*}
From \eqref{LY514}, we deduce that \(\tilde{c}_{1} > 0\) and \({v} > 0\) in a neighborhood of the origin \({z}_{0}\). Therefore, \({c}_{1} > 0\), which implies that the origin is not an extremum point of \({u}\).
\end{proof}

As a direct consequence of \autoref{lma501}, \autoref{lma502} and \autoref{lma503}, we obtain \autoref{thm104}.


\section{The eigenfunction with obtuse Neumann vertex} \label{Sec06EF}

In this section, we analyze the first mixed eigenfunction on a triangle whose Neumann vertex is \emph{obtuse}.
For a Lipschitz domain \({U}\) with a portion \(\Gamma_{1}\) of \(\partial{U}\), we denote by \(\lambda_{1}^{\mathrm{mix}}({U}, \Gamma_{1})\) the first (smallest) eigenvalue of the mixed Dirichlet-Neumann boundary value problem
\begin{equation}
\begin{cases}
\Delta \varphi + \lambda \varphi = 0 & \text{ in } {U}, \\
\partial_{\nu} \varphi = 0 & \text{ on } \Gamma_{1}, \\
\varphi = 0 & \text{ on } \Gamma_{0}: = \partial{U} \setminus \Gamma_{1}.
\end{cases}
\end{equation}
We write \(\lambda_{1}({U})\) for the first \emph{Dirichlet} eigenvalue of \(- \Delta\) in \({U}\).
For a continuous function \({v}\) on \(\overline{{U}}\), we denote its nodal line (zero level set) \(\mathcal{Z}({v})\) as the closure of \(\{ {x} \in {U}: {v}({x}) = 0 \}\), and refer to the connected components of \({U} \setminus \mathcal{Z}({v})\) as the \emph{nodal domains} of \({v}\).

\begin{thm} \label{thm601}
Let the triangle \(\Omega\), together with its Dirichlet boundary \(\Gamma_{D}\) and Neumann boundary \(\Gamma_{N}\), be as defined at the beginning of \autoref{Sec03acu}. Let \({u} > 0\) be the eigenfunction associated with \(\mu = \lambda_{1}^{\mathrm{mix}}(\Omega, \Gamma_{N})\). Assume that the Neumann vertex is obtuse (i.e., \(\gamma > \pi/2\)).
Then the following properties hold:
\begin{enumerate}
\item[\rm(1)]
\({u}\) is strictly decreasing in the inward normal direction to the longer Neumann side.
\item[\rm(2)]
\({u}\) has at most one non-vertex critical point, which lies on the longer Neumann side.
\item[\rm(3)]
Such a non-vertex critical point exists if and only if the triangle \(\Omega\) is not isosceles.
\end{enumerate}
\end{thm}

\begin{proof}
Define the set of non-vertex critical points by
\begin{equation}
\operatorname{crit}_{\mathrm{nv}}({u}): = \{{x} \in \overline{\Omega}: \; |\nabla {u}({x})| = 0 \text{ and } {x} \text{ is not a vertex}\}.
\end{equation}

\textbf{Step 1.}
We show that \(\operatorname{crit}_{\mathrm{nv}}({u}) \cap \Gamma_{N}\) contains at most one point, and at that point the second tangential derivative of \({u}\) is negative.

Let \({v}\) be the directional derivative of \({u}\) in the direction parallel to \(\Gamma_{N}^{ + }\), namely,
\begin{equation}
{v}: = \nabla {u} \cdot \mathbf{e}_{\pi/2 - \beta}.
\end{equation}
Applying the Hopf lemma to \({u}\) on \(\Gamma_{D}\) and using \autoref{lma302}, we obtain
\begin{equation} \label{LY604}
\begin{gathered}
{v} < 0 \text{ on } \operatorname{Int}(\Gamma_{D}), \\
{v} < 0 \text{ in } \mathcal{O}_{{z}_{2}} \cap (\Omega \cup \Gamma_{N}^{ + }) \quad \text{and} \quad
{v} > 0 \text{ on } \mathcal{O}_{{z}_{1}} \cap \Gamma_{N}^{ - },
\end{gathered}
\end{equation}
where \(\mathcal{O}_{{z}_{i}}\) denotes a small neighborhood of \({z}_{i}\) for \(i = 1, 2\). Note that
\begin{equation*}
\Delta {v} + \mu {v} = 0 \text{ in } \Omega \quad \text{and} \quad \partial_{\nu} {v} = 0 \text{ on } \Gamma_{N}^{ + }.
\end{equation*}
By classical nodal set theory (see, e.g., \cite[Theorem 4.1]{HHHO99}), for any \({p} \in \operatorname{crit}_{\mathrm{nv}}({u}) \cap \Gamma_{N}^{ + }\), there exists an integer \({n} \in \mathbb{N}^{ + }\) such that exactly \({n}\) smooth arcs of \(\mathcal{Z}({v})\) emanate from \({p}\), intersecting \(\Gamma_{N}^{ + }\) at angles \((2j - 1)/(2{n})\), \(j = 1, 2, \ldots, {n}\). Thus, \(\operatorname{crit}_{\mathrm{nv}}({u}) \cap \Gamma_{N}^{ + }\) consists of finitely many isolated points. An analogous argument for \(\Gamma_{N}^{ - }\) shows that \({u}\) has only finitely many critical points on \(\Gamma_{N}^{ - }\), and thus on \(\partial\Omega\).
Let \(\overline{\Gamma}_{N}^{\pm}\) denote the closures of the two Neumann sides \(\Gamma_{N}^{\pm}\).

\textbf{Step 1.1.}
\(\overline{\Gamma}_{N}^{ + } \cup \mathcal{Z}({v})\) cannot contain a loop, and \(\mathfrak{D}_{ - }: = \{{x} \in \Omega: {v}({x}) < 0\}\) is connected.

If there were a loop in \(\overline{\Gamma}_{N}^{ + } \cup \mathcal{Z}({v})\), it would enclose a nodal domain \(\mathfrak{D}_{1}\) with \(\partial\mathfrak{D}_{1} \subset \overline{\Gamma}_{N}^{ + } \cup \mathcal{Z}({v})\). Hence, \(\lambda_{1}^{\mathrm{mix}}\bigl(\mathfrak{D}_{1}, \partial\mathfrak{D}_{1} \cap \Gamma_{N}^{ + }\bigr) = \mu\). By the variational characterization of the first eigenvalue,
\begin{equation*}
\lambda_{1}^{\mathrm{mix}}\bigl(\mathfrak{D}_{1}, \partial\mathfrak{D}_{1} \cap \Gamma_{N}^{ + }\bigr)
> \lambda_{1}^{\mathrm{mix}}(\Omega, \Gamma_{N}) = \mu,
\end{equation*}
which is a contradiction. Thus, no such loop can occur.

Suppose, for contradiction, that \(\mathfrak{D}_{ - }\) is disconnected. By \eqref{LY604}, one component of \(\mathfrak{D}_{ - }\) meets \(\operatorname{Int}(\Gamma_{D})\), so there is another component \(\mathfrak{D}_{*}\) whose boundary does not meet \(\operatorname{Int}(\Gamma_{D})\). Since \(\overline{\Gamma}_{N}^{ + }\cup\mathcal{Z}({v})\) has no loop, \(\partial\mathfrak{D}_{*} \cap \Gamma_{D} = \emptyset\) and
\(\Gamma_{*}: = \partial\mathfrak{D}_{*} \cap \overline{\Gamma}_{N}^{ - } \neq \emptyset\). Define the auxiliary function
\begin{equation*}
\phi({x}) =
\begin{cases}
{v}({x}) & \text{ if } {x} \in \mathfrak{D}_{*},
\\
0 & \text{ if } {x} \in \Omega \setminus \mathfrak{D}_{*}.
\end{cases}
\end{equation*}
Then \(\phi\) satisfies
\begin{equation} \label{LY605}
\phi \in W^{1, 2}(\Omega) \quad \text{with} \quad \phi = 0 \text{ on } \Gamma_{D},
\end{equation}
and \(\partial_{\nu}\phi = 0\) on \(\Gamma_{N}^{ + }\). Integration by parts gives
\begin{equation*}
\int_{\Omega} |\nabla \phi|^{2} dx
= \mu \int_{\Omega} |\phi|^{2} dx + \int_{\partial\Omega} \phi \partial_{\nu}\phi ds_x
= \mu \int_{\Omega} |\phi|^{2} dx + \int_{\Gamma_{*}} {v}\partial_{\nu}{v} ds_x.
\end{equation*}
To handle \(\int_{\Gamma_{*}} {v} \partial_{\nu}{v} ds_{x}\), we use a technique inspired by Terence Tao \cite{Pol12}. Denote the tangential and normal vectors of \(\Gamma_{N}^{ - }\) by
\begin{equation*}
\mathit{e} = (\sin\alpha, - \cos\alpha), \quad \mathit{e}^{\perp} = (\cos\alpha, \sin\alpha),
\end{equation*}
and write \({v} = {c}_{1}\partial_{\mathit{e}}{u} + {c}_{2}\partial_{\mathit{e}^{\perp}}{u}\), with \({c}_{1} = \cos \gamma < 0\) and \({c}_{2} = \sin \gamma > 0\). A direct calculation yields
\begin{equation*}
{v}\partial_{\nu}{v}
= ({c}_{1}\partial_{\mathit{e}} {u}) \cdot ( - {c}_{2} (\partial_{\mathit{e}^{\perp}})^{2} {u})
= {c}_{1}{c}_{2} \partial_{\mathit{e}} {u} \cdot ((\partial_{\mathit{e}})^{2} {u} + \mu {u}).
\end{equation*}
Since \({v} < 0\) in \(\mathfrak{D}_{*}\), one infers \(\partial_{\mathit{e}} {u} \geq 0\) on \(\Gamma_{*}\), hence
\begin{equation*}
\int_{\Gamma_{*}} {u} \partial_{\mathit{e}} {u} ds_{x} \geq 0.
\end{equation*}
Moreover, \(\mathcal{Z}({v}) \cap \overline{\Gamma}_{N}^{ - }\) is a finite set of critical points of \({u}\), so \(\Gamma_{*}\) is a union of finitely many isolated points and finitely many line segments \({q}_{2i}{q}_{2i + 1}\), \(i = 0, \ldots, {m}\), where \({q}_{i}\) are distinct critical points of \({u}\) and \({q}_{i}\) lies to the left of \({q}_{i + 1}\). Therefore
\begin{align*}
\int_{\Gamma_{*}} \partial_{\mathit{e}} {u} (\partial_{\mathit{e}})^{2}{u} ds_x
=
\frac{1}{2} \int_{\Gamma_{*}} \partial_{\mathit{e}}\big(\partial_{\mathit{e}}{u}\big)^{2} ds_x
=
\frac{1}{2}\sum_{i = 0}^{m}(|\partial_{\mathit{e}} {u}({q}_{2i + 1})|^{2} - |\partial_{\mathit{e}} {u}({q}_{2i})|^{2} ) = 0.
\end{align*}
Hence,
\begin{equation*}
\int_{\Gamma_{*}} {v}\partial_{\nu}{v} ds_{x}
= {c}_{1}{c}_{2} \int_{\Gamma_{*}} \partial_{\mathit{e}}{u} (\partial_{\mathit{e}})^{2}{u} ds_{x}
+ {c}_{1}{c}_{2} \mu \int_{\Gamma_{*}} {u} \partial_{\mathit{e}}{u} ds_{x} \leq 0,
\end{equation*}
and so \(\phi\) satisfies \eqref{LY605} and
\begin{equation}
\int_{\Omega} |\nabla \phi|^{2} d{x} \leq \mu \int_{\Omega} |\phi|^{2} d{x}.
\end{equation}
By the variational characterization of the first mixed eigenvalue \(\mu = \lambda_{1}^{\mathrm{mix}}(\Omega, \Gamma_{N})\), \(\phi\) must be a multiple of \({u}\), contradicting the definition of \(\phi\). Therefore, \(\mathfrak{D}_{ - }\) is connected.

\textbf{Step 1.2.}
\(\operatorname{crit}_{\mathrm{nv}}({u}) \cap \Gamma_{N}^{ + }\) contains at most one point, and at that point the second-order tangential derivative of \({u}\) is strictly negative.
It is well-known (cf. Section 2 of \cite{JM20}) that the nodal line \(\mathcal{Z}({v})\) decomposes into immersed \({C}^{1}\) loops and properly immersed \({C}^{1}\) arcs. By Step 1.1, no loop can occur, so \(\mathcal{Z}({v})\) is a finite union of properly immersed \({C}^{1}\) arcs. Each such arc has exactly two distinct endpoints, one on \(\overline{\Gamma}_{N}^{ - }\) and one on \(\Gamma_{N}\). Since \(\mathfrak{D}_{ - }\) is connected, there is at most one arc ending on \(\Gamma_{N}^{ + }\). Standard nodal line theory (see \cite{HHHO99}) then shows that \({v}\) has at most one zero on \(\Gamma_{N}^{ + }\), and at that zero the second tangential derivative of \({u}\) is negative.

\textbf{Step 1.3.}
\({u}\) has at most one critical point on \(\Gamma_{N}^{ + } \cup \Gamma_{N}^{ - }\).
Suppose for contradiction that there are two such points. A similar argument as in Steps 1.1 and 1.2 implies that the restriction of \({u}\) to \(\Gamma_{N}^{ - }\) has at most one critical point, which is non-degenerate. Hence, \({u}\) has two critical points, \({p} \in \Gamma_{N}^{ + }\) and \({q} \in \Gamma_{N}^{ - }\). From \eqref{LY604},
\begin{equation}
{v} < 0 \text{ on } \operatorname{Int}(\Gamma_{D}) \cup \operatorname{Int}({z}_{2}{p}) \cup \operatorname{Int}({z}_{0}{q}) \quad \text{and} \quad {v} > 0 \text{ on } \operatorname{Int}({p}{z}_{0}) \cup \operatorname{Int}({q}{z}_{1}).
\end{equation}
Let \(\mathfrak{D}_{2}\) be the nodal domain of \({v}\) whose closure contains the line segment \({p}{z}_{0}\). Since \(\mathfrak{D}_{ - } = \{{x} \in \Omega: {v}({x}) < 0\}\) is connected (by Step 1.1), its boundary must contain both the line segments \({z}_{2}{p}\) and \({z}_{0}{q}\). Hence, \(\mathfrak{D}_{2}\) is bounded by the line segment \({p}{z}_{0}\) together with a portion of \(\mathcal{Z}({v})\), contradicting Step 1.1. This contradiction completes the proof of Step 1.

\textbf{Step 2.}
\(\operatorname{crit}_{\mathrm{nv}}({u})\) has at most one point, which lies on the longer Neumann side.

\textbf{Step 2.1.}
\({u}\) is strictly monotone in some direction.
By Step 1, \({u}\) has at most one non-vertex critical point on \(\partial\Omega\). Without loss of generality, assume there is no critical point on \(\Gamma_{N}^{ - }\). Consider the directional derivative \(\tilde{v} = \nabla {u} \cdot \mathbf{e}_{ - \beta}\) in the direction perpendicular to \(\Gamma_{N}^{ + }\). Then \(\tilde{v} \leq, \not\equiv 0\) on \(\partial\Omega\).
If \(\tilde{v}\) were positive somewhere in \(\Omega\), its positive nodal domain \(\mathfrak{D}_{3}\) would satisfy
\begin{equation*}
\Delta \tilde{v} + \mu \tilde{v} = 0 \text{ in } \mathfrak{D}_{3}, \quad \tilde{v} = 0 \text{ on } \partial\mathfrak{D}_{3},
\end{equation*}
forcing \(\mu = \lambda_{1}(\mathfrak{D}_{3}) > \lambda_{1}(\Omega) > \lambda_{1}^{\mathrm{mix}}(\Omega, \Gamma_{N}) = \mu\), a contradiction. Hence, \(\tilde{v} \leq 0\) in \(\Omega\), and by the strong maximum principle, \(\tilde{v} < 0\) in \(\Omega\). Furthermore, applying the Hopf lemma to \(\tilde{v}\) on \(\Gamma_{N}^{ + }\), we conclude that the inner normal derivative of \({u}\) on \(\Gamma_{N}^{ + }\) is negative. Consequently, the unique non-vertex critical point of \({u}\), if it exists, is a non-degenerate maximum.

\textbf{Step 2.2.}
The global maximum of \({u}\) cannot be attained in the interior of the strictly shorter Neumann side.
Without loss of generality, suppose \(\alpha > \beta\). Then the function \({w}^{0, \gamma/2}\), defined in \eqref{LY305a}, satisfies
\begin{equation*}
\begin{cases}
\Delta {w}^{0, \gamma/2} + \mu {w}^{0, \gamma/2} = 0 & \text{in } {D}_{0, \gamma/2}, \\
\partial_{\nu}{w}^{0, \gamma/2} = 0 & \text{on } \Gamma_{0, \gamma/2}^{2} \subset \Gamma_{N}^{ - }, \\
{w}^{0, \gamma/2} \geq, \not\equiv 0 & \text{on } \partial{D}_{0, \gamma/2} \setminus \Gamma_{0, \gamma/2}^{2}.
\end{cases}
\end{equation*}
If \({w}^{0, \gamma/2}\) were negative somewhere in \(\mathfrak{D}_{4}\), its negative nodal component \(\mathfrak{D}_{4}\) would satisfy
\begin{equation*}
\Delta {w}^{0, \gamma/2} + \mu {w}^{0, \gamma/2} = 0 \text{ in } \mathfrak{D}_{4}, \quad
{w}^{0, \gamma/2} = 0 \text{ on } \partial\mathfrak{D}_{4} \setminus \Gamma_{N}^{ - }, \quad
\partial_{\nu}{w}^{0, \gamma/2} = 0 \text{ on } \partial\mathfrak{D}_{4} \cap \Gamma_{N}^{ - }.
\end{equation*}
This gives \(\mu = \lambda_{1}^{\mathrm{mix}}(\mathfrak{D}_{4}, \partial\mathfrak{D}_{4} \cap \Gamma_{N}^{ - })\). On the other hand,
\begin{equation*}
\lambda_{1}^{\mathrm{mix}}(\mathfrak{D}_{4}, \partial\mathfrak{D}_{4} \cap \Gamma_{N}^{ - }) > \lambda_{1}^{\mathrm{mix}}(\Omega, \Gamma_{N}^{ - }) > \lambda_{1}^{\mathrm{mix}}(\Omega, \Gamma_{N}) = \mu,
\end{equation*}
a contradiction. Thus, \({w}^{0, \gamma/2} \geq 0\) in \({D}_{0, \gamma/2}\). By the strong maximum principle and the Hopf lemma, we obtain
\begin{equation} \label{LY608}
{w}^{0, \gamma/2} > 0 \text{ in } \overline{{D}_{0, \gamma/2}} \setminus {T}_{0, \gamma/2} \quad \text{ and } \quad\nabla {u} \cdot \mathbf{e}_{{\vartheta} + \alpha} > 0 \text{ on } \Omega \cap {T}_{0, {\vartheta}} \; \text{for } {\vartheta} = \gamma/2.
\end{equation}
Hence, the global maximum cannot occur in the interior of the strictly shorter Neumann side.

Combining all the steps above, we conclude that \({u}\) has at most one non-vertex critical point, which is a non-degenerate global maximum located on the longer Neumann side. Moreover, \({u}\) is monotone in the direction perpendicular to the longer Neumann side. For \(\alpha \neq \beta\), \autoref{lma503} implies that the global maximum is not located at the Neumann vertex, while for \(\alpha = \beta\), \autoref{thm401} implies that the global maximum is uniquely located at the Neumann vertex. This completes the proof.
\end{proof}

\begin{rmks}
The same arguments in the proof of \autoref{thm601} can be carried over to the case where the Neumann vertex is non-obtuse. The monotonicity of the first mixed eigenfunction in triangles can also be established via a continuity method under domain deformation \cite{JN00, JM20}.

Alternatively, the proof of \autoref{thm601} can be obtained by applying maximum principles to certain functions. We briefly list the steps in the case \(\alpha > \beta\):
\begin{enumerate}[label = \rm(\arabic*)]
\item
\({w}^{0, \gamma/2} > 0\) in \({D}_{0, \gamma/2}\), and \eqref{LY303a} holds for \(\lambda = 0\) and \(\vartheta = \gamma/2\); see also \eqref{LY608}.
\item
\({R}_{{z}_{0}} {u} > 0\) in the triangle enclosed by \({T}_{0, {\gamma}/2}\), \(\Gamma_{D}\), and \(\Gamma_{N}^{ + }\).
\item
\(\nabla {u} \cdot \mathbf{e}_{\alpha - \pi/2} < 0\) in the triangle enclosed by \({T}_{0, \pi/2}\), \(\Gamma_{D}\), and \(\Gamma_{N}^{ - }\), and \(\operatorname{crit}_{\mathrm{nv}}({u}) \cap \Gamma_{N}^{ - } = \emptyset\).
\item
\(\nabla {u} \cdot \mathbf{e}_{ - \beta} < 0\) in \(\Omega\).
\end{enumerate}
The remaining conclusions then follow from nodal line analysis.
\end{rmks}


\section*{Acknowledgments}
The authors sincerely thank the anonymous referees for their comments which help improve the presentation of the paper. 
The research of Rui Li was supported in part by Natural Science Foundation of Top Talent of SZTU (No. GDRC202213) and the National Natural Science Foundation of China (No. 12101416). The research of Ruofei Yao was supported in part by the Guangdong Basic and Applied Basic Research Foundation (No. 2025A1515011856).

\end{document}